\documentclass[11pt]{amsart}

\usepackage{comment}
\usepackage{amsfonts}
\usepackage{amssymb}
\usepackage{amsthm}
\usepackage{mathtools}   %
\usepackage[alphabetic]{amsrefs}
\usepackage{csquotes}
\usepackage{footmisc}
\usepackage{xpatch}
\usepackage{longtable}
\usepackage[all,cmtip]{xy}
\usepackage{fancyhdr}
\usepackage[normalem]{ulem}

\usepackage{graphicx}
\usepackage{hyperref}

\usepackage{tikz-cd}
\usepackage{tikz}
\usepackage{pifont}
\usepackage{aurical}
\usepackage{stmaryrd}
\usepackage[T1]{fontenc}
\usepackage[outline]{contour}%
\usepackage{xcolor}%
\usepackage{etoolbox}
\usepackage{array}
\usepackage{lipsum}

\usepackage{enumitem}

\usepackage{parskip}

\setlist[enumerate,1]{label={(\alph*)}}

\makeatletter
\DeclareFontEncoding{LS2}{}{\@noaccents}
\makeatother
\DeclareFontSubstitution{LS2}{stix}{m}{n}
\DeclareSymbolFont{largesymbolsstix}{LS2}{stixex}{m}{n}
\DeclareMathDelimiter{\lbrbrak}{\mathopen}{largesymbolsstix}{"EE}{largesymbolsstix}{"14}
\DeclareMathDelimiter{\rbrbrak}{\mathclose}{largesymbolsstix}{"EF}{largesymbolsstix}{"15}

\newcommand{\genlegendre}[4]{%
  \genfrac{(}{)}{}{#1}{#3}{#4}%
  \if\relax\detokenize{#2}\relax\else_{\!#2}\fi
}

\newcommand{\WIP}[1]{\scalebox{#1}{\begin{tikzpicture}[limb/.style={line cap=round,line width=1.5mm,line join=bevel}]
\draw[line width=2mm,rounded corners,fill=yellow] (-2,0) -- (0,-2) -- (2,0) -- (0,2) -- cycle;
\fill (1.5mm,7mm) circle (1.5mm);
\fill(0,-7.5mm) -- ++(10mm,0mm) -- ++(120:2mm)--++(100:1mm)--++(150:2mm) arc (70:170:2.5mm and 1mm);
\draw[limb] (-7.5mm,-6.5mm)--++(70:4mm)--++(85:4mm) coordinate(a)--++(-45:5mm)--(-2.5mm,-6.5mm);
\fill[rotate around={45:(a)}] ([shift={(-0.5mm,0.55mm)}]a) --++(0mm,-3mm)--++
        (7mm,-0.5mm)coordinate(b)--++(0mm,4mm)coordinate(c)--cycle;
\draw[limb] ([shift={(-0.6mm,-0.4mm)}]b) --++(-120:5mm) ([shift={(-0.5mm,-0.5mm)}]c) --++
        (-3mm,0mm)--++(-100:3mm)coordinate (d);
\draw[ultra thick] (d) -- ++(-45:1.25cm);
\end{tikzpicture}}}

\robustify{\WIP}

\definecolor{cec1d24}{RGB}{236,29,36}
\definecolor{cffffff}{RGB}{255,255,255}

\renewcommand{\b}[1]{\textbf{\textit{#1}}}

\newcommand{\lra}{\longrightarrow}

\newcommand{\bA}{\mathbb{A}}

\newcommand{\bC}{\mathbb{C}}

\newcommand{\bF}{\mathbb{F}}

\newcommand{\bH}{\mathbb{H}}

\newcommand{\bP}{\mathbb{P}}
\newcommand{\bQ}{\mathbb{Q}}

\newcommand{\bX}{\mathbb{X}}

\newcommand{\bZ}{\mathbb{Z}}

\newcommand{\Qbar}{{\overline{\mathbb{Q}}}}

\newcommand{\cB}{\mathcal{B}}
\newcommand{\cC}{\mathcal{C}}
\newcommand{\cD}{\mathcal{D}}
\newcommand{\cE}{\mathcal{E}}

\newcommand{\cM}{\mathcal{M}}

\newcommand{\cO}{\mathcal{O}}

\newcommand{\cT}{\mathcal{T}}
\newcommand{\cU}{\mathcal{U}}
\newcommand{\cV}{\mathcal{V}}
\newcommand{\cW}{\mathcal{W}}
\newcommand{\cX}{\mathcal{X}}
\newcommand{\cY}{\mathcal{Y}}
\newcommand{\cZ}{\mathcal{Z}}

\newcommand{\fm}{\mathfrak{m}}

\newcommand{\wt}[1]{\widetilde{#1}}

\newcommand{\Aut}{\operatorname{Aut}}

\newcommand{\Out}{\operatorname{Out}}

\newcommand {\spmatrix}[4]{\left[\begin{smallmatrix}#1&#2\\#3&#4\end{smallmatrix}\right]}

\newcommand{\diag}{{\operatorname{diag}}}

\newcommand{\Jac}{\operatorname{Jac}}

\newcommand{\Gal}{{\operatorname{Gal}}}

\newcommand{\Epi}{\operatorname{Epi}}

\newcommand{\GL}{{\operatorname{GL}}}
\newcommand{\SL}{{\operatorname{SL}}}
\newcommand{\PSL}{\operatorname{PSL}}
\newcommand{\PGL}{\operatorname{PGL}}

\newcommand{\ab}{{\operatorname{ab}}}

\newcommand{\tr}{\operatorname{tr}}

\newcommand{\ord}{\operatorname{ord}}

\newcommand{\abs}{{\operatorname{abs}}}

\newcommand{\ol}[1]{{\overline{#1}}}

\newcommand{\Spec}{\operatorname{Spec}}

\newcommand{\ext}{{\mathrm{ext}}}

\newcommand{\Tate}{\operatorname{Tate}}

\usepackage{lmodern,mathrsfs,microtype}
\usepackage[margin=1.05in]{geometry}
\hypersetup{colorlinks=true,linkcolor=blue!45!black,citecolor=blue!45!black,urlcolor=blue!45!black,
  pdftitle={Bad reduction of three-point covers and a degree bound for non-Markoff components},
  pdfauthor={William Y. Chen}}
\newtheorem{theorem}{Theorem}[section]
\newtheorem{proposition}[theorem]{Proposition}
\newtheorem{conjecture}[theorem]{Conjecture}
\newtheorem{lemma}[theorem]{Lemma}
\newtheorem{corollary}[theorem]{Corollary}
\theoremstyle{definition}
\newtheorem{definition}[theorem]{Definition}
\theoremstyle{remark}
\newtheorem{remark}[theorem]{Remark}

\numberwithin{equation}{section}
\newcommand{\Gp}{G_p}

\newcommand{\Mabs}[1]{M(#1)^{\mathrm{abs}}}

\title[Stable reduction and large Nielsen orbits on generalized Markoff surfaces]{Stable reduction and large Nielsen orbits on generalized Markoff surfaces}
\date{September 16, 2026}
\author{William Y. Chen}
\begin{document}
\begin{abstract}
In this paper we establish lower bounds on the orbit sizes of the action of $\SL_2(\bZ)$ on the $\bF_p$-points of the generalized Markoff equations $x^2 + y^2 + z^2 - xyz = \tau+2$. For $\tau\ne 2$, these $\bF_p$-points correspond to absolutely irreducible representations of the rank 2 free group $\langle a,b\rangle$ in $\SL_2(\bF_p)$ up to $\GL_2(\bF_p)$-conjugation, such that the image of $[a,b] = aba^{-1}b^{-1}$ has trace $\tau$. We show that every orbit corresponding to \emph{surjective} representations has size at least $p$. Together with the recent work of Nathaniel Kingsbury-Neuschotz, this establishes the classification conjectures of McCullough--Wanderley on Nielsen equivalence classes of generating pairs of $\SL_2(\bF_p)$, and gives a complete classification of the components of the Hurwitz stack of 1-point $\SL_2(\bF_p)$-covers of elliptic curves, for all but finitely many primes $p$.

The case $\tau = -2$ was settled by the work of Bourgain, Gamburd, and Sarnak, combined with a result of the author that every orbit has size \emph{divisible} by $p$. While this divisibility by $p$ fails for general $\tau$, the present paper shows that the $\ge p$ lower bound does generalize.
Our method shows that the modular curve associated to any hypothetical orbit must have bad reduction, as a cover of the $j$-line, at some place above $p$. Good reduction results of Grothendieck for tamely ramified covers then force the modular curve to have degree $\ge p$ over the $j$-line. To establish this bad reduction, we note that the modular curve parametrizes $\PSL_2(\bF_p)$-covers of elliptic curves; we use a result of Bouw and Wewers to show that the parametrized covers have bad reduction at some place above $p$, and use the stable reduction theory of Henrio, Obus, Raynaud, and Wewers to show that this bad reduction implies bad reduction of the moduli space.

\end{abstract}
\maketitle
\markboth{\MakeUppercase{Stable reduction and large Nielsen orbits on generalized Markoff surfaces}}{\MakeUppercase{Stable reduction and large Nielsen orbits on generalized Markoff surfaces}}
\setcounter{tocdepth}{1}
\tableofcontents
\section{Introduction}
\label{sec:introduction}

\subsection{\texorpdfstring{$G$}{G}-structures}

We work over $\Qbar$ by default. Let $G$ be a finite group. A \emph{$G$-cover} of a curve
$C$ is a finite morphism $Y\to C$ together with an action of $G$ on
$Y$ over $C$ that makes $Y$ a $G$-torsor over a dense open subset of
$C$; the curve $Y$ may be disconnected, and a connected $G$-cover is a Galois cover with group $G$. Morphisms of $G$-covers are required to be $G$-equivariant.
If a cover is ramified above $\le n$ points in the base, we call it an $n$-point cover.

The isomorphism class of a connected $n$-point $G$-cover $Y\to C$ is called a $G$-structure on the $n$-pointed curve $C$. If we drop the requirement that isomorphisms of covers be $G$-equivariant, the resulting isomorphism class is an \emph{absolute $G$-structure} on $C$. If $C$ has genus $g$, the moduli space of $n$-pointed genus $g$ curves
 with $G$-structures (resp. absolute $G$-structures) is denoted $M_{g,n}(G)$ (resp. $M_{g,n}(G)^\abs$). The map $M_{g,n}(G)\to M_{g,n}(G)^\abs$ identifies the target with the quotient $M_{g,n}(G)/\Out(G)$. In this paper we will mainly consider $M_{g,n}(G)^\abs$. These spaces are typically disconnected, and it is a classical problem to describe their connected components in terms of discrete invariants of covers -- this has seen applications ranging from the irreducibility of the moduli spaces $M_g$ \cite{Fulton} to function field Cohen--Lenstra--Martinet heuristics for class groups \cite{EllenbergVenkateshWesterland}, \cite{LiuWoodZureickBrown}.

In the case $(g,n) = (1,1)$, the moduli spaces $M(G) := M_{1,1}(G)$ and $M(G)^\abs := M_{1,1}(G)^\abs$ were introduced in \cite{ChenModuli} to study noncongruence modular curves (see also \cite{ChenCRM}). Indeed, forgetting the cover and taking $j$-invariants, we obtain a finite ``forgetful'' map
\[
 M(G)^\abs\to M(1)\cong\bA^1_j,
\]
where $M(1)$ is the coarse moduli scheme of elliptic curves; it is an affine line with parameter $j$. Over $\bC$, $M(1)_\bC$ is the quotient of the Poincar\'{e} upper half plane $\bH$ by $\SL_2(\bZ)$, and the components of $M(G)^\abs$ are similarly quotients $\bH/\Gamma$ for some finite index subgroup $\Gamma\le\SL_2(\bZ)$, possibly noncongruence.  For more details on $G$-structures and their moduli spaces, see \cite{BBCL}*{\S2} and \cite{Chen}*{\S2}. 

While it is often straightforward to obtain estimates for the degree of the map $M(G)^\abs\to\bA^1_j$, to study the modular curve components, the first step is to count them, or better classify them in terms of discrete invariants of covers. For example, the picture would look very different if $M(G)$ consisted of a single component of degree 100 and genus 41, as compared to 100 components of genus 0, each mapping isomorphically to $\bA^1_j$. While the components are well understood in the regime where $G$ is fixed and $g$ or $n$ are allowed to be large\footnote{For large $g$ by Dunfield and Thurston \cite{DunfieldThurston}*{Theorem~6.20} in the unramified case and by Catanese, L\"onne and Perroni \cite{CataneseLonnePerroni}*{Theorem~1.5} in general; for large $n$ by the theorem of Conway and Parker \cite{FriedVolklein}*{Appendix} and by L\"onne \cite{Lonne}*{Theorem~1.1}.}, for fixed $(g,n)$, the problem seems to be much more mysterious, and in general there is not even a conjectural description of the components.

The one exception is for the components of $M(\SL_2(\bF_p))^\abs$, where the \emph{classification conjectures} of McCullough and Wanderley would imply that for $p\ge 13$, the components are completely classified by the trace $\tau$ of the local monodromy of the $\SL_2(\bF_p)$-cover around the origin. Moreover, they show that for $p\ge 13$, every trace $\tau\ne 2$ is possible, so $M(\SL_2(\bF_p))^\abs$ would have $p-1$ components, one for each $\tau\in\bF_p - \{2\}$.

In this paper we prove the classification conjectures for all but finitely many primes $p$ (see Corollaries \ref{cor_MW} and \ref{cor_connected} below). The argument uses the dictionary between components of $M(\SL_2(\bF_p))^\abs$ and the mapping class group orbits of $\bF_p$-points of the character variety of $\SL_2$-representations of the fundamental group of a punctured elliptic curve. There are two main ingredients that go into the argument. This paper supplies the following ingredient on the moduli space side:

\begin{theorem}[Component-degree bound]\label{thm:nm-degree}
For a prime $p$, let $X$ be a connected component
of $\Mabs{\Gp}_\Qbar$. Let $j_X$ denote the map given by taking $j$-invariants
$$j_X : X\to\bA^1_j.$$
Then $d_X := \deg j_X\ge p$.
\end{theorem}

This generalizes the result of \cite{Chen}, which proved that $p\mid d_X$ for any component $X$ lying in the \emph{Markoff locus} $\tau = -2$. The other ingredient lives on the character variety side, and is supplied by a recent result of Nathaniel Kingsbury-Neuschotz \cite{KingsburyNeuschotz}, generalizing the work of Bourgain, Gamburd, and Sarnak from \cite{BGS}. The character variety is fibered by a function -- the \emph{local monodromy trace} $\tau$; each fiber is cut out by a \emph{generalized Markoff equation}.

\subsection{Generalized Markoff equations}
\label{intro:markoff}
Let $p$ be odd and let $(a,b)$ be a generating pair of $\Gp := \PSL_2(\bF_p)$. Choose lifts $A,B\in\SL_2(\bF_p)$ of $a,b$ and let
$x=\tr A$, $y=\tr B$, $z=\tr AB$. The commutator trace
$\tau:=\tr(ABA^{-1}B^{-1})$ does not depend on the choices of
lifts, and it satisfies the \emph{generalized Markoff equation}
\begin{equation}\label{eq_generalized_markoff}
 x^2+y^2+z^2-xyz = \tau+2.
\end{equation}
Let $\bX_\tau$ denote the affine cubic surface given by \eqref{eq_generalized_markoff}.
When $\tau = -2$ (the \emph{Markoff case}), for $p\ne3$, \eqref{eq_generalized_markoff} is equivalent, up to the change of variables $(x,y,z)\mapsto (3x,3y,3z)$, to the classical \emph{Markoff equation} $x^2 + y^2 + z^2 - 3xyz = 0$, which first arose in Markoff's work on Diophantine approximation and binary quadratic forms \cite{Markoff1879, Markoff1880}.

Let $E$ be an elliptic curve, and let $\Pi = \pi_1(E^\circ)$, with free generators $\alpha,\beta$. For $\tau\ne 2$, the $\bF_p$-points of $\bX_\tau$ are in bijection with the set $F(p)_\tau^{\operatorname{irr}}$ of $\GL_2(\bF_p)$-conjugacy classes of absolutely irreducible representations $\varphi : \Pi\to\SL_2(\bF_p)$ such that $\tr[\varphi(\alpha),\varphi(\beta)] = \tau$ \cite{Chen}*{Theorem~5.2.10}. As $[\alpha,\beta]$ is a generator of inertia around the puncture of $E^\circ$, $\tau$ is the trace of the local monodromy around $O$ in the $\SL_2(\bF_p)$-cover whose monodromy representation is $\varphi$.\footnote{It is useful to recall that in $\SL_2(\bF_p)$ with $p \ge 3$, for every trace $\ne\pm 2$, there is at most one conjugacy class per trace. For traces $\pm 2$, there are three classes: $\pm I$, $\spmatrix{\pm 1}{1}{0}{\pm 1}$, and $\spmatrix{\pm 1}{a}{0}{\pm 1}$ where $a\in\bF_p^\times$ is a nonsquare.} Every automorphism of $\Pi$ preserves the conjugacy class of $[\alpha,\beta]$ up to inversion, so $\Out(\Pi)$ acts on $\bX_\tau(\bF_p)\cong F(p)_\tau^{\operatorname{irr}}$.
 The subgroup $\Out^+(\Pi)$ which acts with determinant 1 on the abelianization $\Pi/[\Pi,\Pi]$ is isomorphic to $\SL_2(\bZ)$ and is the mapping class group of $E^\circ$.

Write $F(p)_\tau$ for the set of $\GL_2(\bF_p)$-conjugacy classes of surjective representations $\varphi\colon\Pi\to\SL_2(\bF_p)$ with $\tr[\varphi(\alpha),\varphi(\beta)]=\tau$; for $\tau\ne2$ it is a subset of $F(p)_\tau^{\operatorname{irr}}$, and for $\tau = 2$, $F(p)_2=\emptyset$ \cite{Chen}*{Proposition~5.2.16(a)}.
Write
\[
 F(p) := \bigsqcup_{\tau\in\bF_p} F(p)_\tau .
\]

McCullough-Wanderley's classification conjectures concern the $\Out(\Pi)$-action on $F(p)$:
\begin{conjecture}[{Classification conjecture, see \cite{McCulloughWanderley2013}}\footnote{In \cite{McCulloughWanderley2013}, McCullough and Wanderley make multiple variations of the classification conjecture. Martin proved that they are all equivalent \cite{MartinNielsen}*{Theorem~1.5}, so we'll simply refer to it as the classification conjecture.}] For all primes $p\ge 13$ and $\tau\in\bF_p$, $F(p)_\tau$ is nonempty if and only if $\tau\ne 2$, and in this case there is exactly one $\Out(\Pi)$-orbit.
\end{conjecture}

The state of progress on the conjecture is as follows. The case $\tau=-2$ is known, for all but finitely many primes, from results of Bourgain, Gamburd and Sarnak and of the author \cite{BGS, Chen}. Daniel Martin proved the general $\tau$ case for primes $p$ satisfying $24504480\nmid p^2-1$, leaving open a positive density set of primes \cite{MartinNielsen}*{Theorem~1.1}. Finally, a recent result of Nathaniel Kingsbury-Neuschotz, generalizing the work of \cite{BGS}, addresses all primes $p\ge13$ and all $\tau\ne2$, but falls short of transitivity:

\begin{theorem}[Kingsbury-Neuschotz \cite{KingsburyNeuschotz}*{Propositions~6.1, 6.4 and~8.1, and \S7}]\label{thm:NKN} For each prime $p\ge13$ and each $\tau\in\bF_p\setminus\{2\}$, there is a large $\Out(\Pi)$-orbit $\cC(p)_\tau\subset F(p)_\tau$ such that for any $\epsilon > 0$,
\[
 |F(p)_\tau \setminus \cC(p)_\tau| = O(p^\epsilon),
\]
where the implied constant depends only on $\epsilon$, and not on $\tau$.
\end{theorem}

If we could obtain a polynomial-in-$p$ lower bound on the sizes of \emph{all} orbits, this would establish the classification conjecture for all but finitely many primes $p$. This is the role played by Theorem \ref{thm:nm-degree}.

To explain how Theorem~\ref{thm:nm-degree} provides such a lower bound, note that the set of absolute $G$-structures on an elliptic curve $E$ is in bijection with $\Epi(\Pi,G)/\Aut(G)$ \cite{Chen}*{Theorem~2.5.2(4) and Definition~2.5.13}. When $G = \SL_2(\bF_p)$, every automorphism of $\SL_2(\bF_p)$ is realized by conjugation in $\GL_2(\bF_p)$ \cite{Chen}*{Proposition~5.2.13}.
 Thus, by taking the absolute $\SL_2(\bF_p)$-structure on $E$ associated to a representation $\varphi\in F(p)$, we find that $F(p)$ is in bijection with the fiber of the forgetful map
\[
 M(\SL_2(\bF_p))^\abs\lra M(1) := \bA^1_j
\]
above a point $j_0 \ne 0,1728$.\footnote{Here we use the fact that $-I = \spmatrix{-1}{0}{0}{-1}$ in the mapping class group $\SL_2(\bZ)$ acts trivially on trace triples. Otherwise the coarse fiber would be the quotient of $F(p)$ by the action of $\langle -I\rangle$; see \cite{BBCL}*{Proposition~2.4.2 and the discussion after Remark~2.3.3}.}
Then, the orbits of the subgroup $\Out^+(\Pi)$ are in one-to-one correspondence with the components of $M(\SL_2(\bF_p))^\abs$, and the cardinality of an orbit is equal to the degree of the corresponding component over $\bA^1_j$ (see \cite{BBCL}*{\S2-3} and \cite{Chen}*{\S2,\S5}).

As the natural surjection $\SL_2(\bF_p)\to G_p$ induces a surjective (branched) cover
\[
 M(\SL_2(\bF_p))^\abs\to M(G_p)^\abs,
\]
by the correspondence described above, Theorem \ref{thm:nm-degree} implies that for any $(x,y,z)\in F(p)$, we have
\[
 |\Out(\Pi)\cdot (x,y,z)| \ge p.
\]
This gives the desired polynomial lower bound, which combined with Theorem \ref{thm:NKN} gives the classification conjectures for almost all $p$:
\begin{corollary}[See \S\ref{nm:proof-corollaries}]\label{cor_MW} There exists an integer $N$ such that for all primes $p > N$ and all $\tau\in\bF_p$, the group $\Out(\Pi)$ acts transitively on $F(p)_\tau$ if $\tau\ne 2$, and $F(p)_\tau$ is empty if $\tau= 2$.
\end{corollary}

As we move around on $M(\SL_2(\bF_p))^\abs$, the trace of the local monodromy around $O\in E$ defines a locally constant function on $M(\SL_2(\bF_p))^\abs$. For $\tau\in\bF_p$ write $M(\SL_2(\bF_p))^\abs_\tau$ to mean the open and closed subscheme consisting of elliptic curves with absolute structures with local monodromy trace $\tau$; equivalently, this is the union of components coming from orbits of $F(p)_\tau$.

In terms of the components of the moduli schemes $M(\SL_2(\bF_p))^\abs$, Corollary \ref{cor_MW} gives:
\begin{corollary}[See \S\ref{nm:proof-corollaries}]\label{cor_connected} There exists an integer $N$ such that for all primes $p > N$, we have
\[
 M(\SL_2(\bF_p))^\abs = \bigsqcup_{\tau\in\bF_p \setminus \{2\}}M(\SL_2(\bF_p))^\abs_\tau ,
\]
where $M(\SL_2(\bF_p))^\abs_\tau$ is the open and closed subscheme of points with trace invariant $\tau$ \cite{BBCL}*{\S2.5}, and each $M(\SL_2(\bF_p))^\abs_\tau$ is nonempty and connected.
\end{corollary}
The function $\tau$ remains well defined on $M(G_p)^\abs$ \cite{BBCL}*{Remark~2.5.8 and Definition~2.5.9}, so we get a similar decomposition of $M(G_p)^\abs$.

We also derive genus asymptotics for these components. See Proposition~\ref{nm:genus-fiber} for a more precise statement.
\begin{corollary}[Genus asymptotics, see \S\ref{nm:genus}]\label{cor_genus}
Let $N$ be as in Corollary~\ref{cor_connected}, let $p>N$ be prime and let $\tau\in\bF_p\setminus\{2\}$. Then $M(\SL_2(\bF_p))^\abs_\tau$ and $M(\Gp)^\abs_\tau$ are connected curves of degrees $p^2+O(p)$ and $p^2/4+O(p)$ over the $j$-line, and the genera of their smooth compactifications are
\[
 \frac{p^2}{12}+O_\epsilon\bigl(p^{1+\epsilon}\bigr)
 \qquad\text{and}\qquad
 \frac{p^2}{48}+O_\epsilon\bigl(p^{1+\epsilon}\bigr)
\]
for every $\epsilon>0$, where the implied constants depend only on $\epsilon$.
\end{corollary}

\subsection{Strong approximation and relation to the Markoff case}

The classification conjecture is related to strong approximation on the surfaces $\bX_\tau$. Recall that for $\tau\in\bZ$ we have a bijection $\bX_\tau(\bF_p)\cong F(p)_\tau^{\operatorname{irr}}$ whenever $\tau\not\equiv 2\bmod p$. Write $\bX_\tau(\bF_p) = F(p)_\tau\sqcup E(p)_\tau$, where $F(p)_\tau$ is the surjective locus and $E(p)_\tau$ is its complement, the \emph{exceptional locus}. For $\tau = -2$ the exceptional locus is the singleton $\{(0,0,0)\}$ \cite{CamposVargas}*{Theorem~6}, so transitivity on the surjective locus, as predicted by the classification conjecture, together with the existence of an integral lift of a single surjective point, such as $(3,3,3)$, implies that $\bX_{-2}$ satisfies \emph{strong approximation}, in the sense that the reduction map
\[
 \bX_{-2}(\bZ)\to \bX_{-2}(\bF_p)
\]
is surjective. For general $\tau$ strong approximation can fail. However, when transitivity holds (such as for $p > N$ as in Corollary~\ref{cor_MW}), and when some surjective point admits an integral lift, the failure is confined to the exceptional locus, which consists of reductions of the finite $\Out(\Pi)$-orbits over $\Qbar$ found by Dubrovin and Mazzocco \cite{DubrovinMazzocco}; these reductions are listed in \cite{CamposVargas}*{Theorem~6}, and the list gives $|E(p)_\tau|\le 114$ for all $p$ and $\tau$. In particular, for $p>N$, when there exists a non-exceptional integral point, strong approximation fails on at most $114$ points of $\bX_\tau(\bF_p)$.

The Markoff case of Theorem \ref{thm:nm-degree} was proven in \cite{Chen}*{Theorem~1.2.5} by studying certain line bundles on the moduli stack of elliptic curves with $\Gp$-structures.\footnote{A simpler proof was later found by Martin \cite{Martin}. Both the author's and Martin's proofs obtain not only the inequality $d_X\ge p$, but a divisibility $p\mid d_X$. The methods of this paper can also be used to give yet another proof of the divisibility result in the Markoff case, but this requires a little additional machinery that we will develop in a separate paper. We note that the divisibility statement fails for general $\tau$ -- see the tables in \cite{ChenModuli}*{Appendix~B.2}.} The resulting lower bound was then combined with Bourgain, Gamburd, and Sarnak's asymptotic upper bound \cite{BGS} to yield Corollary~\ref{cor_MW} for $\tau = -2$.

The proof of Theorem \ref{thm:nm-degree} in this case was special in that when $\tau = -2$, the corresponding $\SL_2(\bF_p)$-covers of elliptic curves have ramification index $2p$ above the origin. The ``$p$'' here is what leads to the $\ge p$ lower bound. For other values of $\tau$, the ramification indices can, in general, be any divisor of $p\pm 1$ greater than 2, and worse -- they will typically vary with $p$, even when $\tau$ is a fixed integer. The approach in this paper circumvents this issue by studying, instead, the reduction of the covering map $M(G_p)^\abs\to\bP^1_j$ at places above $p$.

As Theorem \ref{thm:nm-degree} can be verified directly for small primes, and the Markoff case $\tau = -2$ was done in \cite{Chen}, it suffices to prove the theorem for non-Markoff components and $p > 7$.
The proof occupies \S\ref{sec:arithmetic}--\ref{nm:full-cusp}
and is assembled in \S\ref{nm:proof-main}. In the rest of the introduction we explain some of the key elements that go into the proof.

\subsection{Bad reduction for three-point \texorpdfstring{$\PSL_2(\bF_p)$}{PSL2(Fp)}-covers}

The first major input concerns $G_p = \PSL_2(\bF_p)$-covers of the projective line and
is independent of the moduli problem.
For a finite group $G$, let $f\colon Y\to\bP^1_{\Qbar}$ be a
connected $G$-cover branched at $0,1,\infty$, and let $v$ be a place
of $\Qbar$ above $p$. We say that $f$ has \emph{good reduction} at $v$
if, for some finite extension $K_v$ of $\bQ_p$ inside the completion
of $\Qbar$ at $v$ over which $f$ is defined, the normalization of
$\bP^1_{R_v}$ in the cover over $K_v$ is \'etale over the complement
of the three sections $0$, $1$ and $\infty$, where $R_v$ is the
valuation ring of $K_v$. Otherwise $f$ has \emph{bad reduction} at $v$.
In \S\ref{bg:reduction} we define it for covers of marked curves in
general and show that, for the covers of this paper, it agrees with the
notions of good reduction used in the literature.
Here, and for every marking in the paper, a cover is marked
above the marked points of its base (\S\ref{bg:reduction}), including potentially unramified points.
For every three-point cover of $\bP^1$ in this paper, the base $\bP^1$ is marked at 3 points: either $0,1,\infty$ or $0,1728,\infty$ according to context.

\begin{theorem}\label{thm:nm-arithmetic}
Let $p>5$ be prime and let
$f\colon Y\to\bP^1_{\Qbar}$ be a connected
$\Gp$-cover branched at $0,1,\infty$. Then there is a place above $p$
at which $f$ has bad reduction. If one of the three ramification indices
of $f$, over $0$, $1$ and $\infty$, is divisible by $p$, then $f$ has bad
reduction at every place above $p$.
\end{theorem}

If one of the ramification indices is divisible by $p$, then $f$ has bad reduction at every place above $p$ (see Lemma~\ref{bg:good-model}(c); cf. \cite{BouwWewers}*{Theorem~6.6(a)}).\footnote{As an example, consider the map $\bP^1_{t,\bZ_p}\to\bP^1_{y,\bZ_p}$ given by $y=t^p$ -- this induces an inseparable map of special fibers.}
The main content of the theorem is thus the case of ramification indices prime to $p$.
In this case, we use the classification of Bouw and Wewers \cite{BouwWewers}*{Theorem~6.6}, which describes the reduction types of $\SL_2(\bF_p)$-covers of $\bP^1$ at $p$ in terms of the triple of conjugacy classes in $\SL_2(\bF_p)$ given by local monodromy around the three branch points (see \S\ref{nm:bw-statement} below). Applying an element of $\Gal(\Qbar/\bQ)$ to the cover modifies the local monodromy classes by powering through the cyclotomic character.
In Lemma \ref{nm:indefinite-place} we show that there exists a Galois conjugate of $f$ that has bad reduction at some place above $p$; this proves Theorem \ref{thm:nm-arithmetic}.

\subsection{Strategy of the proof of Theorem~\ref{thm:nm-degree}}
\label{intro:strategy}

Degeneration is the link between Theorems~\ref{thm:nm-degree} and~\ref{thm:nm-arithmetic}.

Let $X\subset M(G)^\abs$ be a component with smooth compactification $\ol{X}$. Points in $\ol{X} - X$ are called \emph{cusps}.
Following \cite{Chen}*{Definition~2.5.11}, we call the
objects parametrized by cusps \emph{cuspidal}.
The cuspidal object attached to a cusp $c\in\ol{X} - X$ consists of a $G$-cover
$f_c\colon Y_c\to E_c$, the \emph{cuspidal cover}, where $E_c$ is a
nodal rational curve with origin $O_c$, a \emph{$1$-gon} in the
terminology of \cite{DeligneRapoport}*{II.1.1}, and
$f_c$ is the finite map into which nearby $G$-covers of
elliptic curves degenerate at $c$. The cuspidal cover is
admissible in the sense of \cite{Chen}*{Definition~2.1.4}: $Y_c$ is nodal, the map sends nodes to nodes and is
\'etale over the smooth locus of $E_c$ away from $O_c$, and the local
inertia characters at the two branches of each node are mutually
inverse (cf.\ Lemma~\ref{nm:cuspidal-degeneration}).

The normalization $\widetilde E_c$ of
$E_c$ is a projective line carrying three distinguished points: the origin
$O_c$ and the two branches $0_c,\infty_c$ of the node. Normalizing
$Y_c$ gives the \emph{normalized cuspidal cover}
$\widetilde Y_c\to\widetilde E_c$, a $G$-cover branched at these
three points. As a point $t$ of $X$ tends to $c$, a simple closed
curve on the elliptic curve $E_t$, the \emph{vanishing cycle}, is
pinched to the node of $E_c$. Let $(a,b)$ be the monodromies around an
oriented basis of $\pi_1(E_t^\circ)$ whose first element is the class
of the vanishing cycle; we say that $(a,b)$ is a \emph{Nielsen pair
representing $c$}, and we call $h:=\ord(a)$ the \emph{pinched order}
of $c$. Each connected
component of $\widetilde Y_c$ is a Galois cover with group
\[
 H_a=\langle a,\,ba^{-1}b^{-1}\rangle
\]
and local monodromies $a$, $ba^{-1}b^{-1}$, $[a,b]^{-1}$ at
$0_c$, $\infty_c$, $O_c$, whose product is one and whose orders are
$(h,h,e)$. We write
$\nu_c\colon\widetilde Y_c^{1}\to\widetilde E_c$ for one such
component, the \emph{connected normalized cuspidal cover}; this
notation will be used in \S\ref{nm:cuspidal-notation}. The group
$H_a$ can be a proper subgroup of $G$ even though $(a,b)$ generates
$G$. In our proof of Theorem \ref{thm:nm-degree}, we are especially interested in the case where $H_a = G_p$, equivalently $Y_c$ and $\wt{Y}_c$ are irreducible. Such a cusp $c$ will be called \emph{irreducible}.

The proof of Theorem~\ref{thm:nm-degree} has four steps.

\begin{enumerate}[label={(\arabic*)}]
\item \emph{Existence of an irreducible cusp}
(\S\ref{nm:full-cusp}). Every component under consideration has
a cusp $c$ with $H_a=\Gp$ (Lemmas~\ref{nm:large-primitive} and
\ref{nm:full-normalization}). If its pinched order $h$ equals $p$,
the ramification index of $j_X$ at $c$, the \emph{width} of $c$, is
$p$, and so we'd immediately have $d_X\ge p$ 
(Lemma~\ref{nm:width}). Otherwise
$\nu_c\colon\widetilde Y_c\to\widetilde E_c$ is a connected
three-point $\Gp$-cover whose ramification indices $(h,h,e)$ are prime to
$p$, and
Theorem~\ref{thm:nm-arithmetic} gives a place $v$ above $p$ at which
it has bad reduction.

\item \emph{Small degree gives good reduction of the $j$-map}
(\S\ref{nm:proof-main}). Suppose $d_X<p$. The monodromy group
of the $j$-map $j_X\colon\overline X\to\bP^1_j$ embeds in the symmetric
group $S_{d_X}$, so its order is prime to $p$, and Grothendieck's theory
of tame covers (Theorem~\ref{bg:grothendieck}) gives good reduction of
$j_X$ at every place above $p$: over the valuation ring $R_1$ of a
finite extension $K_1$ of $\bQ_p$ inside the completion of $\Qbar$ at
$v$, the curve $\overline X$ has a smooth model which is finite over
$\bP^1_{j,R_1}$ and \'etale outside the sections $0$, $1728$ and
$\infty$.

\item \emph{Good reduction over the generic elliptic curve}
(\S\ref{nm:sec-modular} and Step~1 of
\S\ref{nm:sec-proof-bad-cusp}). Using the moduli interpretation of
$X$, we construct a cover $V$ of $X$ in characteristic $0$ that carries
an elliptic curve and a $\Gp$-cover of it.
Since $V\to X$ has degree at most four, a good model of $j_X$ at $v$
yields a good model $\overline{\cV}$ of the $j$-map of the smooth
compactification $\overline V$ (Lemma~\ref{bg:ascent}). The special
fiber $\overline{\cV}_k$ is smooth and irreducible, and
the reduced $j$-map $\overline{\cV}_k\to\bP^1_k$ is separable. Its
function field $F$ is therefore a finite separable extension of
$k(\bar j)$, so the reduced $j$-invariant $\bar j$ is not a $p$th power
in $F$. Completing the function field of $V$ at the generic point of
$\overline{\cV}_k$ gives a discretely valued field of characteristic
$0$ with residue field $F$ over which the elliptic curve and its
$\Gp$-cover are defined. If this $\Gp$-cover had bad reduction, then an
analysis of the stable reduction would force $\bar j\in F^p$
(Proposition~\ref{nm:generic-criterion}) and hence $\overline{\cV}_k\to\bP^1_k$ must be inseparable. This contradicts our assumption, so it must have good reduction.

\item \emph{Extending good reduction to the cusp}
(\S\ref{nm:cusp-degree}). Let $\bar{c}'$ be a point of $\ol{\cV}$ lying above the specialization $\bar{c}$ of the irreducible cusp $c$ of width prime to $p$. The completed local ring of $\ol{\cV}$ at $\bar{c}'$ takes the form $R_1[[s]]$ for $R_1$ the valuation ring of some finite extension of $\bQ_p$ with uniformizer $\pi$. In \S\ref{nm:sec-modular} we use an explicit Weierstrass model of the elliptic base from the previous step to extend it across $\bar{c}'$, producing a $G_p$-cover $\cY_s\to\cE_s$ over $R_1[[s]]$ which is \'{e}tale above every generic point of every fiber above any codimension-1 point of $R_1[[s]]$. A purity argument implies that the same is true along the fiber above the special cusp $(\pi,s)$. This ultimately shows that the normalization of the restriction of $\cY_s\to\cE_s$ to the locus $s = 0$ is a 3-point $G_p$-cover $\wt{\cY}_c\to \bP^1_{R_1}$ which is a good model of $\nu_c$, contradicting step (1). Therefore $d_X\ge p$.

\end{enumerate}

Steps (3) and (4) are carried out in Proposition~\ref{nm:bad-cusp},
which says that good reduction of $j_X$ at $v$ gives good reduction of
$\nu_c$ at $v$.

\begin{remark}\label{intro:curve-vs-cover}
Our argument does not imply that the component $X\subset M(G_p)^\abs$ has bad reduction as a curve. For example, every component of $M(G_{11})^\abs$ has genus $0$ \cite{ChenModuli}*{Appendix~B.2}, and hence has good reduction as a curve. What our argument shows is that the $j$-map $j_X\colon\overline X\to\bP^1_j$ has bad reduction as a cover (Definition~\ref{bg:good-reduction}) at some place $v$ above $p$: there is no finite extension $K$ of $\bQ_p$ inside the completion of $\Qbar$ at $v$ over which $\overline X$ has a smooth model that is finite over $\bP^1_{j,R}$ and \'{e}tale outside the sections $j = 0,1728,\infty$ (Corollary~\ref{nm:bad-reduction-jmap}).

This bad reduction implies the following structural properties about $j_X : \ol{X}\to\bP^1_j$. First, by the theory of tame covers, its monodromy group must have order divisible by $p$. Second, let $R$ be the valuation ring of an extension $K/\bQ_p$ over which $j_X$ is defined, and let $\cX$ be the normalization of $\bP^1_{j,R}$ in $\ol{X}$. Then after base changing via a potentially ramified finite extension $R\to R'$, the normalization of $\cX_{R'}$ has reduced special fiber \cite{Stacks}*{09IJ}. Purity of the branch locus then forces this special fiber to be inseparable over the special $j$-line.
\end{remark}

In some sense, our argument exploits a tension between the ``subtle'' reduction of 3-point $G_p$-covers of $\bP^1$, where badness depends on the choice of a place above $p$, and the very strong good reduction criterion from the theory of tame covers.

The inseparability criterion used in step (3) is proved in \S\ref{sec:generic}.
Let $K$ be a complete discretely valued field of mixed characteristic
whose residue field $F$ is a one-variable function field over a
perfect field $k$. Let $E$ be an elliptic curve over $K$ with good
reduction whose reduced $j$-invariant $\bar j\in F$ is transcendental
over $k$, and let $Y\to E$ be a geometrically connected $\Gp$-cover,
$p>7$, branched only over the origin, with ramification index $e>1$
prime to $p$. The criterion states that bad reduction of
$Y\to E$ forces $\bar j\in F^p$
(Proposition~\ref{nm:generic-criterion}).

The main difficulty to proving this criterion is that $F$ is imperfect. The stable marked model
$\cY$ of $Y$ is in general only available after a finite extension of $K$,
whose residue extension may be inseparable, and the conclusion
$\bar j\in F^p$ may become invisible under such extensions. Let $\bar Y$
be the special fiber of $\cY$ over an algebraically closed field
containing $F$, the \emph{stable marked source}. We descend $\bar Y$,
with its $\Gp$-action and its unordered marking, to a curve
over $F$ itself. This descent is governed by the group of
automorphisms of $\bar Y$ commuting with $\Gp$ and preserving the
marking. Stable reduction theory bounds the number of tails of the
stable target $\bar X$, the special fiber of $\cY/\Gp$, and determines
their ramification invariants (Lemma~\ref{gen:tails}), and from this
we show that the group has order prime to $p$
(Proposition~\ref{gen:source-centralizer}). The closure of the
$K$-point defined by $Y$ in the moduli stack of stable marked curves
with a $\Gp$-action is then a tame proper stack, and the arithmetic
valuative criterion of Bresciani and Vistoli \cite{BrescianiVistoli}
supplies an $F$-point of its special fiber
(Lemma~\ref{gen:source-descent}). With this in hand, the conclusion $\bar j\in F^p$ ultimately derives from the existence of an inseparable cover in the stable reduction.

\subsection{Acknowledgements and AI assistance}
The author thanks Peter Sarnak for introducing him to this problem. This paper was produced with heavy assistance from ChatGPT 5.6 Sol and Astra via Codex, and Claude Opus 5 and Fable 5.1, using custom pipelines developed by the author in conjunction with components shared by Taylor Dupuy. These models were used in every part of the paper, including research, writing, editing, and adversarial review.

The connection between stable reduction theory and large orbits was first observed by the author in 2020 as a potential path to obtain the lower bounds in the Markoff case $\tau = -2$, for which he ultimately found a different path that produced a stronger result \cite{Chen}. The main pieces of the argument for Theorem \ref{thm:nm-degree} are due to the author, including the usage of the work of Bouw-Wewers \cite{BouwWewers}, the inseparability of the special fiber map over the field of definition of the stable reduction used in Step 3, and their relation via the cusps of the moduli space. The author has learned a great deal from his interactions with AI, which was very helpful in making the ideas precise and bridging gaps, perhaps most notably in applying the arithmetic valuative criterion of \cite{BrescianiVistoli} to descend the inseparability to $F$. The author takes full responsibility for the correctness and quality of this work.

\subsection{Organization and conventions}

In \S\ref{sec:background} we fix the trace-coordinate description
of absolute structures, the structure of subgroups of $\Gp$ used
throughout, and the notion of good reduction of covers. Theorem~\ref{thm:nm-arithmetic} is proved in
\S\ref{sec:arithmetic}, and the generic inseparability criterion in
\S\ref{sec:generic}. They are combined in \S\ref{nm:cusp-degree}:
good reduction of the $j$-map at a place above $p$ extends to good
reduction of the normalized cuspidal cover at a cusp
(Proposition~\ref{nm:bad-cusp}). In \S\ref{nm:full-cusp} we find a
cusp to which Theorem~\ref{thm:nm-arithmetic} applies, and
\S\ref{nm:collation} completes the proof of Theorem~\ref{thm:nm-degree}
and proves Corollaries~\ref{cor_MW} and~\ref{cor_connected}. In \S\ref{nm:genus} we compute
the genus of the curves $M(\SL_2(\bF_p))^\abs_\tau$ and
$M(\Gp)^\abs_\tau$ (Corollary~\ref{cor_genus}). In Appendix~\ref{app:converse}
we correct a minor error in Theorem~6.6(b) of \cite{BouwWewers}.

Throughout the paper all relevant rings and schemes will be excellent. In particular, normalizations are finite and commute with completion of the base ring \cite[07GG, 0C22]{Stacks}.

\section{Subgroups of $\PSL_2(\bF_p)$ and covers of curves}
\label{sec:background}

\subsection{Generating pairs and trace coordinates}
\label{bg:sheets}

Let $(E,O)$ be a pointed complex elliptic curve, and fix an oriented
basis of $\pi_1(E^\circ)$. As in \S\ref{sec:introduction}, an
absolute $G$-structure on $E$ is a generating pair $(a,b)$ of $G$ up to
$\Aut(G)$, and an \emph{inner} structure is a generating pair up to
simultaneous conjugation. The group $\SL_2(\bZ)\cong\Out^+(\pi_1(E^\circ))$ acts on structures
through changes of the oriented basis, and its
orbits on absolute structures correspond to the geometric connected
components of $\Mabs{G}$; this is the topological description of the
algebraic moduli problem in \cite{Chen}*{Theorem~2.5.2, Definition~2.5.13,
Situation~2.5.14}. The fiber of $\Mabs{G}$ over a point $j\ne0,1728$
of the $j$-line is the set of absolute structures on the elliptic
curve $E$ with that $j$-invariant, modulo
$\Aut(E,O)=\{\pm1\}$, and the involution $[-1]$ acts on structures
by $(a,b)\mapsto(a^{-1},b^{-1})$. We call the points of such a fiber of
$j_X$ the \emph{absolute sheets} of $X$. For $\Gp$ the action of $[-1]$
is trivial on absolute structures, and the count of sheets can be
made in trace coordinates.

\begin{lemma}[Absolute sheets in trace coordinates]
\label{bg:absolute-sheets}
Let $p$ be an odd prime, and let $a,b$ be a generating pair of $\Gp$. If $A,B$ are lifts of $a,b$ to $\SL_2(\bF_p)$, the other lifts are $\pm A,\pm B$. To a particular lift $A,B$, associate the trace triple $(x,y,z) = (\tr A, \tr B,\tr AB)$. Then the trace triples of the lifts $(\pm A, \pm B)$ form the set
\begin{equation}\label{bg:signs}
 \{(\epsilon_1 x, \epsilon_2 y, \epsilon_1\epsilon_2 z)\;|\; \epsilon_1,\epsilon_2\in\{1,-1\}\}.
\end{equation}
We call this set the \emph{sign class of trace triples} associated to $(a,b)$.
\begin{enumerate}
\item Every automorphism of $\Gp$ is conjugation by an element of
$\PGL_2(\bF_p)$.
\item Two generating pairs define the same absolute structure if and
only if their associated sign classes agree.
\item The pairs $(a,b)$ and $(a^{-1},b^{-1})$ define the same
absolute structure.
\end{enumerate}
Consequently, for a geometric connected component $X$ of
$\Mabs{\Gp}$, the degree $d_X$ of $X$ over $\bP^1_j$ is the number of sign classes of trace
triples in the $\SL_2(\bZ)$-orbit that defines $X$.
\end{lemma}

\begin{proof}
(a) Conjugation gives a faithful action of $\PGL_2(\bF_p)$ on $\Gp$ whose
image in $\Out(\Gp)$ has order two \cite{Wilson}*{\S3.3.4}, and for odd $p$
the group $\Out(\Gp)$ itself has order two \cite{Wilson}*{Theorem~3.2}.

(b) By (a), if two pairs define the same absolute structure, their lifts
can be written as $(A,B)$ and $(\epsilon_1gAg^{-1},\epsilon_2gBg^{-1})$
with $g\in\GL_2(\bF_p)$, so their trace triples lie in one sign class.
Conversely, if the sign classes agree, choose lifts with equal trace
triples. These lifts are absolutely irreducible, since a common eigenline
over $\overline{\bF}_p$ would place $\Gp$ in the stabilizer of a point of
$\bP^1(\overline{\bF}_p)$, whose finite subgroups have a normal Sylow
$p$-subgroup. Hence they are conjugate under $\GL_2(\bF_p)$
\cite{Chen}*{Theorem~5.2.10}, and the pairs differ by an element of
$\PGL_2(\bF_p)=\Aut(\Gp)$.

(c) Inversion preserves traces in $\SL_2(\bF_p)$, and
$\tr A^{-1}B^{-1}=\tr(BA)^{-1}=\tr AB$; so $(a^{-1},b^{-1})$ has the sign
class of $(a,b)$, and (b) applies.

Finally, $d_X$ is the number of absolute structures on a general elliptic
curve in the component modulo $\Aut(E,O)=\{[\pm1]\}$, which acts
trivially by (c); by (b) these structures are the sign classes of trace
triples in the orbit.
\end{proof}

With $\tau=\tr[A,B]$ as in \S\ref{sec:introduction} and $(x,y,z) = (\tr A, \tr B, \tr AB)$, the trace identity
\begin{equation}\label{bg:fricke}
 \tau=x^2+y^2+z^2-xyz-2
\end{equation}
turns the condition $\tau=-2$ into the Markoff equation. Since
$\tau\ne2$ for a generating pair, a pair off the Markoff locus has
$\tau\ne\pm2$. Its commutator is then a semisimple matrix with
eigenvalues $\mu,\mu^{-1}\ne\pm1$ in $\bF_{p^2}$, and the order $e$ of
$[a,b]$ is a divisor greater than one of $(p-1)/2$ or of $(p+1)/2$; in
particular $p\nmid e$.

\subsection{Subgroups of \texorpdfstring{$\Gp=\PSL_2(\bF_p)$}{Gp = PSL2(Fp)}}
\label{bg:groups}

Three facts about finite groups recur. The first is the generation
of $\SL_2(\bF_p)$ by unipotent matrices, used to lift generating
triples from $\Gp$ to $\SL_2(\bF_p)$.

\begin{lemma}\label{bg:sl2-generation}
Let $p\ge5$. The group $\SL_2(\bF_p)$ is generated by its upper and
lower unipotent subgroups and is perfect. Consequently a subgroup of
$\SL_2(\bF_p)$ that maps onto $\Gp$ is all of
$\SL_2(\bF_p)$.
\end{lemma}

\begin{proof}
Generation by the two unipotent subgroups is row reduction, and for
$p\ge5$ every unipotent matrix is a commutator \cite{Wilson}*{\S3.3.2}.
A subgroup that maps onto $\Gp$ has index at most two in $\SL_2(\bF_p)$,
and a perfect group has no subgroup of index two.
\end{proof}

The second is the classification of finite subgroups of
$\PGL_2$ of order prime to the characteristic, which is classical.

\begin{lemma}[Tame finite subgroups of $\PGL_2$]
\label{bg:finite-subgroups}
Let $\kappa$ be an algebraically closed field and let $H$ be a
finite subgroup of $\PGL_2(\kappa)$ whose order is prime
to the characteristic of $\kappa$. Then $H$ is cyclic, dihedral, or
isomorphic to $A_4$, $S_4$ or $A_5$
\cite{Beauville}*{Introduction and \S1}.
In particular, for $p>5$ no such $H$ has $\Gp$ as a
quotient: cyclic and dihedral groups are solvable, and the other three
groups have order less than $|\Gp|$.
\end{lemma}

The third fact locates the proper subgroups of $\Gp$ that
can occur as decomposition groups or as groups of normalized cuspidal
covers. A \emph{Borel subgroup} is the stabilizer of a point of
$\bP^1(\bF_p)$, of order $p(p-1)/2$; a \emph{torus normalizer}
is the stabilizer of a Frobenius-stable pair of points of
$\bP^1(\bF_{p^2})$.

\begin{lemma}\label{nm:proper-subgroup}
Let $p\ge5$ be prime and let $H$ be a proper subgroup of
$\Gp$. If $p\mid|H|$, then $H$ lies in a Borel subgroup.
If $p\nmid|H|$ and $H$ contains an element of order greater than
five, then $H$ lies in a torus normalizer. In particular, a proper
subgroup containing an element of order greater than five lies in a
Borel subgroup or in a torus normalizer.
\end{lemma}

\begin{proof}
Suppose first that $p\mid|H|$, and choose a Sylow $p$-subgroup $P$
of $H$. A subgroup of order $p$ in $\Gp$ fixes a unique
point of $\bP^1(\bF_p)$, and that point determines the subgroup.
Conjugate so that $P$ is the group of upper unipotent matrices modulo
scalars. If $H$ does not normalize $P$, it contains a conjugate $P'$
fixing a different point of $\bP^1(\bF_p)$; a translation in
$P$ moves that point to $0$, and it conjugates $P'$ to a subgroup of
$H$ of order $p$ fixing $0$, which is the group of lower unipotent
matrices. So $H$ contains the full upper and lower
unipotent groups, which generate $\Gp$ by
Lemma~\ref{bg:sl2-generation}. This contradicts properness. Hence
$H$ normalizes $P$ and lies in the Borel subgroup fixing the point
fixed by $P$.

Now suppose $p\nmid|H|$ and that $H$ contains an element $h$ of order
greater than five. By Lemma~\ref{bg:finite-subgroups}, applied to $H$
as a subgroup of $\PGL_2(\overline{\bF}_p)$, the group $H$ is cyclic or
dihedral, because the elements of $A_4$, $S_4$ and $A_5$ have order at
most five. So $h$ lies in a cyclic normal subgroup $C$ of $H$: the
group $H$ itself, or its cyclic subgroup of index two, which contains
every element of order greater than two. A nontrivial element of
$\PGL_2(\overline{\bF}_p)$ of order prime to $p$ fixes exactly two
points of $\bP^1(\overline{\bF}_p)$, and the nontrivial elements of $C$
fix the same two points as a generator of $C$. Since $C$ is normal, $H$
preserves this pair. The pair is Frobenius-stable, because $h$ is
defined over $\bF_p$; so $H$ lies in a torus normalizer.
\end{proof}

\subsection{Good reduction of covers}
\label{bg:reduction}
In this section we let $K$ denote a complete discretely valued field of mixed characteristic $(0,p)$, with valuation ring $R$ and residue field $k$.
For smooth projective curves over $K$, good reduction is measured in terms of the smoothness of the stable model over $R$, or stable marked model if a marking is present, possibly after passing to a finite extension $K'/K$. For covers of curves, there are multiple potential models relative to which we can define ``good reduction''. In this section we give an overview of the relations between these models and the associated notions of good reduction.

\subsubsection*{Marked curves, covers and modular curves}
\label{bg:marked}

Every curve that serves as the base of a cover in this paper carries a
marking, which is assumed to contain the branch locus.
A \emph{marked curve} $(X,B)$ over a field $L$ is a smooth
proper geometrically connected curve $X$ over $L$ together with a closed
subscheme $B\subset X$, finite \'etale over $L$, such that $2g(X)-2+\deg B>0$.

\begin{definition}[Covers of marked curves]\label{bg:cover} A \emph{cover} of $(X,B)$
is a finite morphism $f\colon Y\to X$ from a smooth proper curve $Y$
which is \'etale over $X\setminus B$; the \emph{marked points} of $Y$
are the points above $B$, whether or not $f$ is ramified at them, and
the \emph{marking} of $Y$ is $B_Y:=f^{-1}(B)_{\mathrm{red}}$, which is
finite \'etale over $L$ when $L$ is perfect, the only case in which we use it.\footnote{Over an imperfect field a point above $B$ may have residue field inseparable over $L$, even with ramification index one, and then $B_Y$ is not \'etale.}
\end{definition}

To remind the reader of the marking, we will sometimes describe a cover as an $n$-point cover if $B$ has degree $n$.
An elliptic curve is marked at its origin. The target projective line of a
three-point cover is marked at $0$, $1$ and $\infty$; the $j$-line
$\bP^1_j$ is marked at $0$, $1728$ and $\infty$; and the normalization
$\widetilde E_c$ of the $1$-gon at a cusp $c$ (\S\ref{intro:strategy}) is
marked at the two branches $0_c$, $\infty_c$ of the node and at the
origin $O_c$, which makes it isomorphic to $(\bP^1;0,1,\infty)$ by the
unique isomorphism carrying $0_c$, $\infty_c$, $O_c$ to $0$, $\infty$, $1$.

A \emph{modular curve} is a
cover of $(\bP^1_j;0,1728,\infty)$, that is, a finite morphism
$j\colon\overline X\to\bP^1_j$ from a smooth proper curve which is
\'etale over $U_j:=\bP^1_j\setminus\{0,1728,\infty\}$. In characteristics $\ne 2,3$, this is the same as a Belyi map. Our definition does not impose any ramification conditions above $j=0,1728$, and so is a slight generalization of the classical notion. Our intention is to view a modular curve as a curve that carries a family of elliptic curves, and $j$ is the induced map to the $j$-line.
The compactified components $\overline X$ of
$\Mabs{G}$, with their $j$-maps, are modular curves in the classical
sense: quotients of the upper half plane by finite index subgroups of
$\SL_2(\bZ)$; the curve $\overline V$ of \S\ref{nm:sec-modular} need not
be.

\subsubsection*{Stable marked models} Recall the notion of a stable marked curve:
\begin{definition}[after \cite{Knudsen}*{Definition~1.1}] A \emph{stable marked curve} over $R$  is
a proper flat $R$-scheme $\cX$ with a closed subscheme $\cB\subset\cX$, finite \'etale over $R$, such that
\begin{enumerate}
	\item the geometric fibers of $\cX/R$ are reduced and connected with at most nodes as singularities,
	\item $\cB$ lies in the smooth locus of $\cX/R$, and
	\item every smooth rational component of a geometric fiber contains at least three points that are nodes or marked points, and every smooth genus 1 component contains at least one node or marked point.
\end{enumerate}
Given a marked curve $(X,B)$, a \emph{stable marked model} of $(X,B)$ over $R$ is a stable marked curve over $R$ with generic fiber $(X,B)$. In particular the geometric points of $B$ specialize to pairwise
distinct smooth points of the special fiber, which is itself a stable
marked curve over $k$.
\end{definition}

By the properness of the moduli stack of stable marked curves
\cite{Knudsen}*{Theorem~2.7}, applied after a finite extension over which $B$ is a union of rational points, there is a
finite extension $K'$ of $K$ such that
$(X_{K'}, B_{K'})$
has a stable marked model over the valuation ring of $K'$.
The stable marked model is unique, its formation commutes with finite extensions of $K$, and automorphisms of the generic fiber $(X,B)$ extend uniquely to the stable marked model: over a finite unramified Galois extension of $K$ the markings of stable marked curves over $R$ become disjoint unions of sections, so these are Knudsen's statements, and the isomorphisms they provide, being unique, descend.

We say that $(X,B)$ has \emph{good reduction} if its stable
marked model, over some finite extension of $K$, is smooth; we then call the stable marked model
\emph{the smooth model of $(X,B)$}.
For $p\ge5$ the marked projective lines $(\bP^1;0,1,\infty)$ and
$(\bP^1_j;0,1728,\infty)$ have good reduction, with smooth model
$\bP^1_R$ and the constant sections, and an elliptic curve marked at its
origin has good (marked) reduction if and only if the curve acquires good
reduction over a finite extension of $K$, the smooth model being the
smooth Weierstrass model.

\begin{definition}[The normal model of a cover]\label{bg:good-reduction}
Let $(X,B)$ be a marked curve over $K$
with good reduction and smooth model $\cX$ over $R$, and let
$f\colon Y\to X$ be a connected cover of $(X,B)$.
Let $\cY$ be the normalization of $\cX$ in $Y$. Since $R$ is complete, the normalization is finite,
and its generic fiber is $f$. The cover $\cY\to\cX$ is the \emph{normal model} of $f$ over $R$. It is a \emph{good model} of $f$ if it is \'etale over the complement of the marking $\cB$ of $\cX$, the closure of $B$. Thus a good model, when it exists, is uniquely
determined by $f$ and $\cX$. The cover $f$ has \emph{good reduction} if
it has a good model after a finite extension of $K$, and \emph{bad reduction} otherwise.

Let $(X,B)$ be a marked curve over $\Qbar$ and let $v$ be a place of
$\Qbar$ above $p$ such that $(X_K,B_K)$ has good reduction, where
$K$ is a finite extension of $\bQ_p$ inside the completion of $\Qbar$
at $v$ over which $X$ and $B$ are defined. A cover $f$ of
$(X,B)$ has good (resp. bad) reduction \emph{at $v$} if
$f\otimes K$ has good (resp. bad) reduction, for any such $K$ over which $f$
is defined.
\end{definition}

\begin{remark} The finite extension is necessary: for $p\ge 5$, let $X/K$ be given by $y^2 = px(x-1)(x-2)$. This equation defines the normal model of the double cover to $\bP^1_{\bZ_p,x}$, but it is not a good model since it ramifies along the entire special fiber. The good model is only available after the ramified base change to $\bZ_p\to \bZ_p[\sqrt{p}]$, over which the change of variables $y\mapsto \sqrt{p}y$ gives a good model of the form $y^2 = x(x-1)(x-2)$.
\end{remark}

\begin{lemma}[Structure of a good model]\label{bg:good-model}
Let $\cY\to\cX$ be a good model over $R$ of a cover $f\colon Y\to X$ of
$(X,B)$.
\begin{enumerate}
\item Locally for the \'etale topology on $\cX$ near $\cB$, with local
equation $s$, the scheme $\cY$ is a disjoint
union of Kummer covers $y^n=s$ with $n$ invertible in $R$.
\item $\cY$ is smooth and proper over $R$, and the reduced preimage
$\cB_Y$ of $\cB$ in $\cY$ is finite \'etale over $R$, with generic fiber
the marking $B_Y$.
\item The ramification indices of $f$ are prime to $p$.
\item The special fiber $\cY_k\to\cX_k$ is finite and separable on each
component.
\item For every extension $K'/K$ of complete discretely valued fields,
with valuation ring $R'$, the
base change $\cY\otimes_RR'\to\cX\otimes_RR'$ is the good model of
$f\otimes K'$.
\end{enumerate}
\end{lemma}

\begin{remark} Part (b) implies that good reduction of $f$ implies good reduction of $Y$. In the other direction, part (c) implies that a cover $\bP^1_K\to\bP^1_K$ with a ramification index divisible by $p$ has bad reduction even though both the source and target have good reduction. Thus, good reduction of covers is strictly stronger than good reduction of the source curve.
\end{remark}

\begin{proof}
Properness follows from finiteness over $\cX$. The marking $\cB$ is
finite \'etale over $R$, hence a regular effective Cartier divisor on
the regular scheme $\cX$, and $\cY\to\cX$ is finite \'etale over
$\cX\setminus\cB$ and tamely ramified at the generic points of $\cB$,
whose residue fields have characteristic zero. Abhyankar's lemma
\cite{Stacks}*{0EYG} gives the Kummer description in (a) over
$\cX\setminus\cB$. Near $\cB$, a local equation $s$ of $\cB$ is a relative coordinate: its zero scheme is
\'etale over $R$, so $s$ defines an \'etale morphism from a neighborhood
of $\cB$ to $\bA^1_R$, and the Kummer chart $y^n=s$, the base
change of this morphism along $y\mapsto y^n$, is \'etale over the affine
line with coordinate $y$, hence smooth over $R$ and normal. These charts
therefore give the normalization across $\cB$, proving (a).

Away from
$\cB$ the map is \'etale over a smooth $R$-curve, so $\cY$ is smooth
everywhere, as asserted in (b); (c) follows since $n$ is the ramification
index. In the chart the reduced preimage of $\cB$ is the divisor $y=0$,
\'etale over $R$; so $\cB_Y$ is finite \'etale over $R$, with generic
fiber $B_Y$. This proves (b). The special fiber
$\cY_k\to\cX_k$ is finite and \'etale over the complement of the
finitely many points of $\cB\cap\cX_k$, hence separable on each component,
which is (d). Finally, $\cY\otimes_RR'$ is smooth over $R'$, hence
normal, and is finite over $\cX\otimes_RR'$ with generic fiber
$f\otimes K'$. It is therefore the normalization in that cover, and
remains \'etale off the marking. This proves (e).
\end{proof}

\begin{definition}[Stable models of $G$-covers]\label{bg:stable-model} The cover $f : Y\to (X,B)$ is a \emph{connected $G$-cover} if it is Galois with group $G$.
Assume further that $Y$ is geometrically connected, and as usual marked by $B_Y$.
Possibly passing to a finite extension of $K$, we may
assume $Y$ admits a stable marked model $\cY/R$. By uniqueness, the
action of $G$ extends to $\cY$, preserving its marking. The \emph{stable model} of the marked cover $f$ is the finite quotient map
$$\cY\to\cY/G.$$
The map on special fibers $\cY_k\to(\cY/G)_k$ is the \emph{stable reduction} of $f$.\footnote{Beware that the formation of quotients does not in general commute with reduction when $|G|$ is not invertible in the residue field, so $(\cY/G)_k, \cY_k/G$, and the stable reduction of $X$ can all be different. Section \ref{gen:sec-centralizer} compares
the two curves $(\cY/G)_k$ and $\cY_k/G$ for $G=\Gp := \PSL_2(\bF_p)$.}
\end{definition}

\begin{remark} The stable model of a smooth curve is normal, since its singularities are in codimension 2, and the quotient map $\cY\to\cY/G$ is finite of degree $|G|$, with $\cY/G$ again normal. On the other hand, even when $(X,B)$ admits a smooth model $\cX/R$, $\cY$ is not in general the normal model, which in general can have nonreduced special fiber. These are related by a canonical birational map $\cY/G\to\cX$ which is an isomorphism on the generic fiber but can contract components of the special fiber. This also implies that $\cY\to\cX$ is in general not finite.
\end{remark}

Bad reduction is studied via the stable model in \cite{WewersBad, Wewers2005, ObusVanishing}. The lemma below describes the relation between the stable and normal models of the cover:

\begin{lemma}[Good reduction and the stable marked model]
\label{bg:comparison}
Let $f\colon Y\to X$ be a connected $G$-cover of $(X,B)$ over $K$ with
$Y$ geometrically connected of genus $\ge 2$, and assume that $(X,B)$ has good reduction.
\begin{enumerate}
\item $f$ has good reduction if and only if $Y$, marked by $B_Y$, has good (marked) reduction. In that case, after a finite extension of $K$, the good model
of $f$ is the stable marked model of $Y$, and its quotient by $G$ is the
smooth model of $(X,B)$.
\item Suppose that $k$ is algebraically closed, that $(X,B)$ is
$(\bP^1;0,1,\infty)$ and that $p$ divides $|G|$ but $p^2\nmid |G|$. Write
$\bar Y\to\bar X$ for the map on special fibers of $\cY\to\cY/G$,
the stable reduction of $f$. Then $f$
has good reduction if and only if $\bar Y\to\bar X$
is separable, if and only if $\bar X$ is smooth.
\end{enumerate}
\end{lemma}

\begin{proof}
For (a), let $\cY\to\cX$ be a good model of $f$. By Lemma~\ref{bg:good-model}(b),
$\cY$ is smooth with marking $\cB_Y$ finite \'etale over $R$, so it is a
stable marked curve, hence the stable marked model of $Y$; and $\cY/G$,
being normal and finite over $\cX$ with the same function field, is
$\cX$. Conversely, suppose that the stable marked model $\cY$ of $Y$ is
smooth. Its quotient $\cY/G$ is smooth over $R$ by the theorem of Katz
and Mazur \cite{KatzMazur}*{p.~508, Theorem}, applied to $G$-stable affine
open subschemes, and $G$ acts transitively on the fibers of
$\cY\to\cY/G$ \cite{SGA1}*{Expos\'e~V, Proposition~1.1}, so distinct
geometric points of $B$ have distinct specializations in $\cY/G$, because distinct
geometric points of $B_Y$ have distinct specializations in $\cY$. Hence
the closure of $B$ in $\cY/G$ is finite \'etale over $R$, and $\cY/G$ is
a smooth model of $(X,B)$, namely the smooth model $\cX$. The special fiber of
$\cY$ is smooth and geometrically connected, hence irreducible,
and $G$ acts faithfully on it \cite{Liu}*{Proposition~10.3.38}; so the
inertia group of its generic point, which acts trivially on the residue
field of that point, is trivial, and $\cY\to\cX$ is unramified at the
generic point of the special fiber. On the generic fiber $f$ is
unramified outside the marked points. By purity of the branch locus
\cite{Stacks}*{0BMB}, $\cY\to\cX$ is \'etale over the complement of
$\cB$, whose reduced preimage is $\cB_Y$ by the transitivity on fibers. Since $\cY$ is normal, finite over $\cX$
and flat over $R$, it is the normalization of $\cX$ in $Y$, hence a good model.

For (b), the branch locus is all of $B$: by Riemann--Hurwitz, a connected
Galois cover of the line with at most two branch points has genus zero.
If $f$ has good reduction, then by (a) the stable reduction
$\bar Y\to\bar X$ is the special fiber of the good model, which is
separable by Lemma~\ref{bg:good-model}(d), with $\bar X=\bP^1_k$
smooth. If $f$ has bad reduction, it has no good model over any finite
extension of $K$, and Theorem~2 of \cite{WewersBad}, which applies
because $p$ divides $|G|$ exactly once and $f$ has three branch points,
shows that $\bar X$ consists of the original component and at least one
tail, so that $\bar X$ is not smooth, and that $\bar Y\to\bar X$ is
inseparable over the original component.
\end{proof}

Definition~\ref{bg:good-reduction} is the good reduction of Wewers
\cite{WewersBad}*{\S1.1} and of Raynaud \cite{Raynaud1999}*{Th\'eor\`eme~0},
who consider Galois covers. By Lemma~\ref{bg:comparison} it agrees, for a
connected $G$-cover with branch locus $B$ whose source is geometrically
connected of genus at least two, with the potentially good reduction of
Obus \cite{ObusVanishing}*{\S2.3}, and, for three-point covers with
$p\,\|\,|G|$ over a field with algebraically closed residue field, with the
good reduction of Bouw and Wewers \cite{BouwWewers}*{\S2}.

The \emph{(geometric) monodromy group} of a geometrically connected cover $f : Y\to X$ is the Galois group of the Galois closure after base changing to an
algebraic closure of the base field. We use the following cover-theoretic
form of Grothendieck's specialization theorem
\cite{SGA1}*{Expos\'e~XIII, Corollaire~2.9 and \S2.10}. For
three-point covers of the line, the good-reduction conclusion is stated
explicitly in \cite{Wewers2005}*{Theorem~1.2(1)}.

\begin{theorem}[Grothendieck]\label{bg:grothendieck}
Let $K$ be a complete discretely valued field of mixed characteristic
$(0,p)$, let $(X,B)$ be a marked curve over $K$ with good
reduction, and let $f\colon Y\to X$ be a geometrically connected cover of $(X,B)$ whose
monodromy group has order prime to $p$. Then $f$ has good reduction:
after a finite extension of $K$, it admits a good model.
\end{theorem}

For a cover over $\Qbar$ whose monodromy group has order prime to $p$,
this gives good reduction at every place above $p$ at which the marked
base has good reduction.
We end the section with a lemma we will need later.

\begin{lemma}[Quotients, ascent, and descent]\label{bg:ascent}
Let $f\colon Y\to X$ be a cover of a marked curve $(X,B)$ over $K$ with
good reduction.
\begin{enumerate}
\item If $\cY\to\cX$ is a good model of $f$ and $H$ is a group of
automorphisms of $Y$ over $X$, then $\cY/H\to\cX$ is a good model of
$Y/H\to X$.
\item Suppose that $Y$ is geometrically connected, and let $g\colon Z\to Y$
be a finite morphism from a geometrically connected smooth proper curve,
\'etale over $Y\setminus f^{-1}(B)$, whose monodromy group has order
prime to $p$; for instance, $g$ may be Galois of degree prime to $p$, or
of degree less than $p$. If $f$ has good reduction, then so has the
cover $f\circ g$ of $(X,B)$.
\item Let $f\colon Y\to X$ be a connected $G$-cover of a marked curve
$(X,B)$ over $\Qbar$ with $Y$ geometrically connected of genus at least
two, let $v$ be a place of $\Qbar$ above $p$ at which $(X,B)$ has good
reduction, and let $K'$ be a complete discretely valued field with
algebraically closed residue field, inside the completion of $\Qbar$ at
$v$, over which $f$ is defined. Then $f$ has good reduction at $v$ if and
only if $f\otimes K'$ has good reduction.
\end{enumerate}
\end{lemma}

\begin{proof}
(a) The action of $H$ extends to $\cY$ by normalization. The quotient
$\cY/H$ is the normalization of $\cX$ in $Y/H$: on affine opens, taking
$H$-invariants in the integral closure gives the integral closure in the
invariant algebra. Over the complement of $\cB$ the
finite \'etale morphism $\cY\to\cX$ is, locally for the \'etale
topology, a disjoint union of copies of the base permuted by $H$, and the
quotient is again a disjoint union of copies of the base.

(b) After a finite extension of $K$, choose a good model $\cY\to\cX$ of
$f$. By Lemma~\ref{bg:good-model}(b), $\cY$ is smooth with finite \'etale
marking $\cB_Y$, and the marking of $Y$ is stable by Riemann--Hurwitz,
so $Y$ has good (marked) reduction with smooth model $\cY$.
Theorem~\ref{bg:grothendieck} gives, after a further finite extension
of $K$, a good model $\cZ\to\cY$ of $g$. Its composite with $\cY\to\cX$
is finite and \'etale away from $\cB$; by
transitivity of normalization, it is a good model of $f\circ g$.

(c) By Lemma~\ref{bg:good-model}(e), a good model over a finite extension
of a field $K$ as in Definition~\ref{bg:good-reduction} base changes to a
good model over a finite extension of $K'$. Conversely, suppose that
$f\otimes K'$ has a good model over a finite extension $K''$ of $K'$, so
that the stable marked model of $Y\otimes K''$ is smooth by
Lemma~\ref{bg:comparison}(a). Let $K_1$ be a finite extension of $K$ over
which the stable marked model of $Y$ exists. Over the compositum $K_1K''$,
taken inside the completion of $\Qbar$ at $v$, the base changes of the two
stable marked models coincide and are smooth, and smoothness descends along
the faithfully flat extension of valuation rings of $K_1\subset K_1K''$;
so the stable marked model over $K_1$ is smooth, and $f\otimes K_1$ has
good reduction by Lemma~\ref{bg:comparison}(a).
\end{proof}

\section{Three-point covers and the theorem of Bouw and Wewers}
\label{sec:arithmetic}

Throughout this section $p>5$ and
$f\colon Y\to\bP^1$ is a connected $\Gp$-cover over
$\Qbar$, branched at $0,1,\infty$. We will show that
Theorem~\ref{thm:nm-arithmetic} is a consequence of the theorem of
Bouw and Wewers \cite{BouwWewers}*{Theorem~6.6} on the reduction of
three-point covers with Galois group $\SL_2(\bF_p)$.
Over a complete discretely valued field with valuation ring of mixed characteristic $(0,p)$, they give a criterion for bad reduction of $\SL_2(\bF_p)$-covers, in terms of the conjugacy classes in $\SL_2(\bF_p)$ of the local monodromies around the three branch points in $\bP^1$.
In Lemmas~\ref{nm:dictionary} and~\ref{nm:indefinite-place}, together with Theorem~\ref{nm:bw-theorem}, we show that for \emph{any} connected three-point $\SL_2(\bF_p)$-cover $\tilde{f}$ defined over a number field there is a place $v\mid p$ at which their criterion applies to $\tilde{f}$ or to a twist of $\tilde{f}$ by a quadratic character with values in the center $\{\pm I\}$, so that $\tilde{f}$ has bad reduction at $v$. Twisting by such a character will be called an \emph{even sign change} below. In Lemma~\ref{nm:lifts} we show that reduction type is invariant under even sign changes and is preserved when passing to the induced $G_p$-cover.

\subsection{Lifts, angles and class vectors at a place}
\label{nm:class-vectors}

Fix an embedding $\iota\colon\Qbar\hookrightarrow\bC$. Topological
fundamental groups and monodromy refer to the covers of Riemann
surfaces obtained through $\iota$. For an integer $n\ge1$ let
$\widetilde\zeta_n\in\Qbar$ be the root of unity with
$\iota(\widetilde\zeta_n)=e^{2\pi i/n}$.

\emph{Lifts.} Let $g_1,g_2,g_3\in\Gp$ be the images of standard loops
around $0,1,\infty$; they are the local monodromies of $f$, with
$g_1g_2g_3=1$. A
\emph{lift} of $(g_1,g_2,g_3)$ is a triple $M=(M_1,M_2,M_3)$ of
matrices in $\SL_2(\bF_p)$ with $M_1M_2M_3=I$ such that $M_i$ maps to
$g_i$ for every $i$. Lifts exist: choose $M_1$ and $M_2$ above $g_1$
and $g_2$ and put $M_3=(M_1M_2)^{-1}$. There are four lifts, and any two
of them differ by changing the signs of two of the three matrices; we
call this operation an \emph{even sign change}. The matrices of a lift
generate $\SL_2(\bF_p)$ by Lemma~\ref{bg:sl2-generation}, since their
images generate $\Gp$. By Riemann's existence theorem, a lift $M$ is
the triple of local monodromies of a connected $\SL_2(\bF_p)$-cover
\[
 f_M\colon Y_M\longrightarrow\bP^1
\]
over $\Qbar$, branched at $0,1,\infty$, and $Y_M/\{\pm I\}$ is
isomorphic to $Y$ as a $\Gp$-cover of $\bP^1$. We call $f_M$ a
\emph{lift of $f$}. Its ramification points are the points of $Y_M$
above the marked points of $Y$.

\emph{Eigenvalues.} Suppose that the ramification indices of $f$ are prime
to $p$. Each $M_i$ is then a noncentral semisimple matrix. Let
$\delta_i^{\pm1}\in\bF_{p^2}^\times$ be its eigenvalues, so that
$\delta_i\ne\pm1$, and put $x_i:=\tr M_i=\delta_i+\delta_i^{-1}$. We
call $M_i$ \emph{split} if $\delta_i\in\bF_p$ and \emph{nonsplit}
otherwise; in the second case $\delta_i^{\,p}=\delta_i^{-1}$. By the
Fricke identity \eqref{bg:fricke},
\begin{equation}\label{nm:fricke-delta}
 \Delta:=x_1^2+x_2^2+x_3^2-x_1x_2x_3-4=\tr[M_1,M_2]-2,
\end{equation}
and $\Delta\ne0$, because the trace of the commutator of a generating
pair of $\SL_2(\bF_p)$ is not $2$
\cite{BrumfielHilden}*{Proposition~4.1}. An even sign change leaves
$\Delta$ unchanged.

\emph{Angles at a place.} Let $v$ be a place of $\Qbar$ above $p$ and
let $k_v$ be its residue field, an algebraic closure of $\bF_p$.
Reduction at $v$ is an isomorphism from the group of roots of unity of
order prime to $p$ in $\Qbar$ onto $k_v^\times$. Choose an embedding
$\bF_{p^2}\subset k_v$; the other embedding differs from it by the
Frobenius, which maps each of the sets $\{\delta_i,\delta_i^{-1}\}$ to
itself. Let $\widetilde\delta_i\in\Qbar$ be the root of unity of order
prime to $p$ that reduces to $\delta_i$ at $v$. The \emph{angle of
$M_i$ at $v$} is the number
\[
 \theta_i=\theta_i(v)\in\bigl(0,\tfrac12\bigr)
 \qquad\text{with}\qquad
 \iota(\widetilde\delta_i)=e^{\pm2\pi i\theta_i};
\]
the values $0$ and $\tfrac12$ are excluded because $\delta_i\ne\pm1$.
The angle only depends on $\iota$, $v$ and $M_i$. Replacing $M_i$
by $-M_i$ replaces $\theta_i$ by $\tfrac12-\theta_i$. Put
$\widetilde x_i:=\widetilde\delta_i+\widetilde\delta_i^{-1}$, so that
$\iota(\widetilde x_i)=2\cos2\pi\theta_i$, and
\begin{equation}\label{nm:lifted-delta}
 \widetilde\Delta_v:=
 \widetilde x_1^2+\widetilde x_2^2+\widetilde x_3^2
 -\widetilde x_1\widetilde x_2\widetilde x_3-4.
\end{equation}
This element of a cyclotomic field reduces to $\Delta$ at $v$, so it is
not zero. Its image $\iota(\widetilde\Delta_v)$ is a nonzero real
number and is independent of the choice of lifts $M_i$.

\begin{definition}\label{nm:indefinite}
Suppose that the ramification indices of $f$ are prime to $p$, and let $v$
be a place of $\Qbar$ above $p$. We say that $f$ is \emph{indefinite
at $v$} if $\iota(\widetilde\Delta_v)>0$, and \emph{definite at $v$} if
$\iota(\widetilde\Delta_v)<0$.
\end{definition}

The names refer to the pair of matrices in $\SL_2(\bC)$ with traces
$\iota(\widetilde x_1)$ and $\iota(\widetilde x_2)$ whose product has
trace $\iota(\widetilde x_3)$. This pair preserves a Hermitian form,
which is definite or indefinite according to the sign of
$\iota(\widetilde\Delta_v)$; see the proof of
Lemma~\ref{nm:indefinite-place}.

\emph{Class vectors.} Let $\zeta_{p-1}\in\bF_p^\times$ and
$\zeta_{p+1}\in\bF_{p^2}^\times$ be the reductions at $v$ of
$\widetilde\zeta_{p-1}$ and $\widetilde\zeta_{p+1}$. As in
\cite{BouwWewers}*{\S4}, the noncentral semisimple conjugacy classes of
$\SL_2(\bF_p)$ are
\begin{align*}
 C_{\rm sp}(l)&:=\{g\in\SL_2(\bF_p):
   \tr g=\zeta_{p-1}^{\,l}+\zeta_{p-1}^{-l}\},
   &&0<l<\tfrac{p-1}2,\\
 C_{\rm nsp}(l)&:=\{g\in\SL_2(\bF_p):
   \tr g=\zeta_{p+1}^{\,l}+\zeta_{p+1}^{-l}\},
   &&0<l<\tfrac{p+1}2,
\end{align*}
the \emph{split} and the \emph{nonsplit} classes; they are denoted by
$C(l)$ and $\widetilde C(l)$ in \cite{BouwWewers}. To such a class $C$
we attach the even integer
\begin{equation}\label{nm:bw-a}
 a(C):=\begin{cases}
 p-1-2l&\text{if }C=C_{\rm sp}(l),\\
 p+1-2l&\text{if }C=C_{\rm nsp}(l).
 \end{cases}
\end{equation}
By the definition of the angle, $M_i$ lies in $C_{\rm sp}(l_i)$ with
$l_i=(p-1)\theta_i$ if it is split, and in $C_{\rm nsp}(l_i)$ with
$l_i=(p+1)\theta_i$ if it is nonsplit. Hence
\begin{equation}\label{nm:a-and-theta}
 a_i:=a(\text{class of }M_i)=
 \begin{cases}
 (p-1)(1-2\theta_i)&\text{if $M_i$ is split,}\\
 (p+1)(1-2\theta_i)&\text{if $M_i$ is nonsplit.}
 \end{cases}
\end{equation}
The labels $l_i$, and with them the integers $a_i$, depend on the place
$v$ through the angles. The integers $a(C)$ come from the analysis of
the stable reduction of three-point $\SL_2(\bF_p)$-covers in
\cite{BouwWewers}.\footnote{In bad reduction, the branch point with class $C$
specializes to a primitive tail of the stable reduction, and the
ramification invariant $\sigma_x$ of \eqref{eq_ramification_invariants}
at the point $x$ of the original component where this tail is attached
(\S\ref{gen:sec-tails}) is $a(C)/(p-1)$ or $a(-C)/(p-1)$, according to
the canonical generator of the tail cover at that point, where $-C$
denotes the class of the negatives of the elements of $C$
(Lemma~\ref{app:pointed}).}

The \emph{class vector} of a connected $\SL_2(\bF_p)$-cover of $\bP^1$
branched at $0,1,\infty$ is the triple of the conjugacy classes of its
canonical generators at $0$, $1$, $\infty$ \cite{BouwWewers}*{\S4}, the
elements of the inertia groups that act on the cotangent spaces through
the chosen roots of unity. A noncentral semisimple class of
$\SL_2(\bF_p)$ is determined by its trace, so for such classes these are
the classes of the local monodromies; in particular the class vector of
$f_M$ at $v$ is the triple of the classes of $M_1$, $M_2$, $M_3$.

\subsection{The theorem of Bouw and Wewers}
\label{nm:bw-statement}

Let $K$ be a complete discretely valued field with valuation ring $R$ of mixed characteristic
whose residue field $k$ is algebraically closed of characteristic $p > 5$. In this subsection, $h\colon Z\to\bP^1_K$ denotes a connected $\SL_2(\bF_p)$-cover branched above $0,1,\infty$ with $Z$ geometrically connected; the theorem below will be applied to the base changes to $K$ of the lifts $f_M$ of \S\ref{nm:class-vectors}. In \cite{BouwWewers}, Bouw and Wewers say that $h$ has good reduction if its stable reduction (in the sense of Definition \ref{bg:stable-model}) is separable. By Lemma \ref{bg:comparison}(b), this agrees with our notion of good reduction in Definition \ref{bg:good-reduction}.\footnote{Indeed, the hypothesis $g(Z)\ge 2$ of that lemma holds, since a Galois cover of $\bP^1$ with $g(Z)\le1$ has group a finite subgroup of $\PGL_2$ (Lemma~\ref{bg:finite-subgroups}) or a group of automorphisms of a genus 1 curve; $\SL_2(\bF_p)$ is neither for $p > 5$. We also need Lemma~\ref{bg:ascent}(c) to descend from our $K$ to a finite extension of the field of definition of the cover.}

\begin{theorem}[Bouw--Wewers]\label{nm:bw-theorem}
Let $p\ge5$, suppose that $R$ contains a square root of
$(-1)^{(p-1)/2}p$, and fix primitive roots of unity
$\zeta_{p-1}\in\bF_p$ and $\zeta_{p+1}\in\bF_{p^2}$ in $k$. Let
$h\colon Z\to\bP^1_K$ be a connected $\SL_2(\bF_p)$-cover branched
at $0,1,\infty$ whose class vector $(C_1,C_2,C_3)$ consists of noncentral
classes.
\begin{enumerate}
\item If one of the classes $C_i$ consists of elements of order $p$,
then $h$ has bad reduction.
\item Suppose that the elements of every $C_i$ have order prime to $p$, and put
$a_i:=a(C_i)$ as in \eqref{nm:bw-a}. If
\begin{equation}\label{nm:bw-hypothesis}
\begin{array}{ll}
a_1+a_2+a_3<p-1, & \text{or}\qquad \\
a_1+a_2+a_3=p-1 & \text{and some $C_i$ is nonsplit,}
\end{array}
\end{equation}
then $h$ has bad reduction.
\end{enumerate}
\end{theorem}

This is \cite{BouwWewers}*{Theorem~6.6}, part (a) and the first two
cases of part (b).

\begin{remark}[The scope of the citation]\label{nm:bw-scope}
(a) The result is stated in \cite{BouwWewers} as a formula for the
number of covers with class vector $(C_1,C_2,C_3)$ that have bad
reduction. Under the hypothesis \eqref{nm:bw-hypothesis} this number is
two, which is the number of all covers with this class vector by
rigidity \cite{BouwWewers}*{Proposition~4.1}. This is the implication
stated above.

(b) The hypothesis \eqref{nm:bw-hypothesis} depends on the lift. An
even sign change replaces $l_i$ by $\frac{p-1}2-l_i$ or by
$\frac{p+1}2-l_i$ for two of the three indices, and hence $a_i$ by
$p-1-a_i$ or by $p+1-a_i$, whereas the four lifts of a $\Gp$-cover
have good reduction or bad reduction together
(Lemma~\ref{nm:lifts}). Thus \eqref{nm:bw-hypothesis} is a
sufficient condition that has to be tested on the four lifts. By
Lemma~\ref{nm:dictionary}, the cover $f$ is indefinite at $v$ exactly
when one of its four lifts satisfies \eqref{nm:bw-hypothesis}, and the
theorem is applied to that lift. The last case of
\cite{BouwWewers}*{Theorem~6.6(b)} states that the covers with all
other class vectors have good reduction. It does not take the sign
changes into account; Appendix~\ref{app:converse} gives a
counterexample and proves the corrected converse.
\end{remark}

\subsection{A place of bad reduction}
\label{nm:bw-application}

\begin{lemma}[Reduction of $f$ and of its lifts]\label{nm:lifts}
Let $v$ be a place of $\Qbar$ above $p$ and let $f_M$ be a lift of
$f$. Then $f$ has good reduction at $v$ if and only if $f_M$ has good
reduction at $v$. In particular, $f$ and its four lifts either all have
good reduction or all have bad reduction at $v$.
\end{lemma}

\begin{proof}
Both $f$ and $f_M$ are covers of $(\bP^1;0,1,\infty)$, and $f_M$ is the
composite of $f$ with the quotient map $Y_M\to Y=Y_M/\{\pm I\}$, a
Galois cover of degree two. This map is branched only at the points of
$Y_M$ whose stabilizer contains $-I$, which are ramification points of
$f_M$ and hence lie above the marked points of $Y$. If $f$ has good
reduction at $v$, then so has $f_M$, by Lemma~\ref{bg:ascent}(b) (recall that we assume $p > 5$ in this section). If
$f_M$ has a good model $\cY_M$, then $\cY_M/\{\pm I\}$ is a good model
of $f$, by Lemma~\ref{bg:ascent}(a).
\end{proof}

\begin{lemma}[The lift selected by an indefinite place]
\label{nm:dictionary}
Suppose that the ramification indices of $f$ are prime to $p$, let $v$ be a
place of $\Qbar$ above $p$, and let $\theta_1,\theta_2,\theta_3$ be the
angles at $v$ of a lift $M$.
\begin{enumerate}
\item Let $T_0:=\theta_1+\theta_2+\theta_3$, and for $i=1,2,3$ let $T_i$ be the same sum with $\theta_i$ negated: $T_1=-\theta_1+\theta_2+\theta_3$, $T_2=\theta_1-\theta_2+\theta_3$, $T_3=\theta_1+\theta_2-\theta_3$. Then
\[
 \iota(\widetilde\Delta_v)=-16\sin(\pi T_0)\sin(\pi T_1)\sin(\pi T_2)\sin(\pi T_3).
\]
\item The cover $f$ is indefinite at $v$ if and only if one of the four
lifts satisfies
\begin{equation}\label{eq_angles}
\theta_1+\theta_2+\theta_3>1.
\end{equation}
In this case, the lift is uniquely determined by this property \eqref{eq_angles}.
\item The angles of $M$ satisfy \eqref{eq_angles} if and
only if the integers $a_1,a_2,a_3$ of \eqref{nm:a-and-theta} satisfy
the hypothesis \eqref{nm:bw-hypothesis} of
Theorem~\ref{nm:bw-theorem}.
\end{enumerate}
Consequently $f$ is indefinite at $v$ if and only if one of its four
lifts satisfies \eqref{nm:bw-hypothesis} at $v$, and this lift has bad reduction at $v$.
\end{lemma}

\begin{proof}
(a) With $\alpha_i:=2\pi\theta_i$ we have
$\iota(\widetilde x_i)=2\cos\alpha_i$, and the assertion is the
identity
\[
 4\sum_i\cos^2\alpha_i-8\prod_i\cos\alpha_i-4
 =-16\,\sin\tfrac{\alpha_1+\alpha_2+\alpha_3}2\,
 \prod_{i=1}^3\sin\tfrac{\alpha_1+\alpha_2+\alpha_3-2\alpha_i}2,
\]
which follows from the product formulas for sines and cosines.

(b) Since $0<\theta_i<\frac12$, we have
$0<T_0<\frac32$ and $-\frac12<T_i<1$, so the first sine in (a)
is negative exactly when $T_0>1$ and the sine of $\pi T_i$ is
negative exactly when $T_i<0$. At most one of these four
conditions holds. Indeed, if $T_0>1$, then
$\theta_j+\theta_k>1-\theta_i>\theta_i$ for
$\{i,j,k\}=\{1,2,3\}$, so every $T_i$ is positive; and if
$T_1<0$, then $T_2=\theta_1-\theta_2+\theta_3>2\theta_3>0$,
likewise $T_3>0$, and $T_0<2\theta_1<1$. No sine vanishes,
because $\widetilde\Delta_v\ne0$. Hence $\iota(\widetilde\Delta_v)>0$
if and only if exactly one of the conditions $T_0>1$, $T_1<0$,
$T_2<0$, $T_3<0$ holds. The even sign change at the indices
$j$ and $k$ replaces $T_0$ by
$\theta_i+(\frac12-\theta_j)+(\frac12-\theta_k)=1-T_i$. So the
four conditions are the condition $T_0>1$ for the four lifts, and
(b) follows.

(c) Let $J$ be the set of indices $i$ for which $M_i$ is nonsplit, and
write $l_i$ for the labels of \S\ref{nm:class-vectors}. The number
\[
 N:=(p^2-1)(T_0-1)
 =(p+1)\sum_{i\notin J}l_i+(p-1)\sum_{i\in J}l_i-(p^2-1)
\]
is an integer, and \eqref{nm:a-and-theta} gives
\begin{equation}\label{nm:dictionary-key}
 p-1-(a_1+a_2+a_3)=\frac2{p+1}\Bigl(N-\sum_{i\in J}a_i\Bigr).
\end{equation}
Suppose that $T_0>1$, so that $N>0$. If $J$ is empty, the right side
is positive and the first alternative of \eqref{nm:bw-hypothesis}
holds. Suppose that $J$ is not empty. Both $N$ and
$\sum_{i\in J}a_i=\sum_{i\in J}(p+1-2l_i)$ are congruent to
$-2\sum_{i\in J}l_i$ modulo $p+1$. Moreover $\theta_i<\frac12$ for
$i\notin J$ and $T_0>1$ give
$\sum_{i\in J}\theta_i>1-\frac12(3-|J|)$, that is,
$\sum_{i\in J}(1-2\theta_i)<1$, so that
$0<\sum_{i\in J}a_i<p+1$. A positive integer congruent to
$\sum_{i\in J}a_i$ modulo $p+1$ is therefore at least
$\sum_{i\in J}a_i$. Thus the right side of \eqref{nm:dictionary-key} is
nonnegative. If it is positive, the first alternative of
\eqref{nm:bw-hypothesis} holds; if it is zero, the second alternative
holds, because $J$ is not empty.

Conversely, suppose that \eqref{nm:bw-hypothesis} holds. If $J$ is
empty, the left side of \eqref{nm:dictionary-key} is positive, so
$N>0$. If $J$ is not empty, the left side is nonnegative, so
$N\ge\sum_{i\in J}a_i>0$. In both cases $T_0>1$.

The first part of the last assertion follows from (b) and (c). The
lift that satisfies \eqref{nm:bw-hypothesis} has bad reduction at
$v$ by Theorem~\ref{nm:bw-theorem}(b), applied over a field $K$ as in
\S\ref{nm:bw-statement} that contains a square root of
$(-1)^{(p-1)/2}p$, together with Lemma~\ref{bg:ascent}(c); the class
vector of the lift over $K$ is the one described in
\S\ref{nm:class-vectors}.
\end{proof}

The existence of a place at which $f$ is indefinite uses the following
fact, which goes back to Fricke.

\begin{lemma}[Trace rigidity]\label{nm:trace-rigidity}
Let $F$ be a finite field or $F=\bC$, and let $(A_1,A_2,A_3)$ and
$(A_1',A_2',A_3')$ be triples of matrices in $\SL_2(F)$ with product $I$
and with the same traces $x_i=\tr A_i=\tr A_i'$. If
$x_1^2+x_2^2+x_3^2-x_1x_2x_3-4\ne0$, then the two triples are
simultaneously conjugate in $\GL_2(F)$.
\end{lemma}

\begin{proof}
The condition says that the pairs $(A_1,A_2)$ and $(A_1',A_2')$ are
absolutely irreducible \cite{Chen}*{Lemma~5.2.9}, and
$\tr A_3=\tr(A_1A_2)$. For a finite field the assertion is
\cite{Chen}*{Theorem~5.2.10}, and over $\bC$ it is
\cite{CullerShalen}*{Proposition~1.5.2}.
\end{proof}

\begin{lemma}[A place of indefinite type]\label{nm:indefinite-place}
Suppose that the ramification indices of $f$ are prime to $p$. Then there is
a place of $\Qbar$ above $p$ at which $f$ is indefinite.
\end{lemma}

\begin{proof}
Fix a lift $M$ and a place $v_0$ of $\Qbar$ above $p$, and let
$\widetilde\delta_i$, $\widetilde x_i$ and
$\widetilde\Delta:=\widetilde\Delta_{v_0}$ be the elements of
\S\ref{nm:class-vectors} for the place $v_0$. They lie in the
cyclotomic field $L:=\bQ(\widetilde\zeta_{p^2-1})$, and the orders of
the $\widetilde\delta_i$ exceed two. The matrices
\[
 \widetilde M_1=\begin{pmatrix}0&-1\\1&\widetilde x_1\end{pmatrix},
 \qquad
 \widetilde M_2=\begin{pmatrix}
 \widetilde\delta_2&
 \widetilde x_3-\widetilde x_1\widetilde\delta_2^{-1}\\
 0&\widetilde\delta_2^{-1}
 \end{pmatrix}
\]
lie in $\SL_2(\cO_L)$ and have
$\tr\widetilde M_1=\widetilde x_1$,
$\tr\widetilde M_2=\widetilde x_2$ and
$\tr(\widetilde M_1\widetilde M_2)=\widetilde x_3$. By
Lemma~\ref{nm:trace-rigidity}, applied over the residue field of $L$
at $v_0$, their reductions at $v_0$ are simultaneously conjugate to
$(M_1,M_2)$, so they generate a conjugate of $\SL_2(\bF_p)$.

The group $\Gamma=\langle\widetilde M_1,\widetilde M_2\rangle$ is
infinite. Otherwise its image in $\PGL_2(\bC)$ would be a
finite group with $\PSL_2(\bF_p)$ as a quotient, by reduction, which
Lemma~\ref{bg:finite-subgroups} excludes for $p>5$.

At every complex embedding of $L$, each $\widetilde x_i$ is real and
lies strictly between $-2$ and $2$, because $\widetilde\delta_i$ is a
root of unity different from $\pm1$. At such an embedding the pair
$(\widetilde M_1,\widetilde M_2)$ is conjugate, by
Lemma~\ref{nm:trace-rigidity} over $\bC$, to
\[
 A_0=\diag(\widetilde\delta_1,\widetilde\delta_1^{-1}),
 \qquad
 B_0=\begin{pmatrix}a&1\\b&\bar a\end{pmatrix},
 \quad
 a=\frac{\widetilde x_3-\widetilde\delta_1^{-1}\widetilde x_2}
          {\widetilde\delta_1-\widetilde\delta_1^{-1}},\quad
 b=\frac{\widetilde\Delta}{4-\widetilde x_1^2},
\]
where $\bar a$ is the complex conjugate of $a$. Direct multiplication
shows that $A_0$ and $B_0$ preserve the Hermitian form
$\diag(-b,1)$, which is definite if $\widetilde\Delta<0$ and
indefinite if $\widetilde\Delta>0$. If $\widetilde\Delta$ were negative
at every embedding, then after a fixed change of basis at each
embedding the matrices of $\Gamma$ would be uniformly bounded. But
$\operatorname{Mat}_2(\cO_L)$ is a discrete lattice in the product of
the archimedean matrix spaces, so its intersection with a bounded set
is finite, contradicting the infinitude of $\Gamma$. Hence there is an
embedding $\iota'\colon L\hookrightarrow\bC$ with
$\iota'(\widetilde\Delta)>0$.

Choose $\varsigma\in\Gal(\Qbar/\bQ)$ with
$\iota\circ\varsigma=\iota'$ on $L$, and let
$v:=v_0\circ\varsigma^{-1}$ be the place whose valuation ring is the
image under $\varsigma$ of the valuation ring of $v_0$. Then
$\varsigma$ induces an isomorphism of the residue field of $v_0$ onto
that of $v$, and the root of unity $\varsigma(\widetilde\delta_i)$
reduces at $v$ to the image of $\delta_i$. Hence
$\widetilde\Delta_v=\varsigma(\widetilde\Delta)$, and
$\iota(\widetilde\Delta_v)=\iota'(\widetilde\Delta)>0$: the cover $f$
is indefinite at $v$.
\end{proof}

\begin{proof}[Proof of Theorem~\ref{thm:nm-arithmetic}]
Suppose first that an ramification index of $f$ is divisible by $p$. A good
model has ramification indices prime to $p$
(Lemma~\ref{bg:good-model}(c)), so $f$ has bad reduction at every place
of $\Qbar$ above $p$.

Suppose now that the ramification indices of $f$ are prime to $p$. By
Lemma~\ref{nm:indefinite-place} there is a place $v$ at which $f$ is
indefinite, and by Lemma~\ref{nm:dictionary} there is a lift $f_M$ with
bad reduction at $v$. Then $f$ has bad reduction at $v$ by
Lemma~\ref{nm:lifts}.
\end{proof}

\begin{remark}\label{nm:definite-remark}
The proof gives the following more precise statement: a connected
three-point $\Gp$-cover with ramification indices prime to $p$ has bad
reduction at every place at which it is indefinite. The converse also
holds: such a cover has good reduction at every place at which
it is definite (Theorem~\ref{app:corrected}). The hypothesis $p>5$ is
used in Lemma~\ref{nm:indefinite-place}.
\end{remark}

\section{Generic inseparability criterion}\label{sec:generic}
\label{nm:generic-stable-reduction}

This section proves the generic inseparability criterion used in step (3) of
\S\ref{intro:strategy}: for a $\Gp$-cover of an elliptic curve
with good reduction over a discretely valued field whose residue field
$F$ is a function field, bad reduction of the cover
(Definition~\ref{bg:good-reduction}) forces the reduced
$j$-invariant to lie in $F^p$ (Proposition~\ref{nm:generic-criterion}).

\noindent\emph{Standing hypotheses.}
Throughout this section $p>7$ and $m=(p-1)/2$.
Let $k$ be a perfect field of characteristic $p$, let $F/k$ be a
one-variable function field, and let $K$ be a complete discretely
valued field of mixed characteristic with valuation ring $R$ and
residue field $F$. Let $E/K$ be an elliptic curve with good
reduction, with special fiber $E_0/F$, and assume that
$\bar j:=j(E_0)$ is transcendental over $k$. Let $Y\to E$ be a
geometrically connected $\Gp$-cover, branched only over the origin, with ramification index
$e$ there nontrivial and prime to $p$. The ramification points of
$Y$ form its marking (Definition~\ref{bg:cover}). Since $e>1$, tame
Riemann--Hurwitz gives $g(Y)\ge2$.

After a finite extension of $K$, the curve $Y$ marked by its
ramification points has a stable marked model, a stable marked
curve over the valuation ring with generic fiber $Y$
(\S\ref{bg:reduction}). It is unique, and its formation commutes
with further extension. The elliptic curve marked at its origin has
good reduction, with smooth model $\cE/R$.
By Lemma~\ref{bg:comparison}(a), the cover $Y\to E$ has good reduction
if and only if its marked source $Y$ has good (marked) reduction.
Thus the cover has bad reduction precisely when the geometric special
fiber of this stable marked model is not smooth. We use this
equivalence throughout the section.

\subsection{The stable target and its ramification invariants}
\label{gen:sec-tails}

In this subsection we recall the structure of the stable reduction of
a Galois cover whose group has order divisible by $p$ but not by
$p^2$, as is the case for $\Gp$. The theory is due to Raynaud
\cite{Raynaud1999}; we use it in the form given by Wewers
\cite{WewersBad}, Bouw and Wewers \cite{BouwWewers} and Obus
\cite{ObusVanishing}, who treats a base of arbitrary genus.

Let $K'/K$ be a finite Galois extension containing a $p$th root of
unity over which the stable marked model
$\cY$ of $Y$ with its ramification marking exists, and let $R'$ be its
valuation ring.
By uniqueness of the stable marked model, the group
$\Gp$ acts on $\cY$, and the quotient $\cX:=\cY/\Gp$ is a semistable
model of $E$ \cite{ObusVanishing}*{\S2.3}.

The sources cited above assume an algebraically closed residue field,
whereas the residue field of $R'$ is a finite, possibly inseparable,
extension of the imperfect field $F$. Let $\kappa$ be an algebraic
closure of the residue field of $R'$. There is a flat local
homomorphism from $R'$ to a local ring whose maximal ideal is generated
by a uniformizer $\pi'$ of $R'$ and whose residue field is $\kappa$
\cite{Stacks}*{03C3}.
Flatness implies that $\pi'$ is a nonzerodivisor in this extension,
and therefore its $\pi'$-adic completion $R''$ is a complete discrete valuation ring
with uniformizer $\pi'$ and residue field $\kappa$. Write
$\cY_{R''}:=\cY\otimes_{R'}R''$ and $\cX_{R''}:=\cX\otimes_{R'}R''$.
The curve $\cY_{R''}$ is the stable marked model of $Y$ over the
fraction field of $R''$, because stable marked curves remain stable
after base change, and $\cX_{R''}=\cY_{R''}/\Gp$, because the
formation of invariants under a finite group commutes with flat base
change. We apply the theory over $R''$ and write
\[
 \bar Y:=\cY\otimes_{R'}\kappa,\qquad
 \bar X:=\cX\otimes_{R'}\kappa
\]
for the special fibers of $\cY_{R''}$ and $\cX_{R''}$, the
\emph{stable marked source} and the \emph{stable target}; the inertia
groups, ramification invariants and differentials below are those of
$\cY_{R''}$.

The dual graph of $\bar X$ is a tree with a distinguished
\emph{original component} $\bar X_0$, which maps birationally to the
special fiber $(E_0)_\kappa$ of the smooth model $\cE$,
and rational
components attached to it.\footnote{Indeed, the minimal desingularization of
$\cX_{R''}$ replaces each singular point by a chain of projective lines
\cite{Liu}*{Corollary~10.3.25}, so its dual graph is a subdivision of
that of $\bar X$. It is a regular model of $E$, so it dominates the
smooth model, which is the minimal regular model, through a sequence of
blow-ups of closed points \cite{Liu}*{Theorems~9.2.2 and~9.3.21}, each of
which adds a leaf to the dual graph or subdivides an edge.}
A component of $\bar X$ is
\emph{terminal} if it meets the rest of $\bar X$ in exactly one node. A
\emph{tail} is a terminal component other than $\bar X_0$; it is
\emph{primitive} if it contains the
specialization of the branch point and \emph{new} otherwise. We write
$\bar X_v$ for the component of $\bar X$ at a vertex $v$ of the tree.
The stable target $\bar X$ is in general different from the quotient
$\bar Y/\Gp$; we will compare them in \S\ref{gen:sec-centralizer}.

Let $\bar Y_v$ be a component of $\bar Y$ above a component $\bar X_v$,
with decomposition group $D$. The inertia group of its generic point is
either trivial or a group $P\cong C_p$ of order $p$, because
$p^2\nmid|\Gp|$ \cite{ObusVanishing}*{Proposition~2.10}; we call
$\bar X_v$ \emph{\'etale} or a \emph{$p$-component} accordingly. The
group $D$ normalizes $P$, so $D\le N_{\Gp}(P)=P\rtimes C_m$, and $D/P$
has order prime to $p$.

For a $p$-component, the map $\bar Y_v\to\bar X_v$ has inseparable degree exactly $p$.
Indeed, both special fibers are reduced, as $\cY_{R''}$ is stable and
$\cX_{R''}$ is semistable (\cite{WewersBad}*{\S1.1},
\cite{ObusVanishing}*{\S2.3}), so $\pi'$ is a uniformizer at the generic
point of $\bar X_v$ and at the generic point of every component of
$\bar Y$ above it. Let $A$ be the local ring of $\cX_{R''}$ at the
generic point of $\bar X_v$, a discrete valuation ring, and let $B$ be
its integral closure in the function field of $\cY_{R''}$, a Galois
extension with group $\Gp$. The maximal ideals of $B$ are the generic
points of the components of $\bar Y$ above $\bar X_v$, one
$\Gp$-orbit with stabilizer $D$; their ramification indices are one
and their residue degrees have a common value $f$, so the degree
formula \cite{Stacks}*{09E8} gives $|\Gp|=[\Gp:D]f$, that is,
$f=|D|$. The residue extension $\kappa(\bar Y_v)/\kappa(\bar X_v)$ is
normal, and $D$ maps onto its automorphism group with kernel the
inertia group $P$ \cite{Stacks}*{09ED}; so its separable degree is
$|D/P|$ and its inseparable degree is $p$. Thus the map factors as
\begin{equation}\label{gen:factorization}
 \bar Y_v\stackrel{D/P}{\lra}\bar Y_v/D\lra\bar X_v,
\end{equation}
in which the first map is Galois with group $D/P$
and the second is purely inseparable of degree $p$. For an \'etale component
the same argument with $P=1$ shows that $\bar Y_v/D\to\bar X_v$ is
birational.

We recall the deformation datum of a $p$-component $\bar X_v$, which
is used to count tails \cite{WewersBad}*{\S1.3}. Let $\bar Y_v$ be a
component of $\bar Y$ above $\bar X_v$, with decomposition group $D$
and inertia group $P$, and let $C$ be the smooth projective curve
whose function field is the maximal separable subextension of
$\kappa(\bar Y_v)/\kappa(\bar X_v)$. Then $C\to\bar X_v$ is a Galois
cover with group $D/P$, of order prime to $p$, and $\kappa(\bar Y_v)$
is purely inseparable of degree $p$ over $\kappa(C)$; the curve $C$ is
the normalization of the corresponding component of the special fiber
of $\cY_{R''}/P$. The deformation datum uses the
factorization
$$\bar Y_v\lra C \stackrel{D/P}{\lra}\bar X_v,$$
as an inseparable degree $p$ map followed by a Galois $D/P$-cover.

The datum records how $P$ acts near the generic point of $\bar Y_v$.
Let $\cO$ be the local ring of
$\cY_{R''}/P$ at the generic point of $C$, a discrete
valuation ring with residue field $\kappa(C)$ in which $\pi'$ is a
uniformizer. By Kummer theory the function field of $\cY_{R''}$ is generated
over the fraction field of $\cO$ by an element $y$ with $y^p=u$; as
$\pi'$ remains a uniformizer at the generic point of $\bar Y_v$, we may
take $u\in\cO^\times$, and $u$ is then determined up to $p$th powers
of units and up to $u\mapsto u^i$ with $p\nmid i$. There are two
cases \cite{WewersBad}*{\S1.3.2}.
\begin{itemize}
\item \emph{Multiplicative case.} The residue $\bar u\in\kappa(C)$ is
not a $p$th power. Put $\omega:=d\bar u/\bar u$.
\item \emph{Additive case.} Otherwise, after a finite extension of
$K'$, the unit $u$ can be chosen of the form
$u=1+\varpi^pf$, where $\varpi\in R''$ is a power of $\pi'$ with
$0<\ord_{\pi'}(\varpi)<\ord_{\pi'}(p)/(p-1)$ and the residue
$\bar f\in\kappa(C)$ of $f\in\cO$ is not a $p$th power.\footnote{The inequalities for $\ord_{\pi'}(\varpi)$ implies that the additive case is impossible if $\ord_{\pi'}(p) = p-1$.} Put
$\omega:=d\bar f$.
\end{itemize}
In both cases $\omega$ is a nonzero meromorphic differential on $C$,
well defined up to a nonzero constant, so its orders of vanishing are
well defined.
The differential is an eigenvector for the group $D/P$:
$\gamma^*\omega=\chi(\gamma)\omega$ for $\gamma\in D/P$, where
$\chi\colon D/P\to\bF_p^\times$ is the character by which $D/P$ acts on
$P$ by conjugation \cite{WewersBad}*{\S1.3.2}; it is faithful, because
$P$ is its own centralizer in $\Gp$. For a closed point
$x\in\bar X_v$ and a point $z\in C$ above it, let $e_x$ be the order of
the stabilizer of $z$ in $D/P$ and put
\begin{equation}\label{eq_ramification_invariants}
 h_x=\ord_z(\omega)+1,\qquad \sigma_x=h_x/e_x.
\end{equation}
These quantities do not depend on the choice of $z$, because $D/P$
permutes the points above $x$ transitively and $\omega$ is an
eigenvector. The point $x$ is
\emph{critical} if $(e_x,h_x)\ne(1,1)$. Note that the quantity $\sigma_x$ is attached to a branch of the node, not just the point $x$.
For a tail $\bar X_b$, attached
to a $p$-component at the point $x_b$, put $\sigma_b:=\sigma_{x_b}$. When $\bar{X}_b$ is the primitive tail, which attaches to the special fiber at $x$, we write $\sigma_{\rm prim} := \sigma_x$ and call it the \emph{ramification invariant} of the primitive tail.

The terminology is justified by
Proposition~\ref{gen:stable-inputs}(b) below, which identifies it with
the unique positive upper jump of the inertia group at $x$. To recall
the convention, let $L/M$ be a finite totally ramified Galois
extension of complete discretely valued fields, with group $G$ and
valuation ring $\cO_L$. Its lower ramification groups are
\[
 G_i:=\{g\in G: gx\equiv x\pmod{\fm_L^{i+1}}
                 \text{ for every }x\in\cO_L\}\qquad(i\ge0).
\]
Thus $G=G_0\supseteq G_1\supseteq\cdots$. For real $s\ge0$, put
$G_s:=G_{\lceil s\rceil}$ and define Herbrand's function and its
inverse by
\[
 \varphi(s):=\int_0^s\frac{dt}{[G_0:G_t]},\qquad
 \psi:=\varphi^{-1}.
\]
The upper numbering is $G^u:=G_{\psi(u)}$, or equivalently
$G^{\varphi(s)}=G_s$. A lower or upper jump is a value at which the
corresponding filtration drops. See
\cite{SerreLocalFields}*{Chapter~IV, \S\S1,3}.

We will use the following inputs from stable reduction theory.

\begin{proposition}[Stable-reduction inputs]\label{gen:stable-inputs}
For the stable marked quotient of the $\PSL_2(\bF_p)$-cover of the
good elliptic base under consideration, the following hold over an
algebraically closed residue field.
\begin{enumerate}
\item Source nodes lie over target nodes. If $Y\to E$ has bad reduction,
every \'etale target component is terminal. A nonoriginal rational
terminal component has smaller $p$-power generic inertia than its
neighbor. A branch point of index $p^an$ with $p\nmid n$ specializes
to a $p^a$-component.
\item Critical points on a $p$-component are nodes or specializations
of branch points of index divisible by $p$. At the latter the datum is
multiplicative with a simple logarithmic pole. At a node between two
$p$-components, the two oriented invariants $\sigma$ are negatives of each other. At
a node directed from a $p$-component toward an \'etale component,
$\sigma$ is the unique positive upper jump of the local inertia.
\item On every $p$-component $\bar X_v$,
\[
 \sum_{x\ {\rm critical}}(\sigma_x-1)=2g(\bar X_v)-2.
\]
If the original component is inseparable, then
\[
 2g(E_0)-2+|B_{\rm wild}|
   =\sum_{\text{\'etale tails }b}(\sigma_b-1),
\]
where $B_{\rm wild}$ is the set of original branch values of index divisible
by $p$. Every new tail has $\sigma_b\ge1+1/m$.
\end{enumerate}
\end{proposition}

\begin{proof}[Source]
The construction of the stable target and part (a) are the
arbitrary-base statements of Obus
\cite{ObusVanishing}*{\S2.3, Proposition~2.13, Lemmas~2.15--2.16};
the rationality hypothesis in the terminal-component assertion is
essential to the proof of his Lemma~2.16. The global formula in (c)
is his Theorem~3.14 and the new-tail bound is his Lemma~4.2(i); we
establish inseparability of the original component before using the
formula. The local datum and its degree identity in (c) are Wewers's
Definition~1.5 and equation~(2) in \cite{WewersBad}, where the base
of the datum has arbitrary genus; his rules for smooth points in
\S1.3.2 apply once the order-$p$ inertia and the tame quotient are
supplied by Obus's Construction~3.4.

The assertions in (b) concerning nodes follow from Wewers's
compatibility theorem for deformation data. The tame stabilizer orders
on the two branches agree, while the corresponding signed conductors
are opposite \cite{WewersBad}*{Proposition~1.8}. Equivalently, the
oriented invariants satisfy $\sigma_e=-\sigma_{\bar e}$
\cite{WewersBad}*{Proposition~1.10(4)}. When $e$ is oriented from a
$p$-component toward an \'etale tail, $\sigma_e$ is the positive upper
ramification jump at the attaching point; see the paragraph following
Proposition~1.10 in loc.~cit. These results occur in \S1 of that paper,
where the base is an arbitrary smooth stably marked curve. Also see
\cite{ObusVanishing}*{Lemmas~3.6 and~3.11(i),(iii)} for the same local
comparison in the language of effective edge invariants.
\end{proof}

\begin{lemma}[Original component and tails]\label{gen:tails}
Assume that $Y\to E$ has bad reduction. Then the tails of
$\bar X$ are \'etale, and every other component, in particular
$\bar X_0$, is a $p$-component. There are exactly one primitive tail
$\bar X_{\rm prim}$ and $N$ new tails
$\bar X_b$, $b\in B_{\rm new}$, where $m = \frac{p-1}{2}$,
\[
 0\le N<m,\qquad
 \sigma_{\rm prim}=1-\frac Nm,\qquad
 \sigma_b=1+\frac1m\quad\text{for every new tail }b.
\]
For $b\in B_{\rm new}$ the inverse image $T_b\subset\bar Y$ of
$\bar X_b$ is irreducible, and the inertia group of a point of $T_b$
above the node of $\bar X_b$ is a Borel subgroup of $\Gp$, of order
$pm$.
\end{lemma}

\begin{proof}
If $\bar X$ had no component besides $\bar X_0$, it would have
no node, and neither would $\bar Y$, contrary to the bad reduction
assumption and Lemma~\ref{bg:comparison}(a). Thus there is a
tail, which by
Proposition~\ref{gen:stable-inputs}(a) is \'etale with a $p$-component
neighbor.

Suppose that $\bar X_0$ were \'etale. By
Proposition~\ref{gen:stable-inputs}(a) every \'etale component is
terminal, so the $p$-components are rational and form a subtree, to
which $\bar X_0$, the new tails and at most one primitive tail are
attached, each at one node. Sum the local degree identities of
Proposition~\ref{gen:stable-inputs}(c) over the tree. On each internal
edge the two oriented $\sigma$ terms cancel, so that edge contributes
$-2$ to the sum of $\sigma-1$. Let $\sigma_E$ be the value of
$\sigma$ at the point where $\bar X_0$ is attached. According as the
primitive tail is absent or present, the sum is
\[
 \sigma_E+\sum_{\rm new}(\sigma_b-1)=-1,\qquad
 \sigma_E+\sigma_{\rm prim}+\sum_{\rm new}(\sigma_b-1)=0.
\]
Both are impossible: $\sigma_E,\sigma_{\rm prim}>0$ by
Proposition~\ref{gen:stable-inputs}(b), and $\sigma_b-1\ge1/m$ by
Proposition~\ref{gen:stable-inputs}(c), so the left sides are positive.
Hence $\bar X_0$ is
a $p$-component. The origin of $E$ specializes to an \'etale component,
hence on the unique primitive tail; all other tails are new, and all
other components are $p$-components. Since $B_{\rm wild}$ is empty, the global
formula in Proposition~\ref{gen:stable-inputs}(c) now reads
\begin{equation}\label{gen:vc}
 \sum_{\rm new}(\sigma_b-1)=1-\sigma_{\rm prim}<1.
\end{equation}

Let $T$ be a component of $T_b$, for a new tail $\bar X_b$, with
decomposition group $D_T$. Its attaching inertia contains $P$. If $D_T\ne \Gp$, then
$p$ divides $|D_T|$, so Lemma~\ref{nm:proper-subgroup} places $D_T$
in a Borel subgroup and its Sylow $p$-subgroup $P$ is normal. The
tame quotient $T/P\to\bP^1$ is branched at most at the
attaching point, and a nontrivial connected tame cover with one branch
point is impossible by Riemann--Hurwitz. Thus $D_T=P$, so
$\sigma_b$ is an integer, greater than one by
Proposition~\ref{gen:stable-inputs}(c), and it contributes at least one to the left
side of \eqref{gen:vc}, a contradiction. Hence $D_T=\Gp$.

Thus $T_b$ is irreducible, $T_b\to\bar X_b$ is a connected new
$\Gp$-tail cover in the sense of \cite{BouwWewers}*{\S2}, and
$1<\sigma_b<2$ by Proposition~\ref{gen:stable-inputs}(c) and
\eqref{gen:vc}. By \cite{BouwWewers}*{Proposition~6.3}, which treats
$\SL_2(\bF_p)$ and $\PSL_2(\bF_p)$, such a tail cover has
$\sigma_b=(p+1)/(p-1)=1+1/m$ and an inertia group of order
$p(p-1)/2=pm$ above the node. This group has a normal Sylow
$p$-subgroup, so it lies in a Borel subgroup, and it has the order of
one. Substituting $\sigma_b=1+1/m$ in
\eqref{gen:vc} gives $N/m=1-\sigma_{\rm prim}$, which is the displayed
formula for $\sigma_{\rm prim}$, and $\sigma_{\rm prim}>0$ gives
$N<m$.
\end{proof}

\subsection{Automorphisms of the stable source}
\label{gen:sec-centralizer}

In this and the next subsection we assume that $Y\to E$ has bad reduction. The
aim is to descend the whole stable marked source to $F$, and the
obstruction to that descent is the group of its automorphisms that
commute with the $\Gp$-action and preserve the marking. We
show that this finite group has order prime to $p$.

\begin{lemma}[The centralizer on a new tail]\label{gen:new-tail-centralizer}
Let $\bP^1_\kappa$ be one of the new tails of
Lemma~\ref{gen:tails}, and let $T\to\bP^1_\kappa$ be its connected preimage in $\bar{Y}$.
The group of automorphisms of $T$ commuting
with the $\Gp$-action has order prime to $p$. We do \emph{not} require these automorphisms to be deck transformations.
\end{lemma}

\begin{proof} We will show that no such automorphism has order $p$.
The unique branch point has inertia a Borel subgroup $B$ of order
$pm$, so its fiber in $T$ is the $\Gp$-set $\Gp/B$ of size $p+1$. An
automorphism commuting with $\Gp$ preserves this fiber and permutes it
$\Gp$-equivariantly; since
\[
 \Aut_{\Gp}(\Gp/B)=N_{\Gp}(B)/B=1,
\]
it fixes all $p+1$ points individually.

The upper jump is $(m+1)/m$, so the lower jump is $m+1$, and the
different exponent at a point over the branch point is
\[
 \delta=(pm-1)+(m+1)(p-1).
\]
Riemann--Hurwitz with $|\Gp|=pm(p+1)$ gives
\[
 2g(T)-2=(p+1)(m-1),\qquad g(T)=m^2.
\]
In particular $g(T)>1$, so the automorphism group of $T$ is finite.
If a commuting automorphism had order $p$, its different exponent at
each of the $p+1$ fixed points found above would be at least
$2(p-1)$, because both its zeroth and first ramification groups have
order $p$. Riemann--Hurwitz for the quotient by this automorphism
would give
\[
 2g(T)-2\ge -2p+2(p+1)(p-1),\qquad g(T)\ge p(p-1),
\]
contradicting $g(T)=m^2$.
\end{proof}

Two quotient curves live in the special fiber: the stable target
$\bar X=(\cY/\Gp)\otimes_{R'}\kappa$ of
\S\ref{gen:sec-tails}, and the quotient $Q:=\bar Y/\Gp$ of the
geometric stable marked source by $\Gp$, with its reduced structure. There is a canonical map $Q\to\bar X$.

\begin{lemma}[The quotient of the stable source]\label{gen:quotient-comparison}
The map $Q\to\bar X$ is a finite universal homeomorphism, and on the
normalization of a component it is purely inseparable of degree $p$
over a $p$-component and an isomorphism over an \'etale component.
Consequently $Q$ and $\bar X$ have the same component graph $\Gamma_Q$: a rooted
tree with one root vertex of genus one, rational vertices otherwise, and
$N+1<p$ nonroot terminal vertices.
\end{lemma}

\begin{proof}
On an affine chart $\Spec A$ of $\cY_{R''}$, with $\pi'$ the uniformizer of $R''$, the map
\[
 A^{\Gp}/\pi'A^{\Gp}\longrightarrow(A/\pi'A)^{\Gp}
\]
is injective: if $\pi'a$ is invariant then so is $a$.
For an invariant
residue class $v$, lift it to $a\in A$ and reduce the invariant
polynomial $\prod_{g\in \Gp}(S-ga)$; the result is $(S-v)^{|\Gp|}$.
Writing $|\Gp|=pn$ with $p\nmid n$, this equals $(S^p-v^p)^n$, whose
coefficient of $S^{|\Gp|-p}$ is $-nv^p$; hence $v^p\in A^{\Gp}/\pi'A^{\Gp}$. Thus
$Q\to\bar X$ is finite and a universal homeomorphism, so the two curves
have the same component graph. The component of $Q$ over a vertex $v$
is the image of a component $\bar Y_v$ of $\bar Y$, so its normalization
is that of $\bar Y_v/D$, and the factorization \eqref{gen:factorization}
gives the degrees on normalizations.
Over the algebraically closed field $\kappa$, a purely inseparable map
of degree $p$ between smooth curves is Frobenius followed by an
isomorphism, so the normalized components of $Q$ have the same genera
as those of $\bar X$, and the tree is as described in
\S\ref{gen:sec-tails} and Lemma~\ref{gen:tails}.
\end{proof}

\begin{proposition}[Tameness of the automorphism group of the stable source]
\label{gen:source-centralizer}
Suppose that $Y\to E$ has bad reduction. The
group of automorphisms of $\bar Y$ that preserve the unordered ramification
marking and commute with the $\Gp$-action has order prime to
$p$.
\end{proposition}

\begin{proof}
Again it suffices to show that there are no such automorphisms of order $p$. Suppose $\alpha$ were an automorphism of order $p$. It acts on
the quotient $Q$ and on its component tree $\Gamma_Q$ (Lemma~\ref{gen:quotient-comparison}).
The root is fixed because
it is the unique vertex of positive genus.  Since there are fewer than $p$
nonroot terminal vertices, each of them is fixed, and hence $\alpha$ fixes $\Gamma_Q$ pointwise.

Over a vertex $v$ of $Q$, let $D\le \Gp$ be the decomposition group of a
source component. The group $D$ always contains a Sylow $p$-subgroup of $\Gp$: if $v$ is a $p$-component, $D$ contains the generic inertia of order $p$; if $v$ is a tail, $D$ contains the inertia group of the attaching node.
The set of source components over $v$ is
the $\Gp$-set $\Gp/D$, whose $\Gp$-equivariant automorphism group is
$N_{\Gp}(D)/D$, which has order prime to $p$ because $p^2\nmid|\Gp|$,
and hence $\alpha$ fixes every individual source component.

On the normalization of the root component of $Q$, a smooth curve of
genus one, $\alpha$ fixes an attaching point, since there are $N+1<p$ tails.
An automorphism of a smooth genus-one curve fixing a point has order
dividing $4$ or $6$ \cite{SilvermanAEC}*{Theorem~III.10.1}, hence prime to $p>3$,
so $\alpha$ is trivial on the root
normalization. Every rational nonterminal component of $Q$ has at
least two distinct attaching points, which must be fixed as $\alpha$ fixes the tree $\Gamma_Q$.
Since a nontrivial element of order $p$ in $\PGL_2(\kappa)$ has exactly one
fixed point, $\alpha$ is trivial on these normalizations too.

Above any of these vertices, the map from a source component
normalization to its quotient normalization in $Q$ has Galois group
$D/P$, of order prime to $p$. The restriction of $\alpha$ acts
trivially on the quotient, so it is a deck transformation of order
dividing $p$, hence the identity.

The primitive-tail component of $Q$ has two distinct distinguished
points, the attaching point and the marked branch point. Both are
fixed, so $\alpha$ is the identity on its rational normalization, and
its restriction to a primitive source component is a deck
transformation of that component commuting with $D$, that is, an
element of the center $Z(D)$. If $D=\Gp$ this center is trivial. If $D$
is a proper subgroup, then $D=P\rtimes C_d$ inside a Borel subgroup; the
nontrivial tame branch inertia forces $d>1$, and the faithful
conjugation action of $C_d$ on $P$ then gives $Z(D)=1$. Finally,
Lemma~\ref{gen:new-tail-centralizer} makes $\alpha$ trivial on every
new-tail source component.

We have shown that $\alpha$ is the identity on every component normalization in the source, and hence $\alpha = 1$, contradicting our assumption $|\alpha| = p$.
\end{proof}

\subsection{Inseparable descent of the stable marked source $\bar{Y}$}
\label{gen:sec-descent}

Proposition~\ref{gen:source-centralizer} concerns the curve $\bar Y$
over $\kappa$. Our goal is to show that if $Y\to E$ has bad reduction,
then $\bar j\in F^p$. For this we need a model of
$\bar Y$ over $F$ itself: the residue field of $R'$ can be inseparable
over $F$, and the property $\bar j\in F^p$ is not preserved by
inseparable extensions. This model is supplied by the arithmetic
valuative criterion of Bresciani and Vistoli, applied to a moduli stack
of stable marked curves with $\Gp$-action.

An algebraic stack with finite inertia is \emph{tame} if its geometric
stabilizers are linearly reductive. The stabilizers of a
Deligne--Mumford stack are unramified, hence \'etale over a field, and
a finite \'etale group scheme is linearly
reductive exactly when its order is invertible in the base field.
For the stable marked source this order is shown to be prime-to-$p$ by
Proposition~\ref{gen:source-centralizer}.

\begin{proposition}[Arithmetic valuative criterion {\cite{BrescianiVistoli}}]
\label{gen:BV-import}
Let $\cZ\to\Spec R$ be a proper tame algebraic stack, where
$R$ is a discrete valuation ring with fraction field $K$ and residue
field $F$. If $\cZ$ has a $K$-object, then its special fiber
has an $F$-object.
\end{proposition}

\begin{proof}[Source]
This is \cite{BrescianiVistoli}*{Theorem~3.1 and Corollary~3.2}: the
given $K$-object extends to a representable morphism from a root stack
of $\Spec R$ to $\cZ$, and the special fiber of a root stack has an
$F$-object.
\end{proof}

\begin{lemma}[Descent of the stable source]\label{gen:source-descent}
If $Y\to E$ has bad reduction, there is a stable marked curve $\bar Y_F$ over $F$,
with a $\Gp$-action, whose base change to $\kappa$ is
isomorphic to $\bar Y$ together with its $\Gp$-action and its unordered
ramification marking.
\end{lemma}

\begin{proof}
Let $g=g(Y)\ge2$ and let $r$ be the number of points of $Y$ over the
branch value. Consider the proper Deligne--Mumford stack
\[
 \cU=[\overline{\cM}_{g,r}/\mathfrak S_r]
\]
of stable curves of genus $g$ with an unordered marking of $r$ points.
The ramification points of $Y$ form a reduced divisor of degree $r$
defined over $K$, while the individual points are in general defined
only over an extension; this is why the marking is unordered. The
marked curve $Y$ thus defines a $K$-object of $\cU$.
As $\ol{\cM}_{g,r}$ is proper and Deligne--Mumford \cite{Knudsen}*{Theorem~2.7, p.~179}, the quotient $\cU$ is as well.

Let $\cU^{G_p}$ be the stack of objects of $\cU$ equipped
with a homomorphism from the constant abstract group $\Gp$ to their
automorphism group. In \cite{Romagny}, $\cU^{G_p}$ is the stack of $G_p$-fixed points, where $G_p$ acts trivially on $\cU$.
By \cite{Romagny}*{Theorem~3.3}, the forgetful map
$\cU^{G_p}\to\cU$ is representable and finite, so $\cU^{G_p}$ is proper and Deligne--Mumford.
For a curve $C$ with $G_p$-action, the automorphism group of $C$ in $\cU^{G_p}$
is the centralizer of that action in the automorphism group of the unordered
marked curve.

Let $\cU^{G_p}_R$ be the base change of $\cU^{G_p}$ to $R$, and let $\cZ$ be the
schematic closure of the
image of the map $\Spec K\to\cU^{G_p}_R$ defined by $Y$ with its
$\Gp$-action \cite{Stacks}*{0CPU}. It is proper over
$R$, and we claim that each $\kappa$-point of $\cZ$ represents
$\bar Y$ as a stable marked curve with $G_p$-action. Let $R'$ and $\cY$ be as in
\S\ref{gen:sec-tails}. The map
$\Spec R'\to\cU^{G_p}_R$ given by $\cY$ is proper, because $R'/R$ is finite and
$\cU^{G_p}_R$ is separated, so its image is closed and contains the
generic point, hence contains $\cZ$. Conversely the map factors
through $\cZ$: the pullback of the ideal defining $\cZ$
vanishes generically on the integral scheme $\Spec R'$, hence
everywhere. The image is thus the underlying space of $\cZ$.
The scheme $\Spec R'$ has one closed point, but its residue field $k'$
can have several $F$-embeddings into $\kappa$, and the $\kappa$-points of $\cZ$ are the curves $\cY\otimes_{R',\tau}\kappa$
for these embeddings $\tau$. They are all isomorphic to $\bar Y$, with
marking and $\Gp$-action. Indeed, $k'/F$ is normal and two embeddings
differ by an $F$-automorphism of $k'$, which is induced by an element
$s$ of $\Gal(K'/K)$; and $\cY$ is isomorphic to its conjugate by $s$,
compatibly with the marking and the action, because the stable marked
model is unique.

The stabilizer of a stable marked curve is finite unramified, hence
finite \'etale over an algebraically closed field, so
Proposition~\ref{gen:source-centralizer} makes every
stabilizer of $\cZ$ over the special fiber linearly reductive, and in
characteristic zero every finite \'etale group scheme is. Thus $\cZ$ is
tame, even though the ambient stack $\cU^{G_p}$ need not be, and
Proposition~\ref{gen:BV-import} supplies an $F$-object $\bar Y_F$ of
its special fiber. The base change of $\bar Y_F$ to $\kappa$ gives a $\kappa$-point of $\cZ$,
hence it is isomorphic to $\bar Y$ with
its marking and $\Gp$-action.
\end{proof}

\subsection{Bad reduction implies inseparability of \texorpdfstring{$\bar j$}{j}}
\label{gen:sec-frobenius}
We maintain the standing hypotheses and notation of \S\ref{sec:generic}.

\begin{proposition}[Generic inseparability criterion]
\label{nm:generic-criterion}
If $Y\to E$ has bad reduction, then $\bar j\in F^p$.
\end{proposition}

\begin{proof}
Lemma~\ref{gen:source-descent} applies. Let $\bar Y_F$ be the model
with $G_p$-action over $F$ that it supplies.
After passing to a finite separable extension $F'/F$, we may assume that every irreducible
component of $\bar Y_{F'}$ is geometrically irreducible
\cite{Stacks}*{054R}, every node is rational, and the two branches at
every node are rational  \cite{Stacks}*{0C4D, 0CBU}.
Fix an $F$-embedding of $F'$ into $\kappa$, so that the base change of
$\bar Y_{F'}$ to $\kappa$ is $\bar Y$
(Lemma~\ref{gen:source-descent}).
Let $W/F'$ be the normalization of the component of $\bar Y_{F'}$ corresponding to a component of $\bar Y$ above
$\bar X_0$.
The curve $W$ is smooth
over $F'$, because $\bar Y_{F'}$ is smooth away from its nodes and its
completed local ring at a node is $F'[[x,y]]/(xy)$.

Let $D\le\Gp$ be the decomposition group of this component; it acts on
$W$. The inertia group $P$ of the generic point of $W$ has order $p$,
because $\bar X_0$ is a $p$-component (Lemma~\ref{gen:tails}). It acts
trivially on $W$, so $D$ acts on $W$ through $D/P$, a group of order
prime to $p$. The quotient $W/D$ is a smooth proper curve over $F'$:
the formation of invariants commutes with the flat base change to
$\kappa$, where $W_\kappa/D$ is a quotient of a smooth curve and hence
smooth, so $W/D$ is geometrically normal. By the factorization
\eqref{gen:factorization}, the map $(W/D)_\kappa\to\bar X_0$ is purely
inseparable of degree $p$, and $\bar X_0$ is isomorphic to
$(E_0)_\kappa$. Over the algebraically closed field $\kappa$, a purely
inseparable subfield of degree $p$ of a one-variable function field is
its subfield of $p$th powers, so this map is Frobenius followed by an
isomorphism. Therefore $W/D$ has genus one and
\[
 \bar j=j\bigl(\Jac(W/D)\bigr)^p.
\]
Both sides of this identity lie in $F'$, so it holds in $F'$. An
element of $F$ cannot acquire a $p$th root in a separable extension,
so $\bar j\in F^p$.
\end{proof}

\begin{corollary}[Separability over $k(\bar j)$ gives good reduction]
\label{gen:generic-good}
If $F/k(\bar j)$ is finite separable, then $Y\to E$ has good
reduction.
\end{corollary}

\begin{proof}
The transcendental element $\bar j$ has no $p$th root in $k(\bar j)$
and cannot acquire one in the separable extension $F$.
Proposition~\ref{nm:generic-criterion} therefore excludes bad
reduction.
\end{proof}

\section{Extending good reduction to the cusp}
\label{nm:cusp-degree}

Theorem~\ref{thm:nm-arithmetic} concerns the normalized cuspidal
cover at a cusp, whereas Proposition~\ref{nm:generic-criterion}
concerns the cover of a varying elliptic curve. This section connects
the two and proves Proposition~\ref{nm:bad-cusp}: if the $j$-map of $X$
has good reduction at a place $v$ above $p$, then the normalization of an irreducible cuspidal cover also has good reduction at $v$.

\subsection{Cusps and cuspidal covers}\label{nm:cuspidal-notation}

In this subsection $G$ is any finite group, $X$ is a geometric
connected component of $\Mabs{G}$, and $c$ is a cusp of $X$, that is,
a point of $\overline X\setminus X$; from
\S\ref{nm:sec-modular} on we return to considering $\Gp$.
Topological notions refer to the fixed embedding $\iota\colon\Qbar\hookrightarrow\bC$ of
\S\ref{nm:class-vectors}.

First we explain how to describe cusps via a generating pair of $G$ (cf. \cite{Chen}*{\S4.8}).
Choose a point $t$ of $X$ close to $c$ and let $E_t$ be the elliptic
curve it parametrizes. As $t$ tends to $c$, the curve $E_t$
degenerates to a 1-gon (a nodal rational curve), and a simple closed curve on
the torus $E_t$, called the \emph{vanishing cycle}, is pinched to form the node
\cite{ACGII}*{Chapter~X, \S9}.
Let $(\alpha,\beta)$ be an oriented basis of the free group
$\pi_1(E_t^\circ)$ whose first element is the class of the vanishing
cycle, let $\varphi\colon\pi_1(E_t^\circ)\to G$ be the monodromy
homomorphism of the cover parametrized by $t$, and put
\[
 a:=\varphi(\alpha),\qquad b:=\varphi(\beta).
\]
We call $(a,b)$ a \emph{Nielsen pair representing $c$}. It depends on
the choices made: the point $t$, the base point, the orientation of the
vanishing cycle, the second basis element $\beta$ and, since $X$
parametrizes absolute structures, the identification of the Galois
group with $G$. A different choice replaces $(a,b)$ by a pair obtained
from it by a sequence of the following operations: simultaneous
conjugation; an automorphism of $G$; simultaneous inversion
$(a,b)\mapsto(a^{-1},b^{-1})$, which corresponds to reversing the
vanishing cycle; and the Dehn twist $(a,b)\mapsto(a,ab)$ about the
vanishing cycle or its inverse, which corresponds to replacing $\beta$
by $\alpha^{\pm1}\beta$. The twist is the monodromy of a loop around
$j=\infty$ in the $j$-line; a loop around $c$ in $X$ acts by its $w$th
power, where $w$ is the width of $c$, equivalently the ramification index of the $j$-map $j_X : X\to\bA^1_j$ at $c$ (\S\ref{nm:full-cusp}).
Put
\begin{equation}\label{eq_heHa}
 h:=\ord(a),\qquad e:=\ord([a,b]),\qquad
 H_a:=\langle a,\,ba^{-1}b^{-1}\rangle
     =\langle a,\,[a,b]\rangle.
\end{equation}
These four operations preserve $h$ and $e$, and they replace $H_a$ by
its image under an automorphism of $G$; for the Dehn twist this holds
because $[a,ab]=a[a,b]a^{-1}$. We call $h$ the \emph{pinched order} of
$c$; the cusp width $w$ always divides the pinched order $h$.
The cusps of $X$ correspond in this way to the cycles of
the Dehn twist on the absolute sheets of $X$, which is the
description used in \S\ref{nm:full-cusp}.

Let $\widetilde E_c$ be the projective line with the three marked
points $0_c:=0$, $O_c:=1$ and $\infty_c:=\infty$, and let $E_c$ be
the nodal rational curve obtained from $\widetilde E_c$ by
identifying $0_c$ with $\infty_c$; thus $\widetilde E_c\to E_c$ is
the normalization, $0_c$ and $\infty_c$ are the two branches of the
node, and $O_c$ is the origin. We call $E_c$ the \emph{cuspidal
curve}. Let
\[
 \widetilde Y_c\longrightarrow\widetilde E_c
\]
be the $G$-cover, \'etale outside the three marked points, whose
local monodromies are
\begin{equation}\label{nm:cusp-triple}
 \bigl(a,\ ba^{-1}b^{-1},\ [a,b]^{-1}\bigr)
 \quad\text{at}\quad(0_c,\infty_c,O_c);
\end{equation}
their product is one and their orders are $(h,h,e)$. By Riemann's
existence theorem $\widetilde Y_c$ exists and is unique up to
isomorphism once standard generators of
$\pi_1(\widetilde E_c\setminus\{0_c,O_c,\infty_c\})$ are fixed. We
call it the \emph{normalized cuspidal cover}. Its connected components
are $H_a$-Galois covers; we fix one component and write
\[
 \nu_c\colon\widetilde Y_c^{1}\longrightarrow\widetilde E_c
\]
for it, the \emph{connected normalized cuspidal cover}. Its group is
$H_a$, and $\widetilde Y_c^{1}=\widetilde Y_c$ exactly when
$H_a=G$. Finally, the fibers of $\widetilde Y_c$ over $0_c$ and
$\infty_c$ are the $G$-sets $G/\langle a\rangle$ and
$G/\langle ba^{-1}b^{-1}\rangle$, and right multiplication by
$b^{-1}$ is a $G$-equivariant bijection $\gamma_b$ between them. Let
$Y_c$ be the nodal curve obtained from $\widetilde Y_c$ by identifying
each point $y$ over $0_c$ with $\gamma_b(y)$; the $G$-action descends,
and $\widetilde Y_c\to\widetilde E_c$ descends to a finite
$G$-equivariant map
\[
 Y_c\longrightarrow E_c,
\]
the \emph{cuspidal cover}. It is a connected nodal curve with
$|G|/h$ nodes, and $\widetilde Y_c\to Y_c$ is its normalization.
Lemma~\ref{nm:cuspidal-degeneration} shows that $Y_c\to E_c$ is the
degeneration of the one-point covers of the elliptic curves near
$c$; \cite{Chen}*{\S4} describes the cuspidal objects of the
compactified moduli stack in the same way.
In particular, as $a,b$ run over all generating pairs of $G$, the associated $\wt{Y}_c\to\wt{E}_c$ and $Y_c\to E_c$ defined using \eqref{nm:cusp-triple} give all the cusps of $M(G)^\abs$.

\begin{proposition}[Good reduction of the $j$-map descends to the cusp]
\label{nm:bad-cusp}
Let $p>7$, let $X$ be a non-Markoff component of $\Mabs{\Gp}$, let
$c$ be a cusp of $X$ whose pinched order $h$ is greater than two and
prime to $p$ and whose normalized cuspidal cover
$\nu_c: \widetilde Y_c\to\widetilde E_c$ is connected, that is, $H_a=\Gp$. Let
$v$ be a place of $\Qbar$ above $p$. If the modular curve
$j_X\colon\overline X\to\bP^1_j$ has good reduction at $v$, then
$\nu_c$ has good reduction at $v$.
\end{proposition}

Equivalently, bad reduction of $\nu_c$ at $v$ forces bad reduction of
$j_X$ at $v$. The proof occupies the rest of the section. Assuming good
reduction of $j_X$ at $v$, we
construct in \S\ref{nm:sec-modular} a modular curve $V$ over $X$
carrying an elliptic curve and a $\Gp$-cover of it,
together with a good model over a valuation ring above
$v$, and we extend the elliptic curve across the cusp $c$ by its
Weierstrass equation. In \S\ref{nm:sec-degeneration} we
prove that in characteristic zero the cover degenerates at $c$
to the cuspidal cover $Y_c\to E_c$.
In \S\ref{nm:sec-proof-bad-cusp} we prove good reduction of the
cover at the generic point of the special fiber and extend it across
the cuspidal fiber by purity, producing a good model of
$\nu_c$ at $v$.

\subsection{A modular family}\label{nm:sec-modular}

Suppose that $j_X$ has good reduction at $v$. Let $K_1$ be a finite
extension of $\bQ_p$ inside the completion of
$\Qbar$ at $v$, large enough for the objects below to
be defined and containing a square root of $-3$; we extend $K_1$
finitely whenever necessary. Let $R_1$ be
its valuation ring, $\pi$ a uniformizer and $k$ the residue field.

Recall that $U_j=\bP^1_j\setminus\{0,1728,\infty\}$, and let $E_j\to U_j$
be the elliptic curve
\[
 E_j\colon\quad y^2=x^3-3\vartheta x-2\vartheta,\qquad
 \vartheta:=\frac{j}{j-1728}.
\]
Let $X^{\circ}:=j_X^{-1}(U_j)$. The restriction
$j_X\colon X^{\circ}\to U_j$ is finite \'etale of degree $d_X$
\cite{Chen}*{Theorem~2.5.2(2)}.

Let $T_1\to U_j$ be the scheme of $\Gp$-structures on $E_j$ whose
absolute class lies in $X^{\circ}$; it is finite \'etale over $U_j$
\cite{Chen}*{Theorem~2.5.2(1),(3),(4)} and maps onto $X^{\circ}$.
Over a point of $X^{\circ}$ its geometric points are the inner
structures in the given absolute class, at most two since $\Out(\Gp)$
has order two, each with the $\Gp$-structures on the curve $E_j$ itself
in that inner class, at most two since $\Aut(E_j,O)=\{\pm1\}$. So
$T_1\to X^{\circ}$ has fibers of at most four points. Let $V$ be a
connected component of $T_1$ and let $j_V\colon V\to U_j$ be its
$j$-map. The map $V\to X^{\circ}$ is finite \'etale and
$X^{\circ}$ is connected, so $V$ maps onto $X^{\circ}$. Put
$E_V:=E_j\times_{U_j}V$, with origin
section $O$. Since $Z(\Gp)=1$, the universal $\Gp$-structure on $E_V$ is
represented by a connected $\Gp$-torsor over $E_V^\circ$
\cite{Chen}*{Remark~2.5.1}; let
\[
 \cY_V\longrightarrow E_V
\]
be the normalization of $E_V$ in this torsor. Its fiber over a point
$t$ of $V$ is the one-point $\Gp$-cover of $E_{j(t)}$ whose absolute class is the image of $t$ in $X^{\circ}$.

Let $\overline V$ be the smooth compactification of $V$ and
$j_{\overline V}\colon\overline V\to\bP^1_j$ the extension of
$j_V$, a modular curve in the sense of \S\ref{bg:reduction}, and extend
$K_1$ so that $V$ is geometrically connected and its cusps are defined
over $K_1$. The map $\overline V\to\overline X$ is finite of degree at
most four, hence less than $p$, and it is \'etale over $X^\circ$, so it
is branched only above the marked points of $\overline X$, the points
over $0$, $1728$ and $\infty$. By hypothesis $j_X$ has a good model at
$v$, so by Lemma~\ref{bg:ascent}(b), after a further finite extension of
$K_1$, the $j$-map $j_{\overline V}=j_X\circ(\overline V\to\overline X)$
has a good model
\[
 j_{\overline{\cV}}\colon
 \overline{\cV}\longrightarrow\bP^1_{j,R_1},
\]
finite and \'etale over the complement of the sections $0$, $1728$ and
$\infty$, which are disjoint because $p\ge5$. By
Lemma~\ref{bg:good-model}, $\overline{\cV}$ is a smooth proper
$R_1$-curve, the reduced $j$-map $\overline{\cV}_k\to\bP^1_{j,k}$ is finite
and separable, and the ramification indices of $j_{\overline V}$ are
prime to $p$. The special fiber $\overline{\cV}_k$ is connected by
Zariski's connectedness principle \cite{Liu}*{\S5.3.2}, hence irreducible.

Let $X$ be a non-Markoff component of $M(G_p)^\abs$, and let $c$ be a cusp of $X$ of the type considered in Proposition \ref{nm:bad-cusp}. Thus, it is given by a generating pair $(a,b)$ with pinched order $h = \ord(a) > 2$ and $p\nmid h$. Let $c'\in\overline V(K_1)$ be a cusp above $c$,
let $\bar c'\in\overline{\cV}_k$ be the closed point
of the section of $\overline{\cV}$ through $c'$, and let
$\ell$ be the ramification index of $j_{\overline V}$ at $c'$, which is
prime to $p$. Moreover $h$ divides
$\ell$: the map $j_{\overline V}$ factors through $j_X$, whose
ramification index at $c$ is $h$ by Lemma~\ref{nm:width}.

The completed local ring of $\overline{\cV}$ at $\bar c'$ is
$R_1[[s]]$ for a coordinate $s$ with $c'=\{s=0\}$, and the completed
local ring of $\bP^1_{j,R_1}$
at the image of $\bar c'$ is
$R_1[[j^{-1}]]$. The map $j_{\overline{\cV}}$ is finite, \'etale
outside the sections $0$, $1728$ and $\infty$, and tamely ramified
along them. By Abhyankar's lemma \cite{Stacks}*{0EYG}, \'etale locally
on $\bP^1_{j,R_1}$ near the image of $\bar c'$ the cover
$\overline{\cV}$ is a disjoint union of Kummer covers $z^{e}=j^{-1}$.
Hence $c'$ is the only cusp that specializes to $\bar c'$, the ideal
generated by $j^{-1}$ in $R_1[[s]]$ is a power of $(s)$, and the
exponent is $\ell$ by comparison on the generic fiber:
\begin{equation}\label{nm:j-and-s}
 j^{-1}=s^{\ell}\epsilon_0,\qquad\epsilon_0\in R_1[[s]]^{\times}.
\end{equation}

We extend $E_V$ across the cusps by the Weierstrass equation of $E_j$.
Its coefficient
\begin{equation}\label{nm:weierstrass-at-infinity}
 \vartheta=\frac{1}{1-1728j^{-1}}
\end{equation}
is regular away from $j=1728$, so the same equation defines a flat
projective curve over the complement of the sections $0$ and $1728$ in
$\bP^1_{j,R_1}$. Let
$\cE_{\overline{\cV}}$ be the pullback of this curve to the preimage of
that complement in $\overline{\cV}$, an open subscheme which contains
$V$ and the sections through the cusps; the restriction of
$\cE_{\overline{\cV}}$ to $V$ is $E_V$. Let
\[
 \cE_s\longrightarrow\Spec R_1[[s]]
\]
be the base change of $\cE_{\overline{\cV}}$ to the completed local
ring of $\overline{\cV}$ at $\bar c'$. By \eqref{nm:j-and-s} it is the
curve defined by the Weierstrass equation of $E_j$ with
$j^{-1}=s^{\ell}\epsilon_0$, and its base change to $K_1((s))$ is the
elliptic curve $E_V\times_V\Spec K_1((s))$. Part (d) of the following
lemma is an explicit form of
\cite{DeligneRapoport}*{VII, Th\'eor\`eme~2.1}, which describes the
compactified moduli stack of elliptic curves near infinity by the Tate
curve.

\begin{lemma}[The Weierstrass family at the cusp]
\label{nm:weierstrass-infinity}
Let $O$ be the origin of $\cE_s$, the point at infinity of the
Weierstrass equation.
\begin{enumerate}
\item The curve $\cE_s$ is flat and projective over $R_1[[s]]$, and it
is smooth outside the node $(x,y)=(-1,0)$ of its fiber over $s=0$.
\item The fiber of $\cE_s$ over $s=0$ is the nodal cubic
$y^2=(x+1)^2(x-2)$ over $R_1$. The function $m:=y/(x+1)$ identifies
the normalization of this cubic with $\bP^1_{R_1}$. The two points over
the node are $m=\pm\sqrt{-3}$, the origin is $m=\infty$, and the
coordinate $(m-\sqrt{-3})/(m+\sqrt{-3})$ carries these three sections
to $0$, $\infty$ and $1$. Thus this fiber is the $1$-gon over
$R_1$, with origin $O_c$; we denote it by $E_c\otimes R_1$.
\item The completed local ring of $\cE_s$ at the node of its closed
fiber is isomorphic, as an $R_1[[s]]$-algebra, to
$R_1[[s,x',y']]/(x'y'-s^{\ell})$.
\item Let $\Tate(q)$ be the Tate curve
$$y^2+xy=x^3+a_4(q)x+a_6(q)$$
over $\bZ[[q]]$, whose $j$-invariant is
$j(q)=q^{-1}+744+\cdots$
\cite{SilvermanAdvanced}*{Theorem~V.1.1}. The substitution
$j^{-1}=1/j(q)$ identifies $\bZ[[j^{-1}]]$ with $\bZ[[q]]$, and over
$\bZ[1/6,\sqrt{-3}][[q]]$ it identifies the curve defined by the
Weierstrass equation of $E_j$ with $\Tate(q)$. The isomorphism is an
affine change of the coordinates $(x,y)$ whose coefficients converge on
a disc around $q=0$ in $\bC$, and at $q=0$ it carries the coordinate
of (b) to the parameter $u$ of the uniformization $x=u/(1-u)^2$,
$y=u^2/(1-u)^3$ of the fiber $y^2+xy=x^3$ of $\Tate(q)$.
\end{enumerate}
\end{lemma}

\begin{proof}
In (a)--(c) we have $j^{-1}=s^{\ell}\epsilon_0$ in
\eqref{nm:weierstrass-at-infinity}. A Weierstrass equation defines a
flat projective curve which is smooth along its point at infinity.
Since $6$ is a unit, the locus where $\cE_s$ is not smooth over
$R_1[[s]]$ is defined by $y=0$, $x^2=\vartheta$ and
$x^3-3\vartheta x-2\vartheta=0$. On this locus
$x^3-3\vartheta x-2\vartheta=-2\vartheta(x+1)$, so $x=-1$ and $\vartheta=1$, that
is, $s^{\ell}=0$, because $1728$ and $\epsilon_0$ are units. This
proves (a). At $s=0$ the equation is $y^2=(x+1)^2(x-2)$, and
$x=m^2+2$, $y=m(m^2+3)$ define a finite birational morphism from
$\bP^1_{R_1}$ onto this cubic, which is an isomorphism outside the
node. The node is the image of the zeros $m=\pm\sqrt{-3}$ of
$x+1=m^2+3$. Since $2\sqrt{-3}$ is a unit of $R_1$, these two sections
are disjoint and the coordinate in (b) is an automorphism of
$\bP^1_{R_1}$. This proves (b).

For (c) put $x_1:=x+1$ and write the equation as
\[
 y^2+3x_1^2(1-x_1/3)=(\vartheta-1)(1-3x_1),\qquad
 \vartheta-1=1728\,\vartheta\,\epsilon_0s^{\ell}.
\]
In the completed local ring the left side is the product of
$\varphi_1:=y+\sqrt{-3}\,x_2$ and $\varphi_2:=y-\sqrt{-3}\,x_2$, where
$x_2:=x_1(1-x_1/3)^{1/2}$, and the elements $s$, $\varphi_1$,
$\varphi_2$ form a system of coordinates because $2\sqrt{-3}$ is a
unit. The right side is $s^{\ell}\epsilon_1$ for the unit
$\epsilon_1:=1728\,\vartheta\,\epsilon_0(1-3x_1)$. Hence
$x':=\epsilon_1^{-1}\varphi_1$ and $y':=\varphi_2$ are coordinates at
the node with $x'y'=s^{\ell}$; compare \cite{Stacks}*{0CBX}.

For (d), put $E_4:=1+240\sum_{n\ge1}\sigma_3(n)q^n$ and
$E_6:=1-504\sum_{n\ge1}\sigma_5(n)q^n$, where $\sigma_r(n)$ is the sum
of the $r$th powers of the divisors of $n$. The coefficients of
$\Tate(q)$ are $a_4=(1-E_4)/48$ and $a_6=(1-3E_4+2E_6)/1728$, so the
standard formulas \cite{SilvermanAEC}*{\S III.1} give its invariants
\[
 c_4=1-48a_4=E_4,\qquad c_6=-1+72a_4-864a_6=-E_6,
\]
and the series $1/j(q)=q-744q^2+\cdots$ has integer coefficients and
leading coefficient one, which identifies the two power series rings.
The Weierstrass equation of $E_j$ has the invariants
$(144\vartheta,1728\vartheta)$, and $\vartheta$ equals $E_4^3/E_6^2$ at
$j=j(q)$. When $6$ is invertible, two Weierstrass curves with
invariants $(c_4,c_6)$ and $(\varrho^4c_4,\varrho^6c_6)$, for a unit
$\varrho$, are isomorphic by an affine change of coordinates
\cite{SilvermanAEC}*{\S III.1}. Here
\[
 (E_4,\,-E_6)=(\varrho^4\cdot144\vartheta,\ \varrho^6\cdot1728\vartheta)
 \qquad\text{for}\qquad\varrho^2:=-\frac{E_6}{12E_4},
\]
and $\varrho$ exists in $\bZ[1/6,\sqrt{-3}][[q]]$ because
$E_6/E_4\in1+q\bZ[[q]]$ and $-1/12=(\sqrt{-3}/6)^2$. We take the root
with $\varrho(0)=-\sqrt{-3}/6$. The isomorphism sends $(x,y)$ to
$(\varrho^2x-1/12,\ \varrho^3y-(\varrho^2x-1/12)/2)$, and at $q=0$ a
direct substitution gives
$m=\sqrt{-3}\,(1+u)/(1-u)$, that is,
$u=(m-\sqrt{-3})/(m+\sqrt{-3})$. The series $E_4$, $E_6$ and $\varrho$
converge on a disc around $q=0$ on which $E_4E_6\ne0$.
\end{proof}

From now on we write $x$ and $y$ for the coordinates $x'$ and $y'$ of
Lemma~\ref{nm:weierstrass-infinity}(c), so that the completed local
ring of $\cE_s$ at the node of its closed fiber is
$R_1[[s,x,y]]/(xy-s^{\ell})$. The Weierstrass coordinates are not used again.

\subsection{The cuspidal degeneration}\label{nm:sec-degeneration}

Put $\cD:=K_1[[s]]$, let $\cE_{\cD}:=\cE_s\otimes_{R_1[[s]]}\cD$ be
the base change of $\cE_s$ to $\cD$, and let
\[
 Y_s\longrightarrow E_s:=E_V\otimes_{K_1(V)}K_1((s))
\]
be the base change of $\cY_V\to E_V$ along the map
$\Spec K_1((s))\to V$ given by the punctured formal disc
at $c'$. Let $\cY_{\cD}$ be the normalization of
$\cE_{\cD}$ in the function field of $Y_s$.

\begin{lemma}[Cuspidal degeneration]\label{nm:cuspidal-degeneration}
Recall that $h\mid\ell$ (\S\ref{nm:sec-modular}). After a finite
extension of $K_1$, the following hold.
\begin{enumerate}
\item $\cY_{\cD}\to\cE_{\cD}$ is \'etale
over the complement of the origin section $O$ in the smooth locus of
$\cE_{\cD}$, and along $O$ it is tamely ramified of
index $e$: at each point over $O_c$ the completed local rings are
$K_1[[s,z]]\to K_1[[s,v]]$, $z\mapsto v^e$, for a coordinate $z$ with
$O=\{z=0\}$.
\item There are $|\Gp|/h$ points of $\cY_{\cD}$ over the node of
$E_c$, and $\cY_{\cD}\to\cE_{\cD}$ is unramified at the generic points
of the two branches of the cuspidal fiber through the node.
\item The fiber of $\cY_{\cD}$ over $s=0$, with its
$\Gp$-action, is the cuspidal cover $Y_c\to E_c$; in particular its
normalization is the normalized cuspidal cover
$\widetilde Y_c\to\widetilde E_c$.
\end{enumerate}
\end{lemma}

\begin{proof}
All objects are defined over a number field $M\subset K_1$ that
contains $\sqrt{-3}$: the curves $V$ and $\overline V$, the cusp $c'$,
the cover $\cY_V$ and the generic fiber of $\cE_{\overline{\cV}}$,
which is the pullback of the Weierstrass family over
$\bP^1_M\setminus\{0,1728\}$. Let $\overline V_M$ be the model of
$\overline V$ over $M$, let $\overline V_M'$ be the complement of
$j_{\overline V}^{-1}(\{0,1728\})$ in it, and let
$\cE_M\to\overline V_M'$ be the model of the Weierstrass family. Let
$\cY_M$ be the normalization of $\cE_M$ in the function field of the
model of $\cY_V$ over $M$, a normal surface over $M$, finite over
$\cE_M$. The morphism $\Spec\cD\to\overline V_M'$ is regular: it is
composed of a completion, a localization,
and the base change $M\to K_1$ in characteristic 0.
Normalization commutes with regular base change \cite{Stacks}*{0C22},
so $\cY_{\cD}=\cY_M\times_{\overline V_M'}\Spec\cD$, and the three
assertions are statements about the algebraic surface $\cY_M$ near the
fiber over $c'$ that may be checked after extending the constant field.
Embed $M\subset\Qbar$ in $\bC$ by $\iota$, through which the Nielsen
pair representing $c$ is defined; the inclusion $M\subset K_1$ is the
one induced by $v$. Analytification commutes with normalization, and
the completed local rings of an algebraic surface and of its
analytification agree, so it suffices to describe the normalization of
the analytic family defined by $\cE_M$ over a small disc
$\Delta\subset\overline V(\bC)$ around $c'$.

\emph{The analytic Tate family.} By
Lemma~\ref{nm:weierstrass-infinity}(d), after shrinking $\Delta$ the
analytic family over $\Delta$ is isomorphic to the pullback of the Tate
curve along the holomorphic function $q$ on $\Delta$ with $j(q)=j$.
Since $1/j(q)=q-744q^2+\cdots$ and $\ell$ is the ramification index of
$j_{\overline V}$ at $c'$, this function has a zero of order $\ell$ at
$c'$. In this proof $s$ denotes a holomorphic coordinate on $\Delta$
with $q=s^{\ell}$; the assertions of the lemma do not depend on the
choice of a coordinate at $c'$. The Tate curve is uniformized by two
explicit series $X(u,q)$, $Y(u,q)$ \cite{SilvermanAdvanced}*{Theorem~V.1.1}.
For $s\in\Delta^*:=\Delta\setminus\{0\}$ they identify the fiber with
$E_s(\bC)=\bC^{\times}/q^{\bZ}$, the origin $O$ being the class of
$u=1$. At $s=0$ they are the uniformization of
Lemma~\ref{nm:weierstrass-infinity}(d), so $u$ is the coordinate on
$\widetilde E_c$ in which $0_c=0$, $O_c=1$ and $\infty_c=\infty$. In
terms of $(x,y):=(u,q/u)$ the two series extend holomorphically to a
neighborhood of the origin, where their linear terms are $x+y$ and
$-y$, and the Tate surface projects isomorphically to the plane of its
Weierstrass coordinates near its node, because the derivative of its
equation with respect to $q$ equals $1$ there. So a neighborhood of the
node in the total space is
\[
 N:=\{(x,y,s):\ xy=s^{\ell},\ |x|,|y|<\delta\},
 \qquad x=u,\quad y=q/u,
\]
and its fiber $N_s$ over $s\ne0$ is the annulus
$\{|q|/\delta<|u|<\delta\}$, whose core circle
$C_s:=\{|u|=|q|^{1/2}\}$ is the vanishing cycle. The complement
$A_s:=E_s\setminus N_s$ is the image of the compact annulus
$A:=\{\delta\le|u|\le\delta^{-1}\}$, which contains $O$. Thus
$\{A_s\}_{s\in\Delta}$ is the trivial family $\Delta\times A$, whose
fiber at $s=0$ is $\widetilde E_c$ minus two small discs around $0_c$
and $\infty_c$.

\emph{Monodromy.} Fix $s_0\in\Delta^*$ and a base point $u_0$ of
$E_{s_0}^\circ$. Let $(\alpha,\beta)$ be the basis of
$\pi_1(E_{s_0}^\circ)$ formed by the class $\alpha$ of $C_{s_0}$ and the
class $\beta$ of a path from $u_0$ to $qu_0$. By the definition of a
Nielsen pair representing $c$, the monodromy homomorphism of the cover
of $E_{s_0}^\circ$ sends $(\alpha,\beta)$ to such a pair $(a,b)$. As
$s$ runs once
around $0$, the parameter $q$ runs $\ell$ times around $0$, and the
monodromy of the family $\{E_s^\circ\}$ on $E_{s_0}^\circ$
is the $\ell$th power of the Dehn twist about $C_{s_0}$, which sends
$(\alpha,\beta)$ to $(\alpha,\alpha^{\ell}\beta)$; the images of these
elements in $\Gp$ are $(a,a^{\ell}b)=(a,b)$, because $h\mid\ell$. The covers
$Y_s$ form a $\Gp$-torsor over the family
$\{E_s^\circ\}_{s\in\Delta^*}$. The fundamental group of the total
space of this family is generated by $\pi_1(E_{s_0}^\circ)$ and a loop
$t$ around $s=0$ through a base point in $A_{s_0}$, which acts on
$\pi_1(E_{s_0}^\circ)$ by this power of the Dehn twist. The image
$g\in\Gp$ of $t$ under the monodromy representation of the torsor
therefore satisfies $gag^{-1}=a$ and $gbg^{-1}=a^{\ell}b=b$. As $(a,b)$
generates $\Gp$ and $Z(\Gp)=1$, we get $g=1$. Consequently the
restriction of the torsor to the complement of the annuli $N_s$, where
the Dehn twist can be taken to be the identity and the family is the
product $\Delta^*\times(A\setminus O)$, is the pullback of its
restriction to one fiber.

\emph{Over the complement of the node.} By the last remark, the
restriction of the family of covers to $\Delta^*\times(A\setminus O)$
is the pullback of a single finite \'etale cover of $A\setminus O$,
namely the restriction of the cover $\widetilde Y_c\to\widetilde E_c$:
cutting $E_{s_0}$ along the vanishing cycle turns the punctured torus
into the annulus $A\setminus O$, whose fundamental group is free on
the loop around the inner boundary circle, which is $\alpha$, and the
loop around $O$; the loop around the outer boundary circle, transported
to the base point along the path $\beta$ across the cut, is
$\beta\alpha^{-1}\beta^{-1}$; and the product of the three loops around
$0_c$, $\infty_c$ and $O_c$ is one. Their images in $\Gp$ form the
triple \eqref{nm:cusp-triple}. Hence the
normalization of $\Delta\times A$ in the family of covers is
$\Delta\times(\widetilde Y_c|_A)$. It is \'etale over
$\Delta\times(A\setminus O)$ and tamely ramified of index $e$ along
$\Delta\times O$, with the Kummer form $z=v^e$ in a coordinate $z$
vanishing on $O$. This proves (a).

\emph{At the node.} Over $N\setminus\{s=0\}$, the family of covers
restricts to a finite \'etale cover of the family of annuli
$\{N_s\}_{s\in\Delta^*}$ whose monodromy group is $\langle a\rangle$,
of order $h$. Its connected components are indexed by
$\Gp/\langle a\rangle$, and each is a connected $h$-fold cyclic cover
of the annulus family. The connected $h$-fold cyclic cover of the
annulus $\{xy=t\}$ is the annulus $\{UV=t^{1/h}\}$ with $x=U^h$ and
$y=V^h$, and over $\Delta^*$ the root $t^{1/h}=s^{\ell/h}$ is
single-valued because $h\mid\ell$. Hence the normalization of $N$ in
the family of covers is the disjoint union of $|\Gp|/h$ copies of the
normal surface $\{UV=s^{\ell/h}\}$, each mapping by
$(U,V,s)\mapsto(U^h,V^h,s)$. Its completed local rings at the points
$U=V=s=0$ are $\bC[[s,U,V]]/(UV-s^{\ell/h})$, mapping to the
completed local ring $\bC[[s,x,y]]/(xy-s^{\ell})$ of the node by
$x=U^h$, $y=V^h$. This proves (b). There are $|\Gp|/h$ points over the
node. The generic points of the two branches of the cuspidal fiber
through the node are the primes $(s,x)$ and $(s,y)$, and at $(s,x)$
the element $s$ remains a uniformizer while the residue field
extension, generated by a $h$th root of $y$, is separable; the same
holds at $(s,y)$.

\emph{The fiber at $s=0$.} The normalization of the total space
restricts over $s=0$ to $\widetilde Y_c|_A$ together with the
$|\Gp|/h$ nodal fibers $\{UV=0\}$ of the surfaces just described. Each
surface $\{UV=s^{\ell/h}\}$ joins, for $s\ne0$, a boundary circle of
$\widetilde Y_c|_A$ over $0_c$ to the boundary circle over $\infty_c$
that the path $\beta$ carries it to; at $s=0$ the two branches $U=0$ and
$V=0$ attach the corresponding points of $\widetilde Y_c$ over $0_c$
and $\infty_c$ to a node. The fiber is therefore the nodal curve
obtained from $\widetilde Y_c$ by identifying each point over $0_c$
with its image under $\gamma_b$, with its $\Gp$-action: this is $Y_c\to E_c$,
which proves (c).
\end{proof}

\subsection{Proof of Proposition~\ref{nm:bad-cusp}}
\label{nm:sec-proof-bad-cusp}

We keep the notation of \S\ref{nm:sec-modular} and
\S\ref{nm:sec-degeneration}, with $h\mid\ell$.

\emph{Step 1: good reduction at the generic point of the special
fiber.}
Let $\eta$ be the generic point of the irreducible special fiber
$\overline{\cV}_k$ and let
\[
 F:=k(\overline{\cV}_k)
\]
be its residue field. Since $\overline{\cV}_k\to\bP^1_{j,k}$
is finite and separable, $F/k(\bar j)$ is finite separable, where
$\bar j$ denotes the image of $j$ in $F$, an element transcendental
over $k$. Let $R_\eta$ be the completion of the local ring of
$\overline{\cV}$ at $\eta$, a complete discrete valuation ring
with uniformizer $\pi$ and residue field $F$, and let $K_\eta$ be its
fraction field, a completion of the function field $K_1(V)$. Put
\[
 E_\eta:=E_V\otimes_{K_1(V)}K_\eta,\qquad
 Y_\eta:=\cY_V\otimes_{K_1(V)}K_\eta .
\]
Then $Y_\eta\to E_\eta$ is a geometrically connected $\Gp$-cover, branched only over the origin with
ramification index $e$, which is nontrivial and prime to $p$. The
elements $j$ and $j-1728$ of $K_\eta$ are units of $R_\eta$ with
nonzero residues, so the Weierstrass equation of $E_j$, whose
discriminant is $1728^2j^2/(j-1728)^3$, defines a smooth model
$\cE_\eta$ of $E_\eta$ over $R_\eta$, namely the base change of
$\cE_{\overline{\cV}}$ to $R_\eta$, with special fiber
$E_0:=\cE_\eta\otimes_{R_\eta}F$ and $j(E_0)=\bar j$. The
standing hypotheses of \S\ref{sec:generic} hold for
$(K,R,F,E,Y)=(K_\eta,R_\eta,F,E_\eta,Y_\eta)$, and since
$F/k(\bar j)$ is separable, Corollary~\ref{gen:generic-good} gives
good reduction of $Y_\eta\to E_\eta$: there are a finite
extension $K'/K_\eta$ with valuation ring $R'$ and a good model
\[
 \cY'\longrightarrow\cE_\eta\otimes_{R_\eta}R'
\]
of $Y_\eta\otimes_{K_\eta}K'\to E_\eta\otimes_{K_\eta}K'$, that is, a
finite $\Gp$-cover of smooth proper $R'$-curves which is \'etale outside
the origin section, along which it is tamely ramified with index $e$
(Lemma~\ref{bg:good-model}). Raynaud's descent theorem
\cite{Raynaud1999}*{Proposition~4.1.2 and Corollary~4.1.3}
now gives good tame reduction over $R_\eta$
itself: the normalization $\cY_\eta$ of $\cE_\eta$ in
$K_\eta(Y_\eta)$ is smooth over $R_\eta$, and
\begin{equation}\label{nm:special-good}
 \cY_\eta\longrightarrow\cE_\eta
 \quad\text{is a tame $\Gp$-cover, \'etale over }
 \cE_\eta^\circ:=\cE_\eta\setminus O.
\end{equation}

\emph{Step 2: transport to the formal neighborhood of the cusp.}
The inclusion of the local ring of $\overline{\cV}$ at
$\bar c'$ into its completion $R_1[[s]]$ induces a local homomorphism
of discrete valuation rings from the local ring at $\eta$ to
$R_1[[s]]_{(\pi)}$, both with uniformizer $\pi$, whose residue
extension $F\to k((s))$ is the completion of $F$ at $\bar c'$.
Let $R_s$ be the completion of $R_1[[s]]_{(\pi)}$, a complete discrete
valuation ring with residue field $k((s))$; the homomorphism above
extends to a local homomorphism $R_\eta\to R_s$. By
Lemma~\ref{bg:good-model}(e), the base change
$\cY_\eta\otimes_{R_\eta}R_s\to\cE_\eta\otimes_{R_\eta}R_s$ is a good
model of the base change of $Y_\eta\to E_\eta$ to the fraction field
of $R_s$: it is \'etale over $\cE_\eta^\circ\otimes_{R_\eta}R_s$. The curves
$\cE_\eta\otimes_{R_\eta}R_s$ and $\cE_s\otimes_{R_1[[s]]}R_s$ are both
the base change of $\cE_{\overline{\cV}}$ along the same morphism
$\Spec R_s\to\overline{\cV}$, so they coincide. Thus the good
reduction obtained over $F$ in Step~1 is available over $k((s))$.

\emph{Step 3: purity along the cuspidal fiber.}
Let $\Omega:=\operatorname{Frac}R_1[[s]]$, let $Y_\Omega\to E_\Omega$
be the base change of $\cY_V\to E_V$ along
$K_1(V)\to\Omega$, and let $\cY_s$ be the normalization of
$\cE_s$ in the function field of $Y_\Omega$. Normalization commutes with
regular base change
\cite{Stacks}*{0C22}, so the base change
of $\cY_s$ along $R_1[[s]]\to\cD=K_1[[s]]$ is $\cY_{\cD}$, and its
base change along $R_1[[s]]\to R_s$ is the normalization of
Step~2.

We describe $\cY_s$ along the cuspidal fiber $s=0$ through three
charts of $\cE_s$, the completed local rings $A$ at the points of the
fiber above the closed point $(\pi,s) = (0,0)$ of $\Spec R_1[[s]]$.
Each $A$ is a normal complete local domain
of dimension three. It is regular in the smooth chart and in the
origin chart, but not in the node chart when $\ell>1$.
In each chart, we either work over the chart itself or pull back along a finite morphism $\Spec B\to\Spec A$ with $B$ regular.
We then show that the normalization of the pullback is \'{e}tale over $\Spec B$ by checking unramifiedness at every height-one prime of $B$ and applying purity of the branch locus \cite{Stacks}*{0BMB}.
A height-one prime $\mathfrak P$ of $A$ cannot contain both $\pi$ and
$s$, because the closed fiber of $A$ over $R_1[[s]]$ has dimension
one. There are three cases.
\begin{itemize}
\item $\pi\in\mathfrak P$. Then $\mathfrak P\cap R_1[[s]]=(\pi)$. After
base change to $R_s$, Step~2 shows that the cover is \'etale over
$\cE_\eta^\circ\otimes_{R_\eta}R_s$, so it is unramified at
$\mathfrak P$.
\item $s\in\mathfrak P$. Then $\mathfrak P\cap R_1[[s]]=(s)$, and the
completion of $R_1[[s]]$ at $(s)$ is $\cD$. After base change to $\cD$
the cover is $\cY_{\cD}\to\cE_{\cD}$, which is described by
Lemma~\ref{nm:cuspidal-degeneration}. In the smooth chart and in the
origin chart $\mathfrak P=(s)$, and the cover is unramified there by
part (a) of the lemma. In the node chart $\mathfrak P$ is $(s,x)$ or
$(s,y)$, the generic point of one of the two branches of the cuspidal
fiber through the node, and the cover is unramified there by part (b).
\item $\pi\notin\mathfrak P$ and $s\notin\mathfrak P$. Put
$W:=\Spec R_1[[s]][1/(\pi s)]$. The morphism
$\Spec A_{\mathfrak P}\to\Spec R_1[[s]]$ factors through $W$, and the
composite $W\to\overline{\cV}$ lands in $V$: it avoids the special
fiber because $\pi$ is invertible on $W$, and it avoids the points of
$\overline V\setminus V$, which lie over $j\in\{0,1728,\infty\}$,
because by \eqref{nm:j-and-s} the functions $j^{-1}=s^{\ell}\epsilon_0$,
$j$ and $1-1728j^{-1}$ are units on $W$. Over $V$ the curve
$\cE_{\overline{\cV}}$ is $E_V$, so $\cE_s\times_{\Spec R_1[[s]]}W
=E_V\times_VW$. The finite cover $\cY_V\times_VW$ of this curve has
the function field of $Y_\Omega$, and it is normal: $\cY_V\to E_V$ is
\'etale over $E_V^\circ$ and is, locally for the \'etale topology along
$O$, a disjoint union of Kummer covers $v^e=z$ for a local equation $z$
of $O$ \cite{Stacks}*{0EYG}; both descriptions persist after the base
change $W\to V$, and these Kummer covers of the regular scheme
$E_V\times_VW$ are regular. Hence
$\cY_s\times_{\Spec R_1[[s]]}W=\cY_V\times_VW$, which is \'etale away
from $O$ and tamely ramified of index $e$ along $O$.
\end{itemize}
So $\cY_s$ is unramified at every height-one prime of every chart,
except at the prime of the origin chart that defines $O$, where it is
tamely ramified of index $e$.

\emph{Smooth chart.} At a point of the cuspidal fiber away from $O_c$
and the node, the completed local ring of $\cE_s$ is regular, and
the cover is unramified at every height-one prime by the case
analysis; so $\cY_s$ is \'etale over this chart.

\emph{Origin chart.} The completed local ring of $\cE_s$ at the
origin of the cuspidal fiber is $R_1[[s,z]]$ with $O=\{z=0\}$, along
which the cover is tamely ramified of index $e$. Pull back along
$z=v^e$ to the regular ring $B_O:=R_1[[s,v]]$. By Abhyankar's lemma
\cite{Stacks}*{0EYG} the normalized pullback is unramified at $(v)$,
the only height-one prime of $B_O$ over $O$, and it is unramified at
the other height-one primes by the case analysis; so it is \'etale
over $B_O$.

\emph{Node chart.} The completed local ring of $\cE_s$ at the
node of the cuspidal fiber is
$A_{\rm node}:=R_1[[s,x,y]]/(xy-s^{\ell})$, and we use the finite map
\begin{equation}\label{nm:purity-chart}
 A_{\rm node}\longrightarrow B_{\rm node}:=R_1[[u,v]],
 \qquad x=u^{\ell},\quad y=v^{\ell},\quad s=uv,
\end{equation}
which identifies $A_{\rm node}$ with the ring of invariants of the
balanced action $u\mapsto\zeta u$, $v\mapsto\zeta^{-1}v$ of
$\mu_{\ell}$ on the regular ring $B_{\rm node}$. Let
$\cY_{\rm node}$ be the normalization of $B_{\rm node}$ in the
finite \'etale algebra obtained from the function field of
$Y_\Omega$ by base change to $\operatorname{Frac}B_{\rm node}$. Let
$\mathfrak Q$ be a height-one prime of $B_{\rm node}$ and
$\mathfrak P$ the prime of $A_{\rm node}$ below it, which has height
one because $A_{\rm node}$ is normal and $B_{\rm node}$ is finite over
it. The node chart does not meet $O$, so by the case analysis $\cY_s$
is unramified, hence \'etale, over the discrete valuation ring
$(A_{\rm node})_{\mathfrak P}$. Its base change to
$(B_{\rm node})_{\mathfrak Q}$ is \'etale over a normal ring, hence
normal, and so it is the localization of $\cY_{\rm node}$ at
$\mathfrak Q$. Hence $\cY_{\rm node}$
is finite, normal and unramified in codimension one over the regular
ring $B_{\rm node}$, so it is finite \'etale by purity. Finite \'etale
covers of the complete local ring $B_{\rm node}$ correspond to finite
\'etale algebras over its residue field $k$, so after a finite
unramified extension of $K_1$ the cover $\cY_{\rm node}$ splits
into copies of $B_{\rm node}$. The normalization of $A_{\rm node}$ in
the cover is the quotient of $\cY_{\rm node}$ by $\mu_{\ell}$;
a copy of $B_{\rm node}$ with stabilizer $\mu_r\subseteq\mu_{\ell}$
contributes the invariant ring
\[
 B_{\rm node}^{\mu_r}=R_1[[U,V,s]]/(UV-s^{r}),
 \qquad U=u^{r},\quad V=v^{r},
\]
which has degree $\ell/r$ over $A_{\rm node}$. The cover has degree
$|\Gp|$ and, by Lemma~\ref{nm:cuspidal-degeneration}(b), $|\Gp|/h$ points
over the node, permuted transitively by $\Gp$; hence $\ell/r=h$.

\emph{The cuspidal fiber.} Let
$\cY_{s=0}:=\cY_s\otimes_{R_1[[s]]}R_1$ be the fiber of
$\cY_s$ over $s=0$. By Lemma~\ref{nm:cuspidal-degeneration}(c)
its generic fiber is the cuspidal cover $Y_c\otimes K_1$. Let
$\widetilde{\cY}_c$ be the normalization of $\cY_{s=0}$.
By the three charts, $\widetilde{\cY}_c$ is \'etale over the
smooth locus of $E_c\otimes R_1$ away from $O_c$, it is the tame
Kummer cover $z=v^e$ near $O_c$, and near the node it is the
normalization of $R_1[[U,V]]/(UV)$, that is, the disjoint union of
the two smooth branches $R_1[[U]]$ and $R_1[[V]]$, on which the points
over the node branches are the sections $U=0$ and $V=0$. Hence
$\widetilde{\cY}_c$ is a smooth proper $R_1$-curve, and
\[
 \widetilde{\cY}_c\longrightarrow
 \widetilde{\cE}_c:=\bP^1_{R_1}
\]
is a finite tame $\Gp$-cover of the normalization of $E_c\otimes R_1$,
branched only along the three disjoint sections $0_c$, $O_c$,
$\infty_c$, with generic fiber the normalized cuspidal cover
$\widetilde Y_c\otimes K_1\to\widetilde E_c\otimes K_1$. After a
finite extension of $K_1$, its connected components are geometrically
connected and the points over the three branch values are
$K_1$-rational. The component
$\widetilde{\cY}_c^{1}$ with generic fiber
$\widetilde Y_c^{1}\otimes K_1$ is smooth over $R_1$ and finite over
$\widetilde{\cE}_c$, hence is its normalization in that generic fiber.
It is therefore a good model of $\nu_c$ over
$R_1$, for the marking of $\widetilde E_c$ at $0_c$, $\infty_c$ and
$O_c$ fixed in \S\ref{bg:reduction}: $\nu_c$ has good reduction at
$v$, which proves Proposition~\ref{nm:bad-cusp}. \qed

\section{Existence of an irreducible cusp}
\label{nm:full-cusp}

To apply Theorem~\ref{thm:nm-arithmetic} to a normalized cuspidal
cover we need a cusp $c$ with $H_a=\Gp$, so that
$\widetilde Y_c\to\widetilde E_c$ is connected
(\S\ref{nm:cuspidal-notation}). This section proves that, for
$p>7$, every geometric connected component of $\Mabs{\Gp}$ has such a
cusp, and then assembles the
proof of Theorem~\ref{thm:nm-degree}. The argument
is group-theoretic: some primitive word in the generating pair has
image of order greater than five (Lemma~\ref{nm:large-primitive}), and
Lemma~\ref{nm:proper-subgroup} turns such an image into a cusp with
$H_a=\Gp$ (Lemma~\ref{nm:full-normalization}). We also compare the
pinched order with the cusp width (Lemma~\ref{nm:width}), which settles
the degree bound directly when that order is $p$.

\subsection{Primitive words and cusp widths}

A word in the free group $F_2$ is \emph{primitive} if it belongs to
some free basis. For a generating pair $(a,b)$ of a finite group
$G$ with associated surjection $\varphi_{a,b} : F_2\twoheadrightarrow G$ which sends the standard basis to $a,b$, the
\emph{primitive images} of $(a,b)$ are the images of the primitive words
under $\varphi_{a,b}$. Each primitive image can be placed first in an oriented free
basis and hence is the pinched generator at a cusp of the same
component. The \emph{absolute cusp width} of that cusp is the length
of the corresponding Dehn-twist cycle on the absolute sheets; it is
the ramification index of the $j$-map $j_X$ at the cusp.

Choose lifts $A,B\in\SL_2(\bF_p)$ of $a,b$ and write
\[
 x=\tr(A),\quad y=\tr(B),\quad z=\tr(AB),\quad
 \tau=\tr([A,B]).
\]
By \eqref{bg:fricke} the Fricke relation is $x^2+y^2+z^2-xyz=\tau+2$,
and $\tau\ne2$ because the pair is absolutely irreducible. By
Lemma~\ref{bg:absolute-sheets}, the absolute sheets over a general
$j$ are the sign classes \eqref{bg:signs} of trace triples. The oriented Nielsen transformations
$(a,b)\mapsto(b,a^{-1})$ and $(a,b)\mapsto(a,ab)$ act on trace
triples by
\begin{equation}\label{nm:nielsen}
 S(x,y,z)=(y,x,xy-z),\qquad
 T(x,y,z)=(x,z,xz-y),
\end{equation}
as follows from $\tr(A^{-1}B)=xy-z$ and $\tr(A^2B)=xz-y$. Thus the
orbit of a trace triple under $S$ and $T$, taken modulo the sign
changes of \eqref{bg:signs}, has cardinality $d_X$. The transformation $T$ is
the Dehn twist about the first generator.

\begin{lemma}\label{nm:width}
If a primitive image has order $h>2$, its absolute cusp width is
$h$. In particular a primitive image of order $p$ gives $d_X\ge p$.
\end{lemma}

\begin{proof}
Let $a$ be the primitive image, placed first. Since $h>2$, its trace
satisfies $x\ne0$. The twist $T$ fixes $x$ and acts on $(y,z)$ by
the matrix
$M_x=\left(\begin{smallmatrix}0&1\\-1&x\end{smallmatrix}\right)$,
so $T^r$ returns the triple to its sign class exactly when
$M_x^r(y,z)^{\mathsf t}=\pm(y,z)^{\mathsf t}$. The vector $(y,z)$ is
nonzero, since $y=z=0$ would make $b$ and $ab$ involutions and the
image dihedral. The matrix $M_x$ has the same characteristic
polynomial as $A$. If $a$ is semisimple with lifted eigenvalues
$\lambda,\lambda^{-1}$, the least $r>0$ with
$M_x^r(y,z)^{\mathsf t}=\pm(y,z)^{\mathsf t}$ is the least $r$ with
$\lambda^r=\pm1$, namely $h$. If $a$ is unipotent, then $x=\pm2$ and
$M_x$ has a nontrivial Jordan block; on its eigenline the Fricke
relation gives $\tau=2$, contrary to $\tau\ne2$, and off that line the projective
period is $p$. Since every cusp cycle has length at most $d_X$, the
last assertion follows.
\end{proof}

For $h=2$ one has $x=0$, the sign $\epsilon_1$ in \eqref{bg:signs}
is free, and the width is one or two.

\subsection{A primitive image of ``large'' order}

\begin{lemma}\label{nm:large-primitive}
Let $p>7$. Every generating pair of $\Gp$ has a primitive
image of order greater than five.
\end{lemma}

\begin{proof}
The traces of elements of projective order two, three, four and five
form a fixed finite set in characteristic zero. We enumerate all
Nielsen orbits confined to that set and then isolate the primes at
which reduction can create an additional confined orbit. The finite
calculations and their executable certificates are supplied in
\cite{NonMarkoffCert}.

Suppose all primitive images had order at most five. A primitive
image is never the identity, since it belongs to a generating pair.
So every coordinate of every triple in the orbit \eqref{nm:nielsen}
would be the reduction of an element of
\[
 \cT=\left\{0,\ \pm1,\ \pm\sqrt2,\
 \pm\frac{1+\sqrt5}{2},\ \pm\frac{1-\sqrt5}{2}\right\},
\]
the nine roots of $t(t^2-1)(t^2-2)(t^4-3t^2+1)$ in
$L=\bQ(\sqrt2,\sqrt5)$.

For $x,y,z,t\in\cT$, the nonzero integers
$|N_{L/\bQ}(xy-z-t)|$ are exactly
\[
 1,4,9,16,25,31,36,41,49,64,81,100,121,256,324.
\]
Let $p\notin\{2,3,5,7,11,31,41\}$, and let $\rho$ be reduction at a
prime of $L$ above $p$. If $x,y,z,t\in\cT$ satisfy $\rho(xy-z)=\rho(t)$,
then $xy-z=t$, since otherwise $p$ would divide the norm of $xy-z-t$;
with $x=0$ this shows that $\rho$ is injective on $\cT$, as $\cT=-\cT$.
The maps $S$, $T$ and their inverses replace one coordinate of a triple
by an expression $xy-z$ in the other two and commute with $\rho$. Hence
a triple in $\rho(\cT)^3$ whose orbit under $S$ and $T$ stays in
$\rho(\cT)^3$ is the reduction of a unique triple in $\cT^3$, the orbit
of that triple stays in $\cT^3$, and $\rho$ maps the one orbit
bijectively onto the other. It therefore suffices to find the orbits in
$\cT^3$ that stay in $\cT^3$, which we call \emph{confined}.

To find every confined orbit, start with the $9^3=729$ triples in
$\cT^3$; delete each triple for which an $S$- or $T$-image
lies outside the current set, and repeat until no deletion occurs. A
confined orbit survives every deletion. Conversely, the final set is
preserved by the injective maps $S$ and $T$, which therefore permute
it, so its orbits are exactly the confined orbits. The two nonempty
deletion rounds remove $464$ and $36$ triples, leaving $229$ triples
in ten orbits.

For nine of these orbits the calculation finds a triple of orders
$(r,s,t)$ with $1/r+1/s+1/t>1$; the generated projective group is
then a quotient of the finite spherical triangle group with these
orders. The tenth orbit requires a change of generators, given below.
\begin{center}
\begin{tabular}{rll}
Number of orbits&Orbit sizes&Spherical type\\\hline
1&16&$(2,3,3)$\\
1&36&$(2,3,4)$\\
2&40,40&$(2,3,5)$\\
1&6&$(2,2,3)$\\
1&6&$(2,2,4)$\\
2&6,6&$(2,2,5)$\\
1&1&$(2,2,2)$\\
1&72&$(3,2,5)$, by the word below
\end{tabular}
\end{center}
For the last orbit take $\tr(A)=\tr(B)=-1$ and $z^2-z-1=0$. Put
$C=AB$ and $V=AC^2$. Cayley--Hamilton gives
\[
 \tr(AC)=1-z,\qquad \tr(V)=0,\qquad \tr(AV)=1-z,
\]
so projectively $A^3=V^2=(AV)^5=1$. Since $C$ has order five and
$C^2=A^{-1}V$, one has $C=(C^2)^3$, so $A$ and $V$ generate the
original image. Every image in the table is consequently a quotient
of a spherical triangle group of order at most $60$. The three trace
identities follow from the Cayley--Hamilton theorem over any
commutative ring, so these relations hold in $\Gp$.

It remains to treat $p=11,31,41$. Form all triples of roots of the
same polynomial in $\bF_p$, perform the same deletions, and choose one
representative of each remaining orbit. Triples with $\tau=2$ are
reducible, and triples with two zero coordinates generate a dihedral
group, so neither can generate $\Gp$. For every other orbit, realize
its representative by two matrices of determinant one and close the
generated subgroup under multiplication, identifying opposite
matrices; this computes its projective order. The exact outputs are
\begin{center}
\begin{tabular}{rrl}
$p$&Surviving triples&Non-skipped subgroup orders\\\hline
11&187&$12,60$\\
31&229&$12,24,60$\\
41&229&$12,24,60$
\end{tabular}
\end{center}
{\sloppy The ancillary files \texttt{small\_primitive\_orders.py} and
\texttt{small\_primitive\_orders.json} record the norm tests, the
orbit deletion and the full matrix closures. The independent
\texttt{small\_order\_magma\_check.m}, run in Magma V2.29-9,
reproduces them with a different number-field representation and
with the faithful action on $\bP^1(\bF_p)$; its output is
included. Since $|\Gp|>60$ for $p>5$, no orbit in either
table contains a generating pair, which completes the finite
reduction and proves the lemma.\par}
\end{proof}

\subsection{Existence of an irreducible cusp}

\begin{lemma}\label{nm:full-normalization}
Let $p\ge5$ and let $(a,b)$ be a generating pair of $\Gp$ whose
absolute structure lies in $X$. If the largest order $M$ of a primitive
image of $(a,b)$ is greater than five, then $X$ has a cusp $c$,
represented by a Nielsen pair $(a',b')$ in the $\SL_2(\bZ)$-orbit of
$(a,b)$, with pinched order $\ord(a')>5$ and $H_{a'}=\Gp$; its
normalized cuspidal cover $\widetilde Y_c\to\widetilde E_c$
is then connected with group $\Gp$.
\end{lemma}

\begin{proof}
First suppose that $h=\ord(a)>2$ and $\max(h,e)>5$, where
$e=\ord([a,b])$, and consider $H_a=\langle a,ba^{-1}b^{-1}\rangle$.
If $H_a$ is proper, Lemma~\ref{nm:proper-subgroup} puts it in a
Borel subgroup or in a torus normalizer, since it contains $a$ or
$[a,b]$ of order greater than five. In a torus normalizer, every
element outside the torus has order two, so the two conjugate
generators $a$ and $ba^{-1}b^{-1}$, of order $h>2$, lie in the torus;
that torus is the unique one containing $a$, and $bab^{-1}$ lies in it
as well, so $b$ normalizes it, contradicting
$\langle a,b\rangle=\Gp$. Thus a proper $H_a$ lies in a Borel subgroup,
and the pair $(a,ba^{-1}b^{-1})$, with trace triple $(x,x,\tau)$, is
reducible. The reducibility criterion for this triple is
\[
 2x^2+\tau^2-x^2\tau-4=(\tau-2)(\tau+2-x^2)=0,
\]
which, as $\tau\ne2$, says $x^2=\tau+2$. Conversely $x^2=\tau+2$
makes the pair reducible, and its common invariant line is then
defined over $\bF_p$: otherwise the Frobenius-conjugate line would be
invariant too, and the two generators would lie in a torus, which was
excluded. Hence, under the present assumptions, $H_a=\Gp$ if and only
if $x^2\ne\tau+2$.

Now choose $a$ of order $M$. If $x^2\ne\tau+2$, we are done.
Otherwise $a$ is split semisimple: a unipotent $a$ would give
$\tau+2=x^2=4$ and hence $\tau=2$, and the invariant line just found
is rational. Write $\lambda,\lambda^{-1}\in\bF_p$ for the eigenvalues
of $A$. Since $x^2=\tau+2$, the Fricke relation becomes
\[
 y^2+z^2-xyz=(z-\lambda y)(z-\lambda^{-1}y)=0.
\]
Neither $y$ nor $z$ vanishes: if one did, so would the other, and
the pair would be dihedral. Choose $\lambda$ so that $z=\lambda y$.
Then, in a basis diagonalizing $A$, the matrix $B$ has lower right
entry zero, and the primitive images $a^rb$ have lifted traces
$\tr(A^rB)=\lambda^ry$. Their squares $\lambda^{2r}y^2$ take $M$
distinct nonzero values as $r$ ranges over $0,\ldots,M-1$, because
$\lambda^2$ has order $M$. A nonidentity element of projective order
at most five has trace square $0$, $1$, $2$, or a root of
$q^2-3q+1$. The value $0$ does not occur among the $\lambda^{2r}y^2$,
so excluding the remaining four values and $\tau+2$ removes at most
five of the $M>5$ values, and some $a^rb$ has order greater than five and
trace square different from $\tau+2$. Apply the first paragraph to an
oriented basis whose first entry is $a':=a^rb$: the cusp it
represents has pinched order greater than five and $H_{a'}=\Gp$.
\end{proof}

\section{Collation}
\label{nm:collation}

This section assembles the proofs of Theorem~\ref{thm:nm-degree} and of
Corollaries~\ref{cor_MW} and~\ref{cor_connected} from the results of the
preceding sections.

\subsection{Proof of Theorem~\ref{thm:nm-degree}}
\label{nm:proof-main}

\begin{proof}[Proof of Theorem~\ref{thm:nm-degree}]
\emph{Markoff components.} By \cite{Chen}*{Theorem~1.2.5}, for
$p\ge3$ every $\SL_2(\bZ)$-orbit of absolute
$\SL_2(\bF_p)$-structures with $\tau=-2$ has cardinality divisible by
$p$. Such an orbit maps equivariantly onto the orbit of the
corresponding absolute $\Gp$-structures, with fibers of
constant cardinality at most four, the number of sign changes of a
lifted pair. For $p\ge5$ the degree $d_X$ of a Markoff component is
therefore divisible by $p$.

\emph{Small primes.} For $p\le 7$, the fact that $d_X\ge p$ can be checked directly. For $p = 2$, $G_2$ is the dihedral group of order 6, and inner and absolute structures coincide. The components are described in \cite{ChenModuli}*{Theorem~4.2.2}, and satisfy $d_X = 3 \ge 2$. For $p = 3$, $G_3\cong A_4$. Since $A_4$ surjects onto $C_3$, any component $X\subset M(A_4)^\abs$ surjects onto a component of $M(C_3)^\abs$ which is the congruence modular curve for $\Gamma_0(3) := \{\spmatrix{a}{b}{c}{d}\in\SL_2(\bZ)\;|\;c\equiv 0\mod 3\}$, of degree $4$ over $\bP^1_j$; hence $d_X\ge4$.

For $p=5,7$, a direct enumeration of the generating pairs up to automorphisms and inversion, and of the Nielsen moves on them (Lemma~\ref{bg:absolute-sheets}), shows that the components of $M(\Gp)^\abs$ have degrees $9,10$ for $p=5$ and $7,16,16,18$ for $p=7$; so $d_X\ge p$ in these cases too. This is consistent with the tables of \cite{ChenModuli}*{Appendix~B.2}, which list the components of $M(\Gp)$ with the indices $10,10,18$ and $7,7,32,32,36$ of their stabilizers in $\SL_2(\bZ)$: the $j$-map of a component of $M(\Gp)$ has degree the index or half of it, according as $-I$ lies in the stabilizer or not, and the map $M(\Gp)\to M(\Gp)^\abs$ has degree one or two on each component.

\emph{The remaining case.} Let $p>7$ and let $X$ be a non-Markoff
component. By Lemmas~\ref{nm:large-primitive} and
\ref{nm:full-normalization}, $X$ has a cusp $c$ with pinched order
$h>5$ and $H_a=\Gp$, so that
its normalized cuspidal cover
$\nu_c\colon\widetilde Y_c\to\widetilde E_c$ is connected with group
$\Gp$. If $h=p$, Lemma~\ref{nm:width}
gives $d_X\ge p$. Otherwise $h$ is prime to $p$, since the order of an
element of $\Gp$ is $p$ or divides $(p\pm1)/2$. By
Theorem~\ref{thm:nm-arithmetic} there is a place $v$ above $p$ at which
$\nu_c$ has bad reduction, so by Proposition~\ref{nm:bad-cusp} the
modular curve $j_X\colon\overline X\to\bP^1_j$ has bad reduction at
$v$. Suppose that $d_X<p$. The monodromy group of $j_X$ embeds in the
symmetric group $S_{d_X}$, so its order is prime to $p$, and
Theorem~\ref{bg:grothendieck} gives good reduction of $j_X$ at every
place above $p$, a contradiction. Hence $d_X\ge p$.
\end{proof}

\begin{corollary}[Bad reduction of the $j$-map]\label{nm:bad-reduction-jmap}
Let $p>7$ and let $X$ be a non-Markoff component of $\Mabs{\Gp}$. Then
the modular curve $j_X\colon\overline X\to\bP^1_j$ has bad reduction at
some place of $\Qbar$ above $p$. In particular, the monodromy group of
$j_X$ has order divisible by $p$.
\end{corollary}

\begin{proof}
Let $c$ be a cusp of $X$ with pinched order $h>5$ and $H_a=\Gp$, as in
the proof of Theorem~\ref{thm:nm-degree}. If $h=p$, the width of $c$ is
$p$ by Lemma~\ref{nm:width}, so $j_X$ has a ramification index divisible
by $p$ and no good model at any place above $p$
(Lemma~\ref{bg:good-model}(c)). Otherwise $h$ is prime to $p$,
Theorem~\ref{thm:nm-arithmetic} gives a place $v$ above $p$ at which
$\nu_c$ has bad reduction, and Proposition~\ref{nm:bad-cusp} gives bad
reduction of $j_X$ at $v$; by Remark~\ref{nm:definite-remark} this
happens at every place at which $\nu_c$ is indefinite. The last
assertion follows from Theorem~\ref{bg:grothendieck}.
\end{proof}

\subsection{Proof of Corollaries~\ref{cor_MW} and~\ref{cor_connected}}
\label{nm:proof-corollaries}

We use the notation of \S\ref{intro:markoff}: $\Pi$ is the free group
on $\alpha,\beta$, and $F(p)_\tau$ is the set of
$\GL_2(\bF_p)$-conjugacy classes of surjections
$\varphi\colon\Pi\to\SL_2(\bF_p)$ with
$\tr[\varphi(\alpha),\varphi(\beta)]=\tau$, on which $\Out(\Pi)$ acts.
By \cite{Chen}*{Theorem~5.2.10}, the trace triple
$(\tr\varphi(\alpha),\tr\varphi(\beta),\tr\varphi(\alpha\beta))$
determines the class of an absolutely irreducible $\varphi$, and every
triple $(x,y,z)\in\bF_p^3$ with $x^2+y^2+z^2-xyz\ne4$ is the trace
triple of an absolutely irreducible representation.

\begin{lemma}[Nongenerating triples]\label{nm:nongenerating}
Let $p\ge5$ and $\tau\in\bF_p\setminus\{2\}$. Let $(x,y,z)\in\bF_p^3$
satisfy $x^2+y^2+z^2-xyz=\tau+2$, and let $(A,B)$ be a pair in
$\SL_2(\bF_p)$ with trace triple $(x,y,z)$. If two of $x,y,z$ are
zero, then $(A,B)$ does not generate $\SL_2(\bF_p)$. Conversely, if
$(A,B)$ does not generate $\SL_2(\bF_p)$, then two of $x,y,z$ are
zero, or $x$, $y$ and $z$ are roots of $t(t^2-1)(t^2-2)(t^4-3t^2+1)$.
In particular, at most $9^3+6$ points of the surface
$x^2+y^2+z^2-xyz=\tau+2$ are not trace triples of generating pairs.
\end{lemma}

\begin{proof}
The pairs $(B,A)$ and $(A^{-1},AB)$ generate the same group as $(A,B)$
and have trace triples $(y,x,z)$ and $(x,z,y)$, so the coordinates may
be permuted. If $y=z=0$, then $B^2=(AB)^2=-I$, hence $BAB^{-1}=A^{-1}$,
and $\langle A,B\rangle=\langle A\rangle\langle B\rangle$ has at most
$8p$ elements, fewer than $|\SL_2(\bF_p)|$.

Conversely, let $H:=\langle A,B\rangle$ be a proper subgroup. The pair
is absolutely irreducible because $\tr[A,B]=\tau\ne2$
\cite{Chen}*{Lemma~5.2.9 and Theorem~5.2.10}. If $p$ divides $|H|$,
then $H$ contains a single Sylow $p$-subgroup, since two distinct ones
generate $\SL_2(\bF_p)$ by Lemma~\ref{bg:sl2-generation}; so $H$
normalizes it and lies in a Borel subgroup, which contradicts absolute
irreducibility. Hence $p\nmid|H|$, and the image $\bar H$ of $H$ in
$\PGL_2(\overline{\bF}_p)$ is cyclic, dihedral, $A_4$, $S_4$ or $A_5$
by Lemma~\ref{bg:finite-subgroups}. If $\bar H$ is cyclic, then $H$ is
abelian, hence diagonalizable over $\overline{\bF}_p$, which again
contradicts absolute irreducibility. If $\bar H$ is dihedral, then $H$
lies in the normalizer of a torus $T$ of $\SL_2(\overline{\bF}_p)$,
and the elements of $H$ outside $T$ have trace zero; not both of $A$
and $B$ lie in $T$, and if exactly one of them does then $AB\notin T$,
so two of $x,y,z$ are zero. In the remaining cases $A$, $B$ and $AB$
are noncentral, because a pair in which $A$, $B$ or $AB$ is central is
reducible, and their images in $\bar H$ have order $2$, $3$, $4$ or
$5$. A noncentral semisimple element of $\SL_2(\bF_p)$ whose image in
$\PGL_2(\overline{\bF}_p)$ has order $k$ has eigenvalues
$\lambda^{\pm1}$ with $\lambda^k=\pm1$ and $\lambda^j\ne\pm1$ for
$0<j<k$, so its trace $\lambda+\lambda^{-1}$ is a root of $t$,
$t^2-1$, $t^2-2$ or $t^4-3t^2+1$ according as $k=2,3,4,5$. Hence
$x$, $y$ and $z$ are among the nine roots of
$t(t^2-1)(t^2-2)(t^4-3t^2+1)$. The last assertion follows: a triple
with two zero coordinates has its third coordinate a square root of
$\tau+2$.
\end{proof}

\begin{lemma}[Symmetric generating pairs]\label{nm:symmetric-pair}
Let $p\ge41$ be prime and let $\tau\in\bF_p$, $\tau\ne2$. There is a
generating pair $(A,B)$ of $\SL_2(\bF_p)$ with $\tr A=\tr B$ and
$\tr[A,B]=\tau$.
\end{lemma}

\begin{proof}
A pair $(A,B)$ with $\tr A=\tr B=x$ and $\tr AB=z$ has $\tr[A,B]=\tau$
exactly when
\begin{equation}\label{nm:symmetric-equation}
 2x^2+z^2-x^2z=\tau+2,
\end{equation}
by the Fricke identity \eqref{bg:fricke}. We first count the solutions
of \eqref{nm:symmetric-equation} in $\bF_p^2$. Since $p$ is odd, the
substitution $w=2z-x^2$ identifies them with the solutions of
\[
 w^2=D(x):=x^4-8x^2+4(\tau+2).
\]
If $\tau\ne-2$, the roots of $D$ are the square roots of
$4\pm2\sqrt{2-\tau}$, which are four distinct elements of
$\overline{\bF}_p$ because $\tau\ne2,-2$. The smooth projective model
of the curve $w^2=D(x)$ is then a double cover of $\bP^1$ branched at
four points, hence of genus one, and its two points at infinity are
rational over $\bF_p$ because the leading coefficient of $D$ is a
square. It is an elliptic curve over $\bF_p$, so by Hasse's theorem
\cite{SilvermanAEC}*{Theorem~V.1.1} the affine curve has at least
$p+1-2\sqrt p-2$ points over $\bF_p$. If $\tau=-2$, then
$D=x^2(x^2-8)$, and the pairs $(x,xu)$ with $x\ne0$ and $u^2=x^2-8$
are solutions; the conic $u^2-x^2=-8$ has $p-1$ points, at most two of
them with $x=0$, so there are at least $p-3$ solutions. In both cases
\eqref{nm:symmetric-equation} has at least $p-1-2\sqrt p$ solutions.

Each solution $(x,z)$ is the trace triple $(x,x,z)$ of a pair $(A,B)$
in $\SL_2(\bF_p)$, which is absolutely irreducible because
$\tr[A,B]=\tau\ne2$ \cite{Chen}*{Lemma~5.2.9 and Theorem~5.2.10}. If $(A,B)$ does not generate $\SL_2(\bF_p)$, then by
Lemma~\ref{nm:nongenerating} two of $x$, $x$, $z$ are zero, so $x=0$,
or $x$ is one of the nine roots of $t(t^2-1)(t^2-2)(t^4-3t^2+1)$,
among which is $0$; for each of these nine values of $x$,
\eqref{nm:symmetric-equation} has at most two solutions $z$. In total
at most $18$ solutions fail to come from a generating pair. Since
$p-1-2\sqrt p>18$ for $p\ge41$, a generating pair exists.
\end{proof}

\begin{proof}[Proof of Corollary~\ref{cor_MW}]
For $\tau=2$ the set $F(p)_\tau$ is empty by definition. Let
$\tau\ne2$. Composing a surjection $\Pi\to\SL_2(\bF_p)$ with
$\SL_2(\bF_p)\to\Gp$ gives a surjection onto $\Gp$, and the
$\Out^+(\Pi)$-orbit of the latter is the set of absolute
$\Gp$-structures of a component $X$ of $\Mabs{\Gp}$, which has $d_X$
elements (Lemma~\ref{bg:absolute-sheets}). Hence every
$\Out^+(\Pi)$-orbit on $F(p)$, and a fortiori every
$\Out(\Pi)$-orbit, has at least $d_X\ge p$ elements by
Theorem~\ref{thm:nm-degree}. Let $C$ be the implied constant of
Theorem~\ref{thm:NKN} for $\epsilon=\frac12$ and put $N:=C^2$. For
$p>N$, every $\Out(\Pi)$-orbit in $F(p)_\tau$ other than
$\cC(p)_\tau$ would lie in a set with fewer than $C\sqrt p<p$ elements.
So $F(p)_\tau=\cC(p)_\tau$.
\end{proof}

\begin{proof}[Proof of Corollary~\ref{cor_connected}]
The trace invariant is locally constant on $M(\SL_2(\bF_p))^\abs$
\cite{BBCL}*{\S2.5} and does not take the value $2$
\cite{Chen}*{Lemma~5.2.9}, which gives the decomposition. Fix
$\tau\ne2$. The components of $M(\SL_2(\bF_p))^\abs_\tau$ correspond
to the $\Out^+(\Pi)$-orbits on $F(p)_\tau$ (\S\ref{intro:markoff}),
and $F(p)_\tau$ is nonempty for $p\ge13$
\cite{McCulloughWanderley2011}*{Theorem~2.1}. Let $N$ be as in
Corollary~\ref{cor_MW} and let $p>\max(N,41)$. Then $F(p)_\tau$ is a
single $\Out(\Pi)$-orbit. Since $\Out^+(\Pi)$ has index two in
$\Out(\Pi)$, the set $F(p)_\tau$ consists of one or two
$\Out^+(\Pi)$-orbits, and in the second case every element of
$\Out(\Pi)\setminus\Out^+(\Pi)$ exchanges the two orbits and acts on
$F(p)_\tau$ without fixed points. The automorphism $\sigma$ of $\Pi$
exchanging $\alpha$ and $\beta$ acts on $\Pi^{\ab}$ with determinant
$-1$, so its class lies in $\Out(\Pi)\setminus\Out^+(\Pi)$, and it
sends the class of $\varphi$ to the class of $\varphi\circ\sigma$,
whose trace triple is $(y,x,z)$ if that of $\varphi$ is $(x,y,z)$.
Since the class of $\varphi$ is determined by its trace triple,
$\sigma$ fixes the class of every $\varphi\in F(p)_\tau$ with
$\tr\varphi(\alpha)=\tr\varphi(\beta)$, and such classes exist by
Lemma~\ref{nm:symmetric-pair}. Hence $F(p)_\tau$ is a single
$\Out^+(\Pi)$-orbit, and $M(\SL_2(\bF_p))^\abs_\tau$ is connected.
\end{proof}

\section{Genus asymptotics}
\label{nm:genus}

In this section $p\ge5$ and $G$ is $\SL_2(\bF_p)$ or $\Gp$.
We extend the method of \cite{Chen}*{\S5.6} to establish the genus asymptotics stated in
Corollary~\ref{cor_genus}.

\subsection{The genus formula}

For a geometric connected component $X$ of $\Mabs{G}$, let $\Omega_X$
be the set of its absolute sheets over a general point of the $j$-line:
for $G=\Gp$ the sign classes \eqref{bg:signs} of the trace triples in
the $\SL_2(\bZ)$-orbit defining $X$ (Lemma~\ref{bg:absolute-sheets}),
and for $G=\SL_2(\bF_p)$ the trace triples themselves
(\S\ref{intro:markoff}). In both cases $d_X=|\Omega_X|$. The
transformations $S$ and $T$ of \eqref{nm:nielsen} and
\begin{equation}\label{nm:order-three}
 R(x,y,z)=(xy-z,\;x,\;x^2y-xz-y),
\end{equation}
the action of $(a,b)\mapsto(ab^{-1},a)$ on trace triples, permute
$\Omega_X$.

\begin{lemma}[Genus formula]\label{nm:genus-formula}
Let $X$ be a geometric connected component of $\Mabs{G}$. Let $c$ be
the number of cycles of $T$ on $\Omega_X$, and let $e_2$ and $e_3$ be
the numbers of fixed points of $S$ and of $R$ on $\Omega_X$. Then
\[
 g(\overline X)=1+\frac{d_X}{12}-\frac{e_2}{4}-\frac{e_3}{3}-\frac c2 .
\]
\end{lemma}

\begin{proof}
Let $\gamma_0$, $\gamma_{1728}$, $\gamma_\infty$ and $\gamma_{-I}$ be
the automorphisms $(a,b)\mapsto(ab^{-1},a)$, $(b^{-1},a)$, $(a,ab)$
and $(a^{-1},b^{-1})$ of $\Pi$; they act on trace triples by $R$, $S$,
$T$ and the identity. Let $\Epi^{\ext}(\Pi,G)$ be the set of
surjections $\Pi\to G$ up to conjugation in $G$, the set of inner
$G$-structures on a general elliptic curve. By
\cite{Chen}*{Theorem~2.5.2(4)}, the smooth compactification of the
coarse scheme $M(G)$ is \'etale over $U_j$, its fiber over a point of
$U_j$ is $\Epi^{\ext}(\Pi,G)/\langle\gamma_{-I}\rangle$, and its fibers
over $0$, $1728$ and $\infty$ are the sets of orbits of $\gamma_0$,
$\gamma_{1728}$ and $\langle\gamma_{-I},\gamma_\infty\rangle$ on
$\Epi^{\ext}(\Pi,G)$. The group $\Out(G)$ acts on $M(G)$ over $M(1)$
through its action on $\Epi^{\ext}(\Pi,G)$, which commutes with that of
$\Aut(\Pi)$, and $\Mabs{G}=M(G)/\Out(G)$
\cite{Chen}*{Definition~2.5.13}. So the fibers of $j_X$ over $0$,
$1728$ and $\infty$ are the sets of orbits of $R$, $S$ and $T$ on the
absolute structures in $X$, that is, on $\Omega_X$, where
$\gamma_{-I}$ acts trivially: by Lemma~\ref{bg:absolute-sheets}(c) for
$\Gp$, and because inversion preserves traces for $\SL_2(\bF_p)$.
Since $\gamma_{1728}^2=\gamma_{-I}$, and $\gamma_0^3$ is $\gamma_{-I}$
followed by conjugation by $ab^{-1}$, the orbits of $S$ and of $R$ on
$\Omega_X$ have one or two, respectively one or three, elements. Hence
\[
 |j_X^{-1}(1728)|=\frac{d_X+e_2}{2},\qquad
 |j_X^{-1}(0)|=\frac{d_X+2e_3}{3},\qquad
 |j_X^{-1}(\infty)|=c,
\]
and Hurwitz's formula \cite{Liu}*{Theorem~7.4.16} for $j_X$, which has
degree $d_X$ and is \'etale over $U_j$, reads
\[
 2g(\overline X)-2
 =-2d_X+\sum_{b\in\{0,1728,\infty\}}\bigl(d_X-|j_X^{-1}(b)|\bigr).
\]
The display follows.
\end{proof}

\subsection{Counting in trace coordinates}

For $\tau\in\bF_p\setminus\{2\}$, the absolute sheets of
$M(\SL_2(\bF_p))^\abs_\tau$ are the generating trace triples
$F(p)_\tau$, and those of $M(\Gp)^\abs_\tau$ are their sign classes.
For $a\in\bF_p$, let $n(a)$ be the order of a noncentral element of
$\SL_2(\bF_p)$ of trace $a$: for $a\ne\pm2$ such an element is
semisimple with eigenvalues $\lambda^{\pm1}$ determined by $a$, and
$n(a)\ge3$ is the order of $\lambda$; further $n(2)=p$ and
$n(-2)=2p$. Let $h(a):=n(a)/\gcd(n(a),2)$ be the order of the image of
such an element in $\Gp$. For an integer $n\ge1$ let $D(n)$ be its
number of divisors and, with $\phi$ Euler's function,
\[
 \Phi(n):=\sum_{m\mid n}\frac{\phi(m)}{m}
 =\prod_{q^k\,\|\,n}\Bigl(1+k\Bigl(1-\frac1q\Bigr)\Bigr)\le D(n),
\]
the product over the prime powers exactly dividing $n$.

\begin{proposition}[Degree, cusps and genus]\label{nm:genus-fiber}
Let $p\ge5$ be prime and let $\tau\in\bF_p\setminus\{2\}$. The implied
constants below are independent of $p$ and $\tau$.
\begin{enumerate}
\item $|F(p)_\tau|=p^2+O(p)$. The degrees of $M(\SL_2(\bF_p))^\abs_\tau$
and $M(\Gp)^\abs_\tau$ over the $j$-line are $|F(p)_\tau|$ and
$|F(p)_\tau|/4$.
\item The number $c_\tau$ of cusps of $M(\SL_2(\bF_p))^\abs_\tau$
satisfies
\[
 c_\tau\le p\,\bigl(D(p-1)+D(p+1)\bigr)+3\le 4p\sqrt{p+1}+3,
\]
and every geometric connected component $X$ of
$M(\SL_2(\bF_p))^\abs_\tau$ or of $M(\Gp)^\abs_\tau$ satisfies
$|g(\overline X)-d_X/12|\le c_\tau/2+3$.
\item If $M(\SL_2(\bF_p))^\abs_\tau$ is nonempty and connected, then
so is $M(\Gp)^\abs_\tau$, and the genera of their smooth
compactifications are
\[
 \frac{p^2}{12}-\frac p4\bigl(\Phi(p-1)+\Phi(p+1)\bigr)+O(p)
 \qquad\text{and}\qquad
 \frac{p^2}{48}-\frac p8\Bigl(\Phi\Bigl(\frac{p-1}2\Bigr)
 +\Phi\Bigl(\frac{p+1}2\Bigr)\Bigr)+O(p).
\]
\end{enumerate}
\end{proposition}

For $\tau=-2$ the counts in the proof are carried out exactly in
\cite{Chen}*{Propositions~5.6.1, 5.6.2 and Theorem~5.6.3}.

\begin{proof}
\emph{The degree.} For $a\in\bF_p$ let $N_a$ be the number of triples
in $F(p)_\tau$ with first coordinate $a$. The points of the surface
$x^2+y^2+z^2-xyz=\tau+2$ with $x=a$ form the conic
$y^2-ayz+z^2=\tau+2-a^2$. If $a\ne\pm2$ and $a^2\ne\tau+2$, its
quadratic form is nondegenerate and its constant term is nonzero, so
it has $p-\chi(a^2-4)$ points, where $\chi$ is the quadratic character
of $\bF_p$; each of the at most four remaining rows has at most $2p$
points. Hence the surface has $p^2+O(p)$ points. Every point of the
surface is the trace triple of an absolutely irreducible pair
(\S\ref{intro:markoff}), and by Lemma~\ref{nm:nongenerating} at most
$9^3+6$ of them are not trace triples of generating pairs. So
$|F(p)_\tau|=p^2+O(p)$, and $N_a\le2p$ for every $a$. The degree of
$M(\SL_2(\bF_p))^\abs_\tau$ is $|F(p)_\tau|$ (\S\ref{intro:markoff}),
and that of $M(\Gp)^\abs_\tau$ is the number of sign classes in
$F(p)_\tau$ (Lemma~\ref{bg:absolute-sheets}), which is
$|F(p)_\tau|/4$: a nontrivial sign change fixes a triple only if two
of its coordinates vanish, and no triple in $F(p)_\tau$ has two zero
coordinates (Lemma~\ref{nm:nongenerating}).

\emph{Cuspidal $\SL_2(\bF_p)$-covers.} The cusps of
$M(\SL_2(\bF_p))^\abs_\tau$ are the cycles of $T$ on $F(p)_\tau$
(Lemma~\ref{nm:genus-formula}). The twist $T$ fixes the first
coordinate $a$ and acts on $(y,z)$ by the matrix
$M_a=\left(\begin{smallmatrix}0&1\\-1&a\end{smallmatrix}\right)$ of
the proof of Lemma~\ref{nm:width}, whose characteristic polynomial is
that of an element of trace $a$. For $a\ne\pm2$ the vector $(y,z)$ is
nonzero, since $y=z=0$ is excluded by Lemma~\ref{nm:nongenerating},
and its period under $M_a$ is $n(a)$, whether or not it is an
eigenvector. For $a=\pm2$ the matrix $M_a$ is $\pm$ a nontrivial
unipotent matrix; on its eigenline the Fricke relation gives $\tau=2$,
and off it the period is $n(a)$, that is, $p$ or $2p$. So every cycle
of $T$ in the row $x=a$ has length $n(a)$, and
\begin{equation}\label{nm:cusp-sum}
 c_\tau=\sum_{a\in\bF_p}\frac{N_a}{n(a)}
 \le 2p\sum_{a\in\bF_p}\frac1{n(a)}.
\end{equation}
For $a\ne\pm2$ the eigenvalues $\lambda^{\pm1}$ lie in $\bF_p^\times$
or in the group of norm-one elements of $\bF_{p^2}^\times$, cyclic of
orders $p-1$ and $p+1$, and each unordered pair
$\{\lambda,\lambda^{-1}\}$ of elements of order $m\ge3$ in one of these
groups arises from exactly one $a$. Since an order $m\ge3$ divides at
most one of $p\pm1$,
\begin{equation}\label{nm:H}
 H_p:=\sum_{a\in\bF_p}\frac1{n(a)}
 =\sum_{\substack{m\mid p-1\\ m\ge3}}\frac{\phi(m)}{2m}
 +\sum_{\substack{m\mid p+1\\ m\ge3}}\frac{\phi(m)}{2m}
 +\frac1p+\frac1{2p}
 =\frac{\Phi(p-1)+\Phi(p+1)}{2}-\frac32+\frac{3}{2p}.
\end{equation}
With $\Phi(n)\le D(n)\le2\sqrt n$, this gives the bounds on $c_\tau$
in (b).

\emph{Cuspidal $\PSL_2(\bF_p)$-covers.} The cusps of $M(\Gp)^\abs_\tau$
are the cycles of $T$ on the sign classes in $F(p)_\tau$, and their
lengths were determined in the proof of Lemma~\ref{nm:width}: for
$a\ne0$, the class of $(a,y,z)$ returns to itself after $r$ steps
exactly when $M_a^r(y,z)^{\mathsf t}=\pm(y,z)^{\mathsf t}$, so its
cycle has length $h(a)$. A class with first coordinate $0$ consists of
the four triples $(0,\pm y,\pm z)$, on which $T(0,y,z)=(0,z,-y)$ has
period $2$, or period $1$ when $z=\pm y$; the latter happens for at
most one class, as then $2y^2=\tau+2$. The $2N_a$ triples with first
coordinate $\pm a\ne0$ form $N_a/2$ classes and the $N_0$ triples with
first coordinate $0$ form $N_0/4$ classes, so the number $c'_\tau$ of
cusps of $M(\Gp)^\abs_\tau$ is
\begin{equation}\label{nm:cusp-sum-proj}
 c'_\tau=\frac14\sum_{a\in\bF_p}\frac{N_a}{h(a)}+\frac\delta2,
 \qquad\delta\in\{0,1\},
\end{equation}
using $h(0)=2$. For $a\ne\pm2$, the image of $\lambda$ in the quotient
of its cyclic group by $\{\pm1\}$ has order $h(a)$. These quotients
are cyclic of orders $(p-1)/2$ and $(p+1)/2$; an unordered pair of
elements of order $h\ge3$ in one of them arises from the two values
$\pm a$, and the element of order $2$, which lies in exactly one of
the two quotients, from $a=0$. Hence
\begin{equation}\label{nm:H-proj}
 H'_p:=\sum_{a\in\bF_p}\frac1{h(a)}
 =\sum_{\substack{h\mid (p-1)/2\\ h\ge3}}\frac{\phi(h)}{h}
 +\sum_{\substack{h\mid (p+1)/2\\ h\ge3}}\frac{\phi(h)}{h}
 +\frac12+\frac2p
 =\Phi\Bigl(\frac{p-1}2\Bigr)+\Phi\Bigl(\frac{p+1}2\Bigr)-2+\frac2p.
\end{equation}

\emph{Elliptic points.} A trace triple fixed by $S$ has $y=x$,
$z=x^2/2$ and $x^4-8x^2+4(\tau+2)=0$; one fixed by $R$ has $y=x$,
$z=x^2-x$ and $x^3-3x^2+\tau+2=0$. So $e_2\le4$ and $e_3\le3$ for
the special linear curve. A sign class is fixed by $S$ exactly when
$S(x,y,z)=(\epsilon_1x,\epsilon_2y,\epsilon_1\epsilon_2z)$ for some
signs $\epsilon_1,\epsilon_2$. Here $\epsilon_1\ne\epsilon_2$ forces
$x=y=0$, which does not occur in $F(p)_\tau$, and
$\epsilon_1=\epsilon_2$ gives $y=\epsilon_1x$, $z=\epsilon_1x^2/2$ and
the same quartic; so at most $8$ triples lie in fixed classes, and
$e_2\le2$ for the projective curve. Likewise
$R(x,y,z)=(\epsilon_1x,\epsilon_2y,\epsilon_1\epsilon_2z)$ gives
$y=\epsilon_2x$, $z=\epsilon_2x^2-\epsilon_1x$ and
$\epsilon_1\epsilon_2x^3-3x^2+\tau+2=0$, so at most $12$ triples and
$e_3\le3$.

\emph{Conclusion.} Let $X$ be a component of either curve. Its cusps
number at most $c_\tau$, because the finite surjective map
$M(\SL_2(\bF_p))^\abs_\tau\to M(\Gp)^\abs_\tau$ of
\S\ref{intro:markoff} sends cusps onto cusps. So
Lemma~\ref{nm:genus-formula} and the bounds on $e_2$ and $e_3$ give
$|g(\overline X)-d_X/12|\le c_\tau/2+3$, which is (b). For (c),
$M(\Gp)^\abs_\tau$ is connected as the image of
$M(\SL_2(\bF_p))^\abs_\tau$, and the degrees are
$|F(p)_\tau|=p^2+O(p)$ and $|F(p)_\tau|/4$. In \eqref{nm:cusp-sum},
the rows with $a\ne\pm2$ and $a^2\ne\tau+2$ have
$N_a=p-\chi(a^2-4)-\nu_a$, where $\sum_a\nu_a\le9^3+6$ counts the
points that are not trace triples of generating pairs, and
$n(a)\ge3$; the remaining at most four rows contribute $O(p)$. Hence
$c_\tau=pH_p+O(p)$, and in the same way
\eqref{nm:cusp-sum-proj} gives $c'_\tau=\frac p4H'_p+O(p)$.
Lemma~\ref{nm:genus-formula} with \eqref{nm:H} and \eqref{nm:H-proj}
gives the two genera.
\end{proof}

\begin{proof}[Proof of Corollary~\ref{cor_genus}]
For $p>N$ the curve $M(\SL_2(\bF_p))^\abs_\tau$ is nonempty and
connected (Corollary~\ref{cor_connected}), hence so is
$M(\Gp)^\abs_\tau$, and the degrees are those of
Proposition~\ref{nm:genus-fiber}(a). Since $\Phi(n)\le D(n)$ and
$D(n)=O_\epsilon(n^\epsilon)$ for every $\epsilon>0$
\cite{HardyWright}*{Theorem~315}, the second terms in
Proposition~\ref{nm:genus-fiber}(c) are $O_\epsilon(p^{1+\epsilon})$.
\end{proof}

\appendix
\section{The converse of the criterion of Bouw and Wewers}
\label{app:converse}

Throughout this appendix $p>5$, and the cover $f$ and its lifts $f_M$
are as in \S\ref{sec:arithmetic}. Theorem~\ref{nm:bw-theorem}(b) gives a
sufficient condition for bad reduction of a lift. We show that the
condition is also necessary once it is tested on the four lifts.

\begin{theorem}\label{app:corrected}
Suppose that the ramification indices of $f$ are prime to $p$, and let $v$ be
a place of $\Qbar$ above $p$. The following are equivalent.
\begin{enumerate}
\item The cover $f$ has bad reduction at $v$.
\item The four lifts of $f$ have bad reduction at $v$.
\item One of the four lifts of $f$ satisfies the hypothesis
\eqref{nm:bw-hypothesis} at $v$.
\item The cover $f$ is indefinite at $v$.
\end{enumerate}
\end{theorem}

Statements (a) and (b) are equivalent by Lemma~\ref{nm:lifts},
statements (c) and (d) by
Lemma~\ref{nm:dictionary}, and (c) implies (b) by
Lemmas~\ref{nm:dictionary} and~\ref{nm:lifts}. The new implication,
from (b) to (c), is proved in \S\ref{app:parity}.

\subsection{A counterexample to the published statement}

After a finite extension of $K$, every connected $\SL_2(\bF_p)$-cover of
$\bP^1_K$ branched at $0,1,\infty$, with $K$ as in
\S\ref{nm:bw-statement}, comes from a cover over $\Qbar$, and an
embedding of $\Qbar$ into an algebraic closure of $K$ determines a place
$v$ at which the roots of unity of Theorem~\ref{nm:bw-theorem} are the
reductions of those of \S\ref{nm:class-vectors}. Writing $-C$ for the
class of the negatives of the elements of a class $C$,
Theorem~\ref{app:corrected} therefore says that an $\SL_2(\bF_p)$-cover
whose class vector $(C_1,C_2,C_3)$ consists of noncentral classes of
elements of order prime to $p$ has bad reduction if and only if one of
the four class vectors $(\epsilon_1C_1,\epsilon_2C_2,\epsilon_3C_3)$ with
$\epsilon_i=\pm1$ and $\epsilon_1\epsilon_2\epsilon_3=1$ satisfies
\eqref{nm:bw-hypothesis}, and, with Remark~\ref{nm:bw-scope}(a) and
Lemma~\ref{nm:lifts}, that the number of covers with class vector
$(C_1,C_2,C_3)$ and bad reduction is two in this case and zero
otherwise. For $a_i:=a(C_i)$, \cite{BouwWewers}*{Theorem~6.6(b)} states
instead that this number is
\[
 \begin{cases}
 2&\text{if }a_1+a_2+a_3<p-1,\\
 2&\text{if }a_1+a_2+a_3=p-1\text{ and some $C_i$ is nonsplit},\\
 0&\text{otherwise.}
 \end{cases}
\]
The first two cases are Theorem~\ref{nm:bw-theorem}(b). The third case
fails for $p=13$. Let $v$ be a place of $\Qbar$ above $13$ at which
$\widetilde\zeta_{12}$ reduces to $2$. The matrices
\[
 M_1=\begin{pmatrix}6&0\\0&11\end{pmatrix},\qquad
 M_2=\begin{pmatrix}8&1\\6&9\end{pmatrix},\qquad
 M_3=\begin{pmatrix}8&7\\12&9\end{pmatrix}
\]
in $\SL_2(\bF_{13})$ have product $I$ and trace $4=2^5+2^{-5}$, and
$M_1$ and $M_2$ generate $\SL_2(\bF_{13})$ by a direct calculation.
So $M=(M_1,M_2,M_3)$ is a lift of a connected
three-point $G_{13}$-cover, and the class vector of $f_M$ at $v$ is
$(C_{\rm sp}(5),C_{\rm sp}(5),C_{\rm sp}(5))$. Here
$(a_1,a_2,a_3)=(2,2,2)$, so $f_M$ has bad reduction at $v$ by
Theorem~\ref{nm:bw-theorem}(b). The even sign change
$M'=(-M_1,-M_2,M_3)$ has the class vector
$(C_{\rm sp}(1),C_{\rm sp}(1),C_{\rm sp}(5))$ at $v$, for which
$(a_1,a_2,a_3)=(10,10,2)$ and $a_1+a_2+a_3=22>12$. The third case
predicts good reduction for $f_{M'}$, whereas $f_{M'}$ has bad
reduction at $v$ by Lemma~\ref{nm:lifts}.

The count in \cite{BouwWewers} is obtained in two steps: the structure
of the stable reduction, through the vanishing cycle formula, restricts
the ramification invariants of the primitive tails of a cover with bad
reduction, and with them its integers $a_i$, to the cases with
$a_1+a_2+a_3\le p-1$; then each remaining case is realized by gluing a
hypergeometric deformation datum, whose differential form comes from a
polynomial solution of a hypergeometric differential equation in
characteristic $p$, to tail covers, and lifting to characteristic zero.
The third case rests on the first step, namely on
\cite{BouwWewers}*{Proposition~6.5(i)}, which states that a primitive
tail cover whose canonical generators above the tame branch point lie in
a class $D$ has the ramification invariant $a(D)/(p-1)$. Its proof glues
the tail cover to a deformation datum by the lifting theorem
\cite{BouwWewers}*{Proposition~2.3}, whose condition (3) compares inertia
groups only, whereas the gluing also requires the canonical generator of
the tail cover at its wild point to be inverse to that of the deformation
datum (Lemma~\ref{app:lifting}). A tail cover and its twist by the
central quadratic character have the same ramification invariant, but
their canonical generators differ by $-I$; accordingly the invariant is
$a(D)/(p-1)$ or $a(-D)/(p-1)$, depending on the canonical generator at
the wild point (Lemma~\ref{app:pointed}).

\subsection{Canonical generators}
\label{app:conventions}

In the rest of the appendix $k$ is an algebraically closed field of
characteristic $p$, and $\zeta_{p-1}\in\bF_p$ and
$\zeta_{p+1}\in\bF_{p^2}$ are primitive roots of unity in $k$, as in
Theorem~\ref{nm:bw-theorem}. For a divisor $n$ of $p-1$ or of $p+1$ we
put $\zeta_n:=\zeta_{p-1}^{(p-1)/n}$ or
$\zeta_n:=\zeta_{p+1}^{(p+1)/n}$; the two definitions agree for $n=2$.

Let a finite group $G$ act on a smooth projective curve $W$ over $k$,
possibly disconnected, such that no element of $G$ other than the
identity acts trivially on a component of $W$, and let $y$ be a point
of $W$ with stabilizer $I_y$. The Sylow $p$-subgroup $P_y$ of $I_y$ is
normal, and the quotient $I_y/P_y$ is cyclic of order $n$ prime to $p$
and acts faithfully on the cotangent space of $W$ at $y$. We assume
throughout that $n$ divides $p-1$ or $p+1$, as is the case for
$G=\SL_2(\bF_p)$. The \emph{canonical generator} at $y$ is the element
$\gamma_y$ of $I_y/P_y$ that acts on this cotangent space through
$\zeta_n$; here a group element acts on functions by composition. Two
facts about canonical generators are used repeatedly.
\begin{enumerate}
\item[(N)] \emph{Wild points.} If $n$ divides $p-1$, then $\gamma_y$
acts through $\zeta_n$ on a uniformizer of the subfield of the completed
local field at $y$ fixed by $P_y$: for $g\in I_y$ of order $n$ mapping
to $\gamma_y$ and a uniformizer $\varrho$ at $y$ with
$g^*\varrho=\zeta_n\varrho$, the norm
$N(\varrho)=\prod_{\alpha\in P_y}\alpha^*\varrho$ is a uniformizer of
the fixed field, and $g^*N(\varrho)=\zeta_n^{|P_y|}N(\varrho)=\zeta_nN(\varrho)$
because $g$ normalizes $P_y$.
\item[(F)] \emph{Inseparable components.} Let a subgroup $H$ of $T$ act
on a smooth projective connected curve $\bar Y_0$ over $k$, and let
$\bar Y_0\to\bar Y_0'$ be the relative Frobenius. An element of $H$
fixing a point $y$ of $\bar Y_0$ acts on the cotangent spaces of
$\bar Y_0$ at $y$ and of $\bar Y_0'$ at the image of $y$ through the
same character, because a uniformizer of $\bar Y_0'$ there is the $p$th
power of one of $\bar Y_0$ at $y$, and $\lambda^p=\lambda$ for
$\lambda\in\bF_p^\times$.
\end{enumerate}

Following \cite{BouwWewers}*{Definition~2.2}, a \emph{$G$-tail cover}
is a $G$-cover $\phi\colon W\to\bP^1_k$, with $W$ possibly
disconnected, that is wildly ramified above $\infty$, with inertia
groups of order $pn$, $p\nmid n$, and tamely ramified above at most
one other point; it is \emph{primitive} if there is such a point,
which we take to be $0$, and \emph{new} otherwise. Its
\emph{ramification invariant} is $\sigma(\phi):=h/n$, where $h$ is the
conductor at a point $y$ above $\infty$, that is, the largest
integer such that the $h$th lower ramification group of
$I_{y}$ is not trivial. For deformation data we use the terminology of
\cite{BouwWewers}*{\S\S2--3}: special deformation data of a given type,
their critical points with the invariants $\sigma$, which are called
wild, primitive or new according to whether $\sigma=0$,
$0<\sigma\le1$ or $\sigma>1$, and hypergeometric deformation data with
their signature.

Let $B\subset\SL_2(\bF_p)$ be the subgroup of upper triangular
matrices, $P$ its subgroup of unipotent matrices, and $T$ the subgroup
of diagonal matrices. We identify $T$ with $\bF_p^\times$ by
$\diag(\beta,\beta^{-1})\mapsto\beta$, and we write $\beta$ for the
matrix as well; thus $-1$ is the matrix $-I$. The group $T$ acts on
$P$ by conjugation through the character $\beta\mapsto\beta^2$, the
group $P$ is the only subgroup of order $p$ of $B$, and $B$ is the
normalizer of $P$. In particular, if the stabilizer $I_y$ of a point of
an $\SL_2(\bF_p)$-cover contains $P$, then $I_y\subset B$ and
$\gamma_y\in I_y/P\subset B/P=T$.

\begin{lemma}\label{app:local}\leavevmode
\begin{enumerate}
\item\emph{Kummer covers.} Let $n$ be a divisor of $p-1$, let
$\nu\in k(x)^\times$, and let $W$ be the normalization of $\bP^1_x$ in
the algebra $k(x)[z]/(z^n-\nu)$, on which the group of $n$th roots of
unity acts by multiplication on $z$. At every point of $W$ above a point
$t$ of $\bP^1_x$ the canonical generator is $\zeta_n^{\,\ord_t\nu}$. In
particular the product of these canonical generators over all points $t$
of $\bP^1_x$, which is $1$ for all but finitely many $t$, equals $1$,
because $\sum_t\ord_t\nu=0$.
\item\emph{Wild points.} Let $y$ be a point of an
$\SL_2(\bF_p)$-cover of $\bP^1_k$, possibly disconnected, whose
stabilizer $I_y$ contains $P$. Let $n$ be the order of $I_y/P$, let $h$
be the conductor at $y$, and put $a:=(p-1)h/n$. If
$\gamma_y=\zeta_{p-1}^{\,c}$, then
\[
 2c\equiv-a\pmod{p-1}.
\]
In particular $a$ is an even integer, and $\gamma_y$ is one of
$\pm\zeta_{p-1}^{-a/2}$.
\item\emph{Central twist.} Let $\phi\colon W\to\bP^1_x$ be an
$\SL_2(\bF_p)$-cover, possibly disconnected, which is branched at most
above $0$ and $\infty$ and tamely ramified above $0$. There is an
$\SL_2(\bF_p)$-cover $\phi^-\colon W^-\to\bP^1_x$, the twist of $\phi$
by the quadratic character of $k(\sqrt x)/k(x)$ with values in
$\{\pm I\}$, with the following properties. It is branched at most
above $0$ and $\infty$; for every point $y$ of $W$ above $0$ or
$\infty$ there is a point $y^-$ of $W^-$ above the same point with
$P_{y^-}=P_y$ and $\gamma_{y^-}=-\gamma_y$; and if $\phi$ is a tail
cover, then so is $\phi^-$, with $\sigma(\phi^-)=\sigma(\phi)$.
\item\emph{The new tail cover.} The quotient of the curve
$s^{p+1}=x^p-x$ by the action of $\SL_2(\bF_p)$ of
\cite{BouwWewers}*{Lemma~4.3}, in which $\SL_2(\bF_p)$ acts on $x$ by
fractional linear transformations and $\beta\in T$ acts by
$(x,s)\mapsto(\beta^2x,\beta s)$, is a new tail cover. It has exactly
one point with stabilizer $B$, and the canonical generator at this
point is $\zeta_{p-1}^{-1}$.
\item\emph{Nodes.} Let $\cW$ be a semistable curve over a complete
discrete valuation ring $R$ with residue field $k$, with smooth
generic fiber, and let $g$ be an automorphism of $\cW$ over $R$ of
finite order that fixes a node $\eta$ of the special fiber and
preserves the two branches at $\eta$. Then $g$ acts on the cotangent
spaces of the two branches at $\eta$ through characters that are
inverse to each other. In particular, let $\cW$ be the stable model of a
connected $\SL_2(\bF_p)$-cover of $\bP^1_K$ with its ramification points
marked, let $\bar Y_0$ be a component of its special fiber with generic
inertia group $P$, decomposition group $G_0\subset B$ and $H:=G_0/P$, so
that $\bar Y_0\to\bar Y_0'\to\bar X_0$ is the relative Frobenius followed
by a Galois cover with group $H$ onto the component $\bar X_0$ of the
special fiber of $\cW/\SL_2(\bF_p)$ below $\bar Y_0$
\cite{WewersBad}*{Lemma~1.4(2)}, and let $\eta$ be a node at which a
component of a tail cover meets $\bar Y_0$. Then the canonical
generators of the tail cover at $\eta$ and of $\bar Y_0'\to\bar X_0$
above the image of $\eta$ are inverse to each other in $H$.
\end{enumerate}
\end{lemma}

\begin{proof}
(a) Let $x_t$ be a local parameter at $t$, and write $\nu=x_t^{\,b}q$
with $b:=\ord_t\nu$ and $q$ a rational function with neither a zero nor
a pole at $t$. Let $d:=\gcd(n,b)$, with $d:=n$ if $b=0$, and write
$n=dn'$ and $b=db'$. At a point $y$ of $W$ above $t$ the function
$u:=z^{n'}/x_t^{\,b'}$ is a unit, because its $d$th power is $q$. The
stabilizer of $y$ fixes the value $u(y)\ne0$ and acts on $u$ through
$\xi\mapsto\xi^{n'}$, so it lies in the group of $n'$th roots of unity;
on the other hand the ramification index $e$ at $y$ satisfies
$n\ord_yz=be$, so $n'$ divides $e$. Hence $e=n'$, the stabilizer is the
group of $n'$th roots of unity, and the orders of $x_t$ and $z$ at $y$
are $n'$ and $b'$. Choose integers $i_1$ and $i_2$ with $i_1b'+i_2n'=1$.
Then $z^{i_1}x_t^{\,i_2}$ is a uniformizer, on which an $n'$th root of
unity $\xi$ acts through $\xi^{i_1}$. So the canonical generator is the element $\xi$ with
$\xi^{i_1}=\zeta_{n'}=\zeta_n^{\,d}$, which is
$\xi=\zeta_n^{\,db'}=\zeta_n^{\,b}$, because $i_1b'\equiv1$ modulo $n'$.
The last assertion follows, because the divisor of $\nu$ has degree
zero.

(b) Let $g\in I_y$ be an element of order $n$ that maps to $\gamma_y$.
Since $n$ is prime to $p$, there is a uniformizer $\varrho$ at $y$ with
$g^*\varrho=\zeta_n\varrho$. For $\alpha\in P$ different from the
identity we have
$\alpha^*\varrho=\varrho+\lambda_\alpha\varrho^{h+1}+\cdots$ with
$\lambda_\alpha\ne0$, by the definition of the conductor, and
$\lambda_{\alpha^i}=i\lambda_\alpha$. The automorphism
$g\alpha g^{-1}$ acts on functions by $(g^{-1})^*\alpha^*g^*$, so
\[
 (g\alpha g^{-1})^*\varrho
 =\zeta_n\,(g^{-1})^*\bigl(\varrho+\lambda_\alpha\varrho^{h+1}+\cdots\bigr)
 =\varrho+\zeta_n^{-h}\lambda_\alpha\varrho^{h+1}+\cdots.
\]
On the other hand $g\alpha g^{-1}=\alpha^{\beta^2}$ if $g$ maps to
$\beta\in T$. Hence $\beta^2=\zeta_n^{-h}$, which for
$\beta=\zeta_{p-1}^{\,c}$ and $\zeta_n=\zeta_{p-1}^{(p-1)/n}$ is the
congruence $2c\equiv-a$ modulo $p-1$. The last assertion follows,
because $\zeta_{p-1}^{(p-1)/2}=-1$.

(c) Let $W_2$ be the normalization of the pullback of $\phi$ along
$r\mapsto x=r^2$, and let $\jmath$ be the automorphism of $W_2$ induced
by $r\mapsto-r$. The group $\SL_2(\bF_p)\times\langle\jmath\rangle$
acts on $W_2$ with quotient $\bP^1_x$. The element $(-I,\jmath)$ is
central of order two, and the quotient of the group by it is
$\SL_2(\bF_p)$ through $(g,\jmath^i)\mapsto(-I)^ig$. Put
$W^-:=W_2/\langle(-I,\jmath)\rangle$, with the induced action. Let
$y$ be a point of $W$ above $0$ or $\infty$, with local coordinate $x$
or $x^{-1}$ on the line, which we denote by $x$ again, let $n$ be the
order of $I_y/P_y$, and let $y_2$ be a point of $W_2$ above $y$. The
fixed field of $P_y$ in the local field at $y$ is $k((z))$ with
$z^n=x$, and $\gamma_y$ acts on $z$ through $\zeta_n$ by (N). The tame part of
the local extension at $y_2$ over $k((x))$ is the compositum of
$k((z))$ and $k((x^{1/2}))$, a Kummer extension of degree
$n_2:=\operatorname{lcm}(n,2)$ with a uniformizer $\varrho$ such that
$\varrho^{n_2/n}$ and $\varrho^{n_2/2}$ are constant multiples of $z$
and $x^{1/2}$. The element acting on $\varrho$ through $\zeta_{n_2}$
acts on $z$ through $\zeta_n$ and on $x^{1/2}$ through $-1$, so the
canonical generator at $y_2$ is $(\gamma_y,\jmath)$. Let $y^-$ be the
image of $y_2$. If $(-I,\jmath)$ does not fix $y_2$, then $\varrho$ is
a uniformizer of the tame part at $y^-$, and the canonical generator
there is the image $-\gamma_y$ of $(\gamma_y,\jmath)$. Otherwise
$\varrho^2$ is such a uniformizer, the inertia group has tame quotient
of order $n_2/2$, and the image of $(\gamma_y,\jmath)$ acts on
$\varrho^2$ through $\zeta_{n_2}^2=\zeta_{n_2/2}$; again
$\gamma_{y^-}=-\gamma_y$. The Sylow $p$-subgroups are unchanged because
$p$ is odd. The ramification invariant of a tail cover is the jump of
the upper ramification filtration at a wild point, which is compatible
with quotients \cite{SerreLocalFields}*{IV, \S3}; the local extensions
at $y$ and at $y^-$ are quotients of the one at $y_2$ by subgroups of
order prime to $p$, so the three have the same wild upper jump.

(d) There is one point above $x=\infty$, and its stabilizer is $B$, the
stabilizer of $\infty$ in $\bP^1(\bF_p)$. The stabilizers of the other
wild points are the other conjugates of $B$, because $B$ is its own
normalizer. The orders of $x$ and $s$ at this point are $-(p+1)$ and
$-p$, so $s/x$ is a uniformizer, on which $\beta\in T$ acts through
$\beta\cdot\beta^{-2}=\beta^{-1}$. The canonical generator is therefore
$\zeta_{p-1}^{-1}$.

(e) The completed local ring of $\cW$ at $\eta$ is
$R[[s_1,s_2]]/(s_1s_2-\varpi)$ with $\varpi\in R$ neither zero nor a
unit, and the two branches are $s_2=0$ and $s_1=0$ in the special
fiber. The divisor of $s_1$ is a multiple of the branch $s_1=0$, which
$g$ preserves. So $g^*s_1=u_1s_1$ for a unit $u_1$ of this normal ring,
and likewise $g^*s_2=u_2s_2$. Since $s_1s_2=\varpi$ is fixed by $g$, we
have $u_1u_2=1$. The residues of $u_1$ and $u_2$ are the two
characters.

For the second assertion let $g\in I_\eta$ have order $n:=|I_\eta/P|$
and map to the canonical generator of the tail cover at $\eta$. The two
branches at $\eta$ lie above different components of the special fiber
of $\cW/\SL_2(\bF_p)$, on which the group acts trivially, so $g$
preserves them, and $P\subset I_\eta\subset G_0$. By the first assertion
$g^{-1}$ acts on the cotangent space of $\bar Y_0$ at $\eta$ through
$\zeta_n$, hence so does its image in $H$ on that of $\bar Y_0'$ at the
image of $\eta$, by (F); this image is the canonical generator there.
\end{proof}

\subsection{Lifting and pointed primitive tail covers}

The lifting theorem \cite{WewersBad}*{Corollary~4.6}, in the form
\cite{BouwWewers}*{Proposition~2.3}, produces a three-point cover from
a special deformation datum and tail covers whose inertia groups satisfy
the condition \cite{BouwWewers}*{Proposition~2.3, (3)}. For
$\SL_2(\bF_p)$ this condition does not determine the canonical
generators at the attaching points, and we state the theorem with the
hypotheses that its proof uses. We also restrict the invariants to the
range of \cite{WewersBad}*{Definition~2.7}, which is where the lifting
theorem is proved; the deformation data used below have primitive
invariants in $(0,1)$ and new invariants $(p+1)/(p-1)$, so they lie in
that range.

\begin{lemma}[Lifting]\label{app:lifting}
Let $H\subset T$ be a subgroup, let $Z_0\to\bP^1_k$ be a connected
Galois cover with group $H$, and let $\omega$ be a logarithmic
differential on $Z_0$ such that $(Z_0,\omega)$ is a special deformation
datum for the group $H$ acting on $P$ through $\beta\mapsto\beta^2$,
whose invariants $\sigma$ at the critical points satisfy $\sigma<2$ and
$\sigma\ne1$, so that it is special also in the sense of
\cite{WewersBad}*{Definition~2.7}, and whose critical points with
$\sigma<1$ are $0$, $1$ and $\infty$. Put $G_0:=P\rtimes H\subset B$.
For every critical point $t_j$ with $\sigma_j>0$ let $\phi_j$ be an
$\SL_2(\bF_p)$-tail cover, primitive if $\sigma_j<1$ and new if
$\sigma_j>1$, whose ramification invariant is $\sigma_j$, let $y_j$ be a
wild point of $\phi_j$ whose stabilizer contains $P$, and let $G_j$ be
the decomposition group of the component of $\phi_j$ through $y_j$.
Suppose that
\begin{enumerate}
\item for every $j$ the canonical generator of $\phi_j$ at $y_j$ is
inverse to the canonical generator of $Z_0\to\bP^1_k$ above $t_j$, and
\item the group $\SL_2(\bF_p)$ is generated by $G_0$ and the groups
$G_j$.
\end{enumerate}
Then there is a connected $\SL_2(\bF_p)$-cover of $\bP^1_K$ branched at
$0,1,\infty$, for a field $K$ as in \S\ref{nm:bw-statement}, whose
stable reduction gives rise to the deformation datum $(Z_0,\omega)$ and
to the tail covers $\phi_j$, attached at the points $y_j$.
\end{lemma}

\begin{proof}
Choose points $\xi_j$ of $Z_0$ above the $t_j$; their stabilizers $H_j$
do not depend on the choice, since $H$ is cyclic. The datum
$(Z_0,\omega)$ with the $\xi_j$, the groups $G_j$ and the components of
the $\phi_j$ through the $y_j$, pointed at $y_j$, is a special
$\SL_2(\bF_p)$-deformation datum \cite{WewersBad}*{Definition~2.13}: its
condition (1) is hypothesis (b), and its condition (2), that the
stabilizer of $y_j$ be $P\rtimes H_j$ with the ramification type of
$t_j$, follows from hypothesis (a), since the stabilizer of $y_j$
contains $P$ and is generated by $P$ and the subgroup of $T$ generated by
its canonical generator. The cover is given by
\cite{WewersBad}*{Corollary~4.6}, proved for the special $G$-maps of
\cite{WewersBad}*{Definition~2.15}; this is where the hypothesis on the
invariants enters.

Hypothesis (a) enters as follows; the condition
\cite{BouwWewers}*{Proposition~2.3, (3)} compares inertia groups only.
The lift is constructed from a patching datum \cite{WewersBad}*{\S4.2.2}:
one isomorphism $\lambda$, the identity on $H$, from the group of an
auxiliary cover built from $(Z_0,\omega)$ onto $G_0$, and for every $j$
an equivariant isomorphism between the completions of $\phi_j$ at $y_j$
and of the auxiliary cover at its tail over $t_j$.
\cite{WewersBad}*{Proposition~4.4} asserts that these exist, by
Lemma~2.12 there; the normal form in the proof of that lemma shows when
they exist: locally at a wild point a tail cover with
tame inertia of order $m$ and conductor $h$ is $z^m=x$, $y^p-y=z^h$,
where a generator of $P$ acts by $y\mapsto y+1$ and the tame inertia
acts on $z$ through a character $\psi$ and on $y$ through the
conjugation character, and the changes of coordinates allowed there,
\cite{WewersBad}*{(11)}, multiply $z$ by a unit and translate $y$. So
for a fixed $\lambda$ two such local covers are equivariantly
isomorphic exactly when their characters $\psi$, that is, their
canonical generators, agree, and since the changes of coordinates are
made tail by tail, one $\lambda$ serves all $j$. The tail of the
auxiliary cover over $t_j$ belongs to a stable reduction whose central
component carries $(Z_0,\omega)$, so its canonical generator at the
attaching point is inverse to that of $Z_0\to\bP^1_k$ above $t_j$ by
Lemma~\ref{app:local}(e). Hence hypothesis (a) is exactly the existence
of the patching data. When the conjugation character is injective on
$H_j$, it determines $\psi$ together with $h$, and hypothesis (a)
follows from the equality of inertia groups; for $\SL_2(\bF_p)$ its
kernel $\{\pm I\}$ leaves the sign of the canonical generator open.
\end{proof}

\begin{lemma}\label{app:pointed}
Let $\phi\colon W\to\bP^1_k$ be a primitive $\SL_2(\bF_p)$-tail cover
whose ramification invariant satisfies $0<\sigma(\phi)<1$ and whose
canonical generators above $0$ lie in a noncentral class $D$. Let
$y_\infty$ be a point of $W$ above $\infty$ whose stabilizer contains
$P$, and put $a:=(p-1)\sigma(\phi)$. Then $a$ is an even integer, and
there is a unique sign $\epsilon\in\{\pm1\}$ with
\[
 \gamma_{y_\infty}=-\epsilon\,\zeta_{p-1}^{-a/2}.
\]
For this sign, $a(\epsilon D)=a$.
\end{lemma}

\begin{proof}
The first assertion is Lemma~\ref{app:local}(b). Suppose that
$\epsilon=1$, and put $N:=(p-1-a)/2$, a positive integer. By
\cite{BouwWewers}*{Proposition~3.2(ii)} there is a hypergeometric
deformation datum of signature $(0,0,a)$. As in the proof of
\cite{BouwWewers}*{Proposition~6.5}, we pull it back to the cyclic cover
$z^{p-1}=u(x)$ of $\bP^1_x$, on which $\beta\in T$ acts by
$z\mapsto\beta z$; here $u$ is a monic polynomial of degree $N$ with
simple roots different from $0$ and $1$. This cover is connected,
because the class of $u$ in $k(x)^\times/k(x)^{\times(p-1)}$ has order
$p-1$, as $u$ has a simple root. The result is a special deformation
datum for the group $T$ acting on $P$ through $\beta\mapsto\beta^2$. Its
critical points are $0$ and $1$, which are wild, the point $\infty$,
which is primitive with invariant $a/(p-1)$, and the roots of $u$, which
are new. By Lemma~\ref{app:local}(a) the canonical generator of the
cyclic cover is $\zeta_{p-1}^{-N}=-\zeta_{p-1}^{\,a/2}$ above $\infty$
and $\zeta_{p-1}$ above the roots of $u$.

We attach $\phi$ at $\infty$ through the point $y_\infty$, and a copy
of the new tail cover of Lemma~\ref{app:local}(d) at every root of $u$
through its point with stabilizer $B$. The canonical generators at the
attaching points are inverse to each other, by the hypothesis
$\gamma_{y_\infty}=-\zeta_{p-1}^{-a/2}$ and by Lemma~\ref{app:local}(d),
and $G_0=B$ and the group $\SL_2(\bF_p)$ of a new tail cover generate
$\SL_2(\bF_p)$. Lemma~\ref{app:lifting} provides a connected
$\SL_2(\bF_p)$-cover of $\bP^1_K$ branched at $0,1,\infty$ whose stable
reduction has the primitive tail cover $\phi$. Its ramification indices
at $0$ and $1$ are equal to $p$, because these points specialize to the
central component and the cyclic cover is unramified above them
\cite{WewersBad}*{Lemma~1.4(3),(4)}, and its third class
is $D$, because the canonical generators of a cover at a branch point
and those of its stable reduction at the specialization are the same
elements (footnote~\ref{app:specialization} in \S\ref{app:parity}). By
\cite{BouwWewers}*{Lemma~6.1}, whose proof is given there for split
classes and indicated for nonsplit classes, the primitive tail of such
a cover has the ramification invariant $a(D)/(p-1)$. Hence $a(D)=a$.

Suppose that $\epsilon=-1$. The twist $\phi^-$ of
Lemma~\ref{app:local}(c) is a primitive tail cover with the same
ramification invariant, its canonical generator at $y_\infty^-$ is
$-\gamma_{y_\infty}=-\zeta_{p-1}^{-a/2}$, and its canonical generators
above $0$ lie in $-D$. The first case gives $a(-D)=a$.
\end{proof}

\begin{remark}\label{app:unaffected}
Among the results of \cite{BouwWewers} used in this paper,
Proposition~6.3 (in \S\ref{sec:generic} for $\Gp$ and in
\S\ref{app:parity} for $\SL_2(\bF_p)$), Lemma~6.2, on which it rests,
and the first two cases of Theorem~6.6(b), which are
Theorem~\ref{nm:bw-theorem}(b), are proved by gluing with
\cite{BouwWewers}*{Proposition~2.3}; Proposition~6.5 is not used, and
Theorem~6.6(a), Propositions~3.2(ii) and~4.1, Lemma~4.3 and Lemma~6.1,
proved there for split classes and indicated for nonsplit ones, are used
as stated. We check the hypotheses of Lemma~\ref{app:lifting} in the
proofs that glue.

(a) \emph{Proposition~6.3.} We use it for connected new tail covers
with $1<\sigma<2$, in \S\ref{app:parity} and, for $\Gp$, in the proof
of Lemma~\ref{gen:tails}. The new tail covers of a stable reduction with
$1<\sigma<2$ are connected: a component with a proper decomposition
group $D$ has $p\mid|D|$, so $D$ lies in a Borel subgroup
(Lemma~\ref{nm:proper-subgroup}, applied to the image of $D$ in $\Gp$,
which is proper because a subgroup of $\SL_2(\bF_p)$ mapping onto $\Gp$
contains $-I$, the square of a lift of an element of order two); the
quotient of the component by $P$ is a connected tame cover of $\bP^1_k$
branched at one point at most, hence trivial by the Riemann--Hurwitz
formula; so $D=P$, and $\sigma$ is an integer. For $\Gp$ the conjugation
character of $T/\{\pm I\}$ on $P$ is injective, so the inertia groups
and conductors determine the canonical generators, as noted in the proof
of Lemma~\ref{app:lifting}, and the proof of the proposition stands with
Lemma~6.2 applied to $G=\Gp$, the generation condition being met by the
connected new tail cover with the full group. For $\SL_2(\bF_p)$, the quotient $\phi/\{\pm I\}$ of a
connected new tail cover $\phi$ with $1<\sigma(\phi)<2$ is a connected
new $\Gp$-tail cover with the same invariant, the upper ramification
jump being compatible with quotients \cite{SerreLocalFields}*{IV, \S3},
hence the quotient of the cover of Lemma~\ref{app:local}(d), with
$\sigma(\phi)=(p+1)/(p-1)$; and $\phi$ is determined by its quotient,
because two $\SL_2(\bF_p)$-covers of $\bA^1_k$ with the same quotient by
$\{\pm I\}$ differ by a twist with an \'etale double cover of $\bA^1_k$,
of which there is none for $p$ odd.

(b) \emph{Theorem~6.6(b), first two cases.} For a class vector
$(C_1,C_2,C_3)$ satisfying \eqref{nm:bw-hypothesis}, with $a_i:=a(C_i)$
and $N:=(p-1-a_1-a_2-a_3)/2$, the proof glues the hypergeometric datum
of signature $(a_1,a_2,a_3)$, pulled back to the cyclic cover
\[
 z^{p-1}=x^{b_1}(x-1)^{b_2}\prod_{j=1}^N(x-t_j'),\qquad
 b_i:=\frac{a_i}2+\frac{p-1}2,
\]
to copies of the cover of Lemma~\ref{app:local}(d) at the $t_j'$ and, at
$0$, $1$, $\infty$, to the primitive tails of covers with class vectors
$(pA,pA,C_i)$ or $(pA,pB,C_i)$, where $pA$, $pB$ are the classes of
order $p$; these tails have the invariants $a_i/(p-1)$ by Lemma~6.1 and
the classes $C_i$ at their tame points. By Lemma~\ref{app:local}(a) the
canonical generators of the cyclic cover are $\zeta_{p-1}^{\,b_i}$ at
$0$, $1$, $\infty$, the last since $b_1+b_2+b_3+N\equiv0$ modulo $p-1$,
and $\zeta_{p-1}$ at the $t_j'$. A tail of such a cover, with class $C$
and canonical generator $\gamma$ at its wild point, has
$\gamma=\zeta_{p-1}^{-b}$ for $b:=a(C)/2+(p-1)/2$: in the stable
reduction the cover $\bar Y_0'\to\bar X_0$ of Lemma~\ref{app:local}(e)
is unramified above the two wild points \cite{WewersBad}*{Lemma~1.4(4)},
its canonical generators are $\gamma^{-1}$ above the primitive point
(Lemma~\ref{app:local}(e)) and $\zeta_{p-1}$ above each of the $N'$ new
points ((a) and Lemma~\ref{app:local}(d)), where $a(C)+2N'=p-1$ by the
vanishing cycle formula \cite{WewersBad}*{Corollary~1.11}, and the
product rule of Lemma~\ref{app:local}(a) gives
$\gamma=\zeta_{p-1}^{\,N'}=-\zeta_{p-1}^{-a(C)/2}$, the case $\epsilon=1$
of Lemma~\ref{app:pointed}. Thus hypothesis (a) of
Lemma~\ref{app:lifting} holds. Hypothesis (b) holds through a new tail
cover if $N\ge1$, and for $N=0$ through the tail cover of a nonsplit
class $C_i$, present by the second case of \eqref{nm:bw-hypothesis}:
the decomposition group of its component through the wild point
contains $P$ and an element of $C_i$, which lies in no Borel subgroup,
so its image in $\Gp$ is $\Gp$ by Lemma~\ref{nm:proper-subgroup}, and
it is $\SL_2(\bF_p)$ by the argument of (a). If the cyclic cover is
disconnected, the construction applies to a connected component with
its stabilizer $H\subset T$ in place of $T$, with the same canonical
generators.
\end{remark}

\subsection{The parity of the signs}
\label{app:parity}

\begin{proof}[Proof that (b) implies (c) in Theorem~\ref{app:corrected}]
Let $f_M$ be a lift of $f$ with bad reduction at $v$, let $K$ be as in
\S\ref{nm:bw-statement}, chosen inside the completion of $\Qbar$ at
$v$, and let $(C_1,C_2,C_3)$ be the class vector of $f_M$ at $v$, for
the roots of unity of \S\ref{nm:class-vectors} and their reductions at
$v$. By Lemma~\ref{bg:ascent}(c), $f_M\otimes K$ has bad reduction,
that is, no good model over any finite extension of $K$, which is the
hypothesis of the structure theorem \cite{WewersBad}*{Theorem~2}. We use
the description of the stable reduction $\bar Y\to\bar X$ of
$f_M\otimes K$ in \cite{BouwWewers}*{\S2}, which is taken from
\cite{WewersBad}*{\S\S1--2}. The curve $\bar X$ consists of a central
component $\bar X_0=\bP^1_k$ and of tails attached to it. Let
$\bar Y_0$ be a component of $\bar Y$ above $\bar X_0$. The inertia
group of its generic point has order $p$; we choose $\bar Y_0$ such
that it is $P$. The decomposition group $G_0$ of $\bar Y_0$ normalizes
$P$, so $G_0\subset B$; put $H:=G_0/P\subset T$, and let
$\bar Y_0\to\bar Y_0'\to\bar X_0$ be the factorization of
Lemma~\ref{app:local}(e).

Because the ramification indices of $f_M$ are prime to $p$, the three
branch points specialize to three primitive tails
\cite{WewersBad}*{Lemma~1.4(3)}; in particular the reduction is not
exceptional in the sense of \cite{WewersBad}*{\S2.1}, which would leave
no room for a primitive tail. Let $\sigma_1,\sigma_2,\sigma_3$ be the
ramification invariants of the primitive tails and let $N$ be the
number of new tails. The vanishing cycle formula
\cite{WewersBad}*{Corollary~1.11}, together with the positivity of the
$\sigma_i$ and the inequality $\sigma_b>1$ for the invariant of a new
tail, gives
\[
 0<\sigma_i<1,\qquad 1<\sigma_b<2
\]
\cite{WewersBad}*{\S2.1.2, (6)}. By Remark~\ref{app:unaffected}(a) the
new tail covers are copies of the cover of Lemma~\ref{app:local}(d),
with $\sigma_b=(p+1)/(p-1)$, so the vanishing cycle formula reads
$\sigma_1+\sigma_2+\sigma_3+N\bigl(\frac{p+1}{p-1}-1\bigr)=1$, and the
numbers $a_i:=(p-1)\sigma_i$ satisfy
\begin{equation}\label{app:sum}
 a_1+a_2+a_3+2N=p-1.
\end{equation}

Let $\bar f_i$ be the tail cover over the $i$th primitive tail. Its
canonical generators above the specialization of the $i$th branch point
lie in $C_i$.\footnote{\label{app:specialization}The ramification points of
$f_M\otimes K$ above a branch point are marked points of the stable
marked model (Definition~\ref{bg:stable-model}), so they extend to
sections through its smooth locus with distinct specializations. The
inertia group of such a point fixes its section and acts on the conormal
module of the section, free of rank one over the valuation ring, through
a root of unity; on the generic fiber this is its action on the
cotangent space at the point, which defines the class through the roots
of unity of \S\ref{nm:class-vectors}, and on the special fiber its
action on the cotangent space of $\bar Y$ at the specialization, which
defines the canonical generator of the tail cover through their
reductions. As the specializations are distinct, the stabilizer of the
specialization is the inertia group of the point; so the two canonical
generators coincide, and the ramification indices agree
\cite{WewersBad}*{Lemma~1.4(3)}.} It has a wild point
$\eta_i$ that lies on $\bar Y_0$ and above the point $t_i$ of $\bar X_0$
at which the tail is attached. The stabilizer of $\eta_i$ is contained
in $G_0$ and contains $P$; let $\gamma_i\in T$ be the canonical
generator of $\bar f_i$ at $\eta_i$. By Lemma~\ref{app:local}(b) the
integer $a_i$ is even and $\gamma_i=-\epsilon_i\zeta_{p-1}^{-a_i/2}$ for
a sign $\epsilon_i\in\{\pm1\}$, and Lemma~\ref{app:pointed}, applied to
$\bar f_i$ with its tame branch point at $0$ and $\eta_i$ above
$\infty$, gives
\[
 a(\epsilon_iC_i)=a_i\qquad(i=1,2,3).
\]

By Lemma~\ref{app:local}(e) the canonical generator of
$\bar Y_0'\to\bar X_0$ above $t_i$ is $\gamma_i^{-1}$, and above the
point $t_j'$ at which the $j$th new tail is attached it is
$\zeta_{p-1}$, by Lemma~\ref{app:local}(d); the cover is unramified
elsewhere \cite{WewersBad}*{Lemma~1.4(2)}. If $N\ge1$, then
$\zeta_{p-1}\in H$ and $H=T$. The connected cover $\bar Y_0'\to\bar X_0$
with the cyclic group $H$ of order $n$ is a Kummer cover $z^n=\nu(x)$
for a suitable identification of $H$ with the group of $n$th roots of
unity, so the product rule of Lemma~\ref{app:local}(a) gives
$\gamma_1^{-1}\gamma_2^{-1}\gamma_3^{-1}\zeta_{p-1}^{\,N}=1$. Hence, by
\eqref{app:sum},
\[
 \epsilon_1\epsilon_2\epsilon_3
 =-\zeta_{p-1}^{\,(a_1+a_2+a_3)/2+N}
 =-\zeta_{p-1}^{\,(p-1)/2}=1.
\]

Let $M'$ be the triple of the matrices $\epsilon_iM_i$. It is a lift
of $f$, the class vector of $f_{M'}$ at $v$ consists of the classes
$\epsilon_iC_i$, and its integers are $a_1,a_2,a_3$. By
\eqref{app:sum} they satisfy $a_1+a_2+a_3\le p-1$, with equality only
if $N=0$. Suppose that $N=0$ and that the three classes
$\epsilon_iC_i$ are split, with labels $l_i$. Then
$\sum_i(p-1-2l_i)=p-1$, that is, $l_1+l_2+l_3=p-1$, and the eigenvalues
$\delta_i:=\zeta_{p-1}^{\,l_i}$ of the matrices of $M'$ satisfy
$\delta_1\delta_2\delta_3=1$. For traces
$x_i=\delta_i+\delta_i^{-1}$ with $\delta_1\delta_2\delta_3=1$ the
expression \eqref{nm:fricke-delta} vanishes, which contradicts
$\Delta\ne0$. So $M'$ satisfies \eqref{nm:bw-hypothesis} at $v$.
\end{proof}

\end{document}